\documentclass[3p,times]{elsarticle}

\usepackage[hang,flushmargin]{footmisc}
\usepackage{array}
\usepackage{longtable}
\usepackage{csquotes}
\usepackage{changepage}
\usepackage{amsfonts}
\usepackage{amsmath}
\usepackage{mathrsfs}
\usepackage{amssymb}
\usepackage{mathrsfs}
\usepackage{amscd}
\usepackage{mathtools}
\usepackage{here}
\usepackage{tikz}
\usetikzlibrary{decorations.markings}
\usetikzlibrary{arrows.meta,calc}
\usetikzlibrary{calc}

\usepackage{tikz-cd}
\usepackage[all]{xy}
\usepackage{xcolor}
\usepackage{bbm}
\usepackage{colortbl}
\usepackage{selinput}
\usepackage{lipsum}
\usepackage{lineno}
\usepackage{hyperref}
\hypersetup{colorlinks=true,
	linkcolor=blue,
	filecolor=magenta,
	urlcolor=cyan,
	pdftitle={Overleaf Example},
	pdfpagemode=FullScreen,
}
\makeatother
\newenvironment{proof}{\noindent{\bf Proof.}}
{\noindent \ \hfill$\Box$\par}
\newtheorem{theorem}{Theorem}[subsection]
\newtheorem{definition}[theorem]{Definition}
\newtheorem{pro}[theorem]{Proposition}
\newtheorem{lem}[theorem]{Lemma}

\newtheorem{cor}[theorem]{Corollary}

\newtheorem{thm0}{Theorem}[section]

\newtheorem{pro0}[thm0]{Proposition}
\newtheorem{lem0}[thm0]{Lemma}

\newtheorem{reminner}{Remark}[section]
\newtheorem{rem0inner}[reminner]{Remark}

\newenvironment{rem}
  {\begin{reminner}\upshape}
  {\hfill\(\triangleleft\)\end{reminner}}

\newenvironment{rem0}
  {\begin{rem0inner}\upshape}
  {\hfill\(\triangleleft\)\end{rem0inner}}

\newcommand{\hc}{\circ}
\newcommand{\n}{\{ }
\newcommand{\nn}{\} }
\newcommand{\ca}{\eta}
\newcommand{\Fun}{\mathrm{Fun}}
\newcommand{\hCW}{\mathbf{hCW}_{(p)}^{0}}

\newcommand{\Vect}{\mathbf{Vect}_{\mathbbm{k}} ^{0}}

\newcommand{\Coalg}{\mathbf{Coalg}_{\mathbbm{k}}}

\newcommand{\Hop}{\mathbf{Hop}_{\mathbbm{k}}}

\newcommand{\Vecpt}{\mathbf{Vect}_{\z/p}^{0}}

\newcommand{\cg}{\sigma}
\newcommand{\lt}{\varepsilon}
\newcommand{\ch}{ \bar{\nu}  }

\newcommand{\ck}{\zeta}

\newcommand{\fe}{F^{(2)}}
\newcommand{\fs}{F^{(3)}}

\newcommand{\cq}{\bar{\mu}}

\newcommand{\kj}{\phi}

\newcommand{\w}{\omega}

\newcommand{\wq}{\infty}
\newcommand{\af}{\alpha}
\newcommand{\sk}{\mathrm{sk}}
\newcommand{\Ker}{\mathrm{Ker}}
\newcommand{\cok}{\mathrm{Coker}}

\newcommand{\m}{\;\mathrm{mod}\,}
\newcommand{\pa}{\partial}

\newcommand{\x}{\langle}
\newcommand{\xx}{\rangle}

\newcommand{\e}{\iota}

\newcommand{\dyd}{\supseteq}
\newcommand{\xyd}{\subseteq}

\newcommand{\ty}{ \equiv}

\newcommand{\lo}{\Omega}
\newcommand{\tg}{\cong}

\newcommand{\jia}{\oplus}
\newcommand{\z}{\mathbb{Z}}

\newcommand{\cpt}{\mathbb{C}P^{2}}

\newcommand{\s}{\Sigma}

\newcolumntype{Y}{>{\centering\arraybackslash}X}

\usepackage{booktabs}

\begin{document}
	\begin{frontmatter}
		
		\date{}
		\title{A \enquote{Periodicity} Phenomenon for Attaching Maps of Suspended Two-Cell Complexes}
		
		\author[1]{Juxin Yang}
		
		\author[2]{Fengchun Lei}
	
		\author[3]{Jie Wu}

		\address[1]{School of Mathematical Sciences, Dalian University of Technology, Dalian 116024,
			China, juxinyanghomotopy@gmail.com }
		
		\address[2]{School of Mathematical Sciences, Dalian University of Technology, Dalian 116024,
			China; Beijing Key Laboratory of Topological Statistics and Applications for Complex Systems, Beijing Institute of Mathematical Sciences and Applications, Beijing 101408, China.  leifengchun@bimsa.cn}
		
		\address[3]{School of Mathematical Sciences, Hebei Normal University, Shijiazhuang 050024, China; Beijing Key Laboratory of Topological Statistics and Applications for Complex Systems, Beijing Institute of Mathematical Sciences and Applications, Beijing 101408, China. wujie@bimsa.cn}

		\begin{abstract}  In this paper, we determine the skeleton of $F$ consisting of three cells, where $F$ is the homotopy fiber of the canonical pinch map from a suspension of a simply connected two-cell complex onto a sphere. The main result is stated $p$-locally:  for  $p=2$, and for   $p\geq5$ under an additional assumption. This provides a  tool for computing the unstable homotopy groups of suspended 2-cell complexes. The proof is based on Selick-Wu's $\mathrm{A}^{\mathrm{min}}$-theory and the machinery of the Eilenberg-Moore spectral sequence. As an application,  we compute the 2-primary component of   $\pi_{18} (\Sigma^{3}\mathbb{C}P^{2})$, a homotopy group outside the metastable range.
		\end{abstract}

	\end{frontmatter}

	\tableofcontents
	
	\section{Introduction}

	In this paper, we aim to provide effective tools for computing the unstable homotopy groups $\pi_{*} (\s (S^{n}\cup e^{n+k+1}))$ of suspended two-cell complexes.

	The development of homotopy theory has been strongly driven by the problem of computing homotopy groups. This problem has produced many important tools and ideas, such as the Toda bracket, the Adams spectral sequence, the EHP spectral sequence, and Goodwillie calculus. These have influenced not only topology, but also many other areas of mathematics, including homological algebra, algebraic geometry, K-theory,  motivic homotopy theory, and category theory.

	Two-cell complexes form one of the most elementary classes of finite CW complexes beyond spheres. The projective planes  $\mathbb{R}P^{2}$, $\mathbb{C}P^{2}$, $\mathbb{H}P^{2}$ and $\mathbb{O}P^{2}$ provide basic examples: each admits a CW decomposition with two cells.  Despite the simple cell structure, the homotopy theory of two-cell complexes is already highly nontrivial, since the attaching map can encode substantial unstable information.  Thus two-cell complexes provide a natural testing ground for methods designed to study how cell attachments affect homotopy groups.

	Suspended two-cell complexes, constituting a basic class of two-cell complexes, arise in many contexts. For example, Moore spaces $M(\z/r,n)$ with $r\geq2,\,n\geq2$ are suspended two-cell complexes.  Moreover,  certain skeletons of Lie groups and homogeneous spaces often have the form of suspended  two-cell complexes; see \cite[Chapter 19]{Hand}. Understanding the homotopy groups of such spaces is  a basic problem in unstable homotopy theory.
	
	However, even for a suspended two-cell complex, the homotopy groups are rarely determined by the homotopy groups of the  spheres related to its cell decomposition.  The attaching map affects the homotopy groups through the homotopy fiber.  Consequently, computing the homotopy groups of suspended two-cell complexes requires tools that can detect how the original attaching map propagates into higher unstable phenomena.

	Let $f\in\pi_{n+k} (S^{n})$ with $n\ge 2$ and $k\ge 0$. Let $F$ be the homotopy fiber of the pinch map
	\[
	\Sigma C_f= S^{n+1}\cup_{\Sigma f} e^{n+k+2} \longrightarrow S^{n+k+2},
	\]
	which pinches the bottom $S^{n+1}$ to the basepoint.\footnote{
		When $k=0$ and working $p$-locally, we additionally assume that
		the degree of the  map $f\colon S^{n}\longrightarrow S^{n}$ is  $p^{a}$ with $a\in\z_{+}$, so as to exclude the trivial case.} Then we have a  fiber sequence
	\begin{equation}\notag \label{jhaaa}
		F \longrightarrow \s C_{f} \xrightarrow[]{pinch} S^{n+k+2}.
	\end{equation} 
	\noindent This fiber sequence provides a basic framework for studying the homotopy groups  \(\pi_{*}(\Sigma C_f)\)
	of the suspended two-cell complex. Note that it was already understood in the 1950s or earlier (see James \cite{james} for example): information on the relative homotopy groups 
	$\pi_{*} (S^{n+1}\cup_{\Sigma f} e^{n+k+2},\, S^{n+1})$
	implies that
	$$
	\pi_{i} (S^{n+1}\cup_{\Sigma f} e^{n+k+2})
	\cong \pi_{i} (S^{n+1})\quad (\forall\, i<n+k+1)
	\quad \text{and}\quad
	\pi_{n+k+1} (S^{n+1}\cup_{\Sigma f} e^{n+k+2})
	\cong \pi_{n+k+1} (S^{n+1})/\langle \Sigma f\rangle .
	$$ \noindent For $i>n+k+1$, the homotopy groups $\pi_{*} (F)$  play essential roles in computing $\pi_{i} (\s C_{f})$.  In 1973,   B. Gray \cite{Gray} gave a combinatorial model for $F$, which is  effective for analyzing the homotopy-theoretic properties of 
	$F$, and which also completely determines the homology of $F$. For any CW complex $X$ and any $i\in\z_{+}$, there exists a canonical group isomorphism  $$\pi_{i} (X)\tg\pi_{i} (\mathrm{sk}_{i+1} (X)),$$ where  $\mathrm{sk}_{i+1} (X)$ denotes the skeleton of dimension $i+1$.   Hence, the skeletons of the homotopy fiber $F$ are central to the computation. Here and hereafter, we always regard $F$ as equipped with a CW structure whose cells correspond  to chosen generators of $H_{*}(F;\z)$. (For this  correspondence, see Hatcher's textbook \cite[Proposition~4C.1, p.~429]{xh}.)

	Denote by $F^{(m)}$  the skeleton  of $F$ consisting of $m$ cells. Gray explicitly determined the homotopy type of $F^{(2)}$ \cite{Gray}. However, once $m\geq3$, the homotopy type of $F^{(m)}$  becomes difficult to determine, whether integrally or $p$-locally. Thus the first genuinely new obstruction occurs in the attaching map of the third cell. 
	
	Note that for any prime $p$, the $p$-local skeleton $F^{(m)}_{(p)}$ of $F_{(p)}$ consisting of
	$m$  $p$-local cells can be identified with the $p$-localization
	$(F^{(m)})_{(p)}$; see Remark~\ref{Fdll}.

	In this paper, we determine $F^{(3)}$ after localization at the prime $p$, which serves as a tool for computing the $p$-primary components of $\pi_{*} (\s C_{f})$.

	\indent To state the main theorem,  localize spaces at  a prime $p$ such that   either
	\begin{itemize}
		\item[$\bullet$]  $p=2$,
		\item[$\bullet$] or  $p\geq5$ and $n\equiv k+1\equiv0\m{2}$. 
	\end{itemize}
	\noindent

	Our main theorem is as follows.
	
	\medskip
	
	Informally, the theorem identifies the attaching map of the third cell of $F$ in terms of the original attaching map $f$. More precisely, the 3-cell skeleton (the skeleton having 3 cells) is given by 
	\[
	F^{(3)}=\fe\cup_{\beta}e^{3n+2k+3};
	\]
	the attaching map $\beta$ is congruent, up to an explicit indeterminacy, to the naturally induced map
	$j_{3}j_{4}\circ \Sigma^{2n+k+2}f$. Thus the theorem should be viewed as saying that the third cell of $F$ is controlled by a suitable iterated suspension of the original attaching map $f$. For more details on the notation used in the following theorem, see Remark~\ref{Fdll} and the notation table in Subsection~\ref{hbg1}.

	By default, we assume that $f\in\pi_{n+k}(S^{n})$ with $n\geq 2$ and $k\geq 0$, and that
	$f=p^{a}\mathrm{id}$ for some $a\in\z_{+}$ when $k=0$.
	
	\begin{thm0}\label{hao1} (\textbf{Main Theorem}; see~Theorem~\ref{zdlzm}).  Under the above $p$-localization hypothesis, suppose that the cup product  $H^{n} (C_{f};\z/p)\otimes H^{n} (C_{f};\z/p)\stackrel{\cup}\longrightarrow H^{2n} (C_{f};\z/p)$ is  trivial \footnote{For example, $C_{f}=S^{n}\cup _{f} e^{n+k+1}$ is a suspension}, where $C_{f}=S^{n}\cup_{f}e^{n+k+1}$ with $n\geq2$ and $k\geq0$. Then, the homotopy fiber $F$  of  the pinch map  $ \s C_{f}=S^{n+1}\cup _{\s f} e^{n+k+2} \xrightarrow[]{pinch} S^{n+k+2}$ is given by 
		
		\[F=\fe\cup_{\beta}  e^{3n+2k+3}\cup e^{4n+3k+4}  \cup\cdots,\;\;\big(\beta\in\pi_{3n+2k+2} (\fe)\big).\]  
		\noindent
		For its attaching map $\beta$,  the following relation holds:
		$$\beta\ty j_{3}  j_{4}\hc \s^{2n+k+2}f\m \;I\hc\s^{2n+k+2}f,$$  \noindent for a $\z_{(p)}$-submodule  $I\xyd \pi_{3n+k+2} (\fe)$    whose  precise description is given in Theorem~\ref{zdlzm}. \end{thm0}

	\indent 
	We provide a convenient criterion for verifying that the hypotheses of the main theorem are satisfied. 
	\begin{rem0}\textbf{(A convenient criterion)} 
		After localization at $p=2$,    if   $f$ is a suspended map or $k+1\neq n$,
		then the hypotheses  of the main theorem are always satisfied.
		
		In concrete applications, the case $p=2$ is often the most significant. Moreover, at the prime $2$, our theorem in fact covers  all essential cases, except when the map $f\in \pi_{2n-1}(S^{n})$ with $n\in\n 2,4,8\nn$  has Hopf invariant one.
	\end{rem0}

	The attaching map of the third cell of   the homotopy fiber  is therefore determined, up to explicit indeterminacy, by an iterated suspension of the original attaching map $ f$. 
	In other words, the iterated suspensions of the original attaching map exhibit a \enquote{periodicity} phenomenon.  
	
	We provide an application of the main theorem. In this theorem, we assume  by default that the required homotopy groups of spheres are known, with  explicit generating sets and the composition relations among the generators.
	
	\begin{thm0}\label{dkh}(\textbf{Application}; see Theorem \ref{0504qqq}). Assume that the hypotheses of  Theorem \ref{hao1} hold.  Then, in principle, all homotopy groups
		$\pi_{r}(S^{n+1}\cup_{\s f} e^{n+k+2})$
		with $r\leq 4n+2k$ can be determined. In particular, this family includes the groups
		$\pi_{r}(S^{n+1}\cup_{\s f} e^{n+k+2})$
		in the non-metastable range, namely,
		$3n+2\leq r\leq 4n+2k$,
		and also those in the $J_{3}$-range, namely,
		$3n+2k+2\leq r\leq 4n+2k$.\end{thm0}

	We now comment on the range of $r$ in Theorem~\ref{dkh}.
	In special computations, the main theorem can still be used to compute
	$\pi_{r}(S^{n+1}\cup_{\s f} e^{n+k+2})$
	for $r\leq 4n+3k+2$.   We make the following remark about Theorem~\ref{dkh}.
	
	\begin{rem0} 
		Assume that the hypotheses of Theorem~\ref{hao1} hold. In concrete calculations, the range of computability in Theorem \ref{dkh} can often be extended to $\pi_{r}(S^{n+1}\cup_{\s f} e^{n+k+2}),
		\; r\leq 4n+3k+2,$ apart from possible exceptional cases. \end{rem0}

	We continue with several further comments on the main theorem.
	
	\begin{rem0}\upshape(\textbf{On the Indeterminacy}). In Theorem \ref{hao1}, the main term
		$j_{3}j_{4}\hc\Sigma^{2n+k+2}f$
		and the indeterminacy term  $I\hc \s^{2n+k+2}f$ come from rather different geometric sources.    The indeterminacy  term $I\hc \s^{2n+k+2}f$  has trivial intersection with the subgroup generated by the main term $j_{3}j_{4}\hc\Sigma^{2n+k+2}f$.  See Lemma \ref{jsidyizh}.  	
		When calculating homotopy groups, we may choose a different generating set for a homotopy group. The idea is similar to changing basis in a vector space. For example, let the vector space $V=\x u,v\xx\cong \mathbb{Q}^{2}$, and  suppose that $u'\in V$ and $u'\ty u \m v$.
		Then we  have 
		$V=\x u',v\xx\cong \mathbb{Q}^{2}$.
	\end{rem0}

	\begin{rem0}\upshape(\textbf{Roles of the hypotheses}). \begin{itemize}
			\item[\rm (1)]
			To prove the main theorem, we use Selick-Wu's $\mathrm{A}^{\min}$-theory, which analyzes the loop suspension functor $\lo\s$ and  connects  homotopy theory and  representation theory of symmetric groups. Selick-Wu's $\mathrm{A}^{\min}$-theory is used  in this paper to isolate a functorial factor of $\Omega\Sigma C_f$ whose homological information detects  the  attaching map $\beta$. See Proposition \ref{aminq} and Proposition \ref{amindl} for the $\mathrm{A}^{\min}$-theory.  
			The proof  relies on the assumption  $p\neq3$, which ensures that   $\mathrm{SQ}_{3}^{\mathrm{max}} (X)=\mathop{\mathrm{hocolim}}\limits_{\beta_{3}/3}\,\s X^{ \wedge 3}$ is well-defined; see~Lemma~\ref{q3max}.  It also relies on the assumption  $n\equiv k+1\equiv0\m{2}$ if $p\geq5$; this ensures that all homogeneous elements of $\widetilde{H}_{*}(C_{f};\z/p)$ have even degree.  Additionally, the proof   requires the assumption that  $H^{*} (C_{f};\z/p)$ is a trivial algebra, ensuring that 
			$H_{*} (\lo\s C_{f};\z/p)\tg T(\widetilde{H}_{*} (C_{f};\z/p))$ as primitively generated tensor Hopf algebras; see Proposition \ref{neihopf} and Lemma \ref{hjg}.
			
			\item[\rm (2)] In the case that $p=3$ or that $H^{*} (C_{f};\z/p)$ is a nontrivial algebra, we are unable to determine whether  a similar result holds  with current techniques.  \end{itemize}
	\end{rem0}
	\begin{rem0}(\textbf{The landscape around $\pi_{*} ( S^{n+1}\cup_{\s f} e^{n+k+2})$}). Before the present work, the only known nontrivial determinations of homotopy groups of $2$-cell complexes in the $J_{3}$ range were due to Mukai \cite{xdwp}, Wu \cite{wu2003}, and Zhu-Jin \cite{zzj}. ( ``The $J_{3}$ range''  is the range in which the third cell of $F$ affects the computation of $\pi_{*} (\s C_{f})$. All $\pi_{*} (\s C_{f})$ in the $J_{3}$ range are outside the metastable range.) These works all focus on Moore spaces of the form $S^{n+1}\cup_{2^{i}}e^{n+2}$, corresponding to the case $k=0$.

		Mukai's approach first determines the attaching map $\beta$ by using James' theorem on relative Whitehead products, and hence does not readily extend to general suspended $2$-cell complexes. Wu's computations avoid an explicit prior determination of $\beta$, but depend  on special decompositions of the loop spaces of mod $2$ Moore spaces. Zhu-Jin likewise begin by determining $\beta$. Their method combines Wu's computations, Wu's splitting of the loop spaces of mod $2$ Moore spaces, and properties of higher Whitehead products. Their approach appears to rely substantially on Wu's special ingredients.
		
		Under the hypotheses of the main theorem, our result provides a general tool for computing
		$\pi_{i}(S^{n+1}\cup_{\s f} e^{n+k+2})$, covering both the case $k=0$ and the cases $k\geq 1$. 
	\end{rem0}

	The following theorem is an application of the main theorem.    When writing  homotopy group isomorphisms with generating sets,  we list the generators so that they correspond, in order, to the displayed direct summands. For the meaning of the generator symbols, see Theorem~\ref{tlq} and its proof.
	\begin{thm0}\label{gp18} (\textbf{A computational template}; see Theorem~\ref{tlq}). The 2-primary component of $\pi_{18} (\s^{3}\cpt)$ is given as follows:\begin{align}
			\pi_{18} (\s^{3}\cpt\colon2) &= \x j_{0}j_{1}j_{2}j_{2.5}\nu_{5}\cg_{8}\nu_{15}  ,\;j_{0}j_{1}j_{2}j_{3}j_{4}\nu_{15},\; j_{0}j_{1}j_{2}\hc \mathrm{coext}_{\nu_{5}\ca_{8}^{2}} (2\cg_{10}),\;\mathrm{coext}_{\ca_{5}} (\ck_{6})\xx \notag  \\
			&\tg \z/2\jia\z/4\jia\z/16\jia\z/8. \notag 
	\end{align}\end{thm0}

	The next  theorem provides the homological input for the proof of the main theorem, arises as a byproduct of that proof, and may also be of independent interest.	The theorem is a generalization of a result of Cohen-Moore-Neisendorfer  \cite[Corollary~9.4, p.~150]{CMN1979}. Their  result is recovered as a special case of ours by setting $r-m=1,\;\mathrm{deg} (f)=p^{a}$ for some $a\in\z_{+}$; see Remark \ref{jsxs}. 
	
	Note that 
	Theorem \ref{dlyde}  is stated integrally; we do not take any localization here. Moreover, the hypotheses for the following theorem are independent of those of the main theorem.  For  a tensor algebra $T(V)$, we say that $T(V)$ is a Hopf algebra by declaring all elements in $V$ to be primitive.   
	
	\begin{thm0}\label{dlyde} (see Theorem \ref{dlydexj}). Let $p$ be an arbitrary prime, and take homology with coefficients in $\mathbb{Z}/p$. Let  $X=S^{m}\cup e^{r}$, where $2\leq m< r.$ Let $F_{q}$ be the homotopy fiber of the pinch map $q\colon X\rightarrow S^{r}.$ If   $H_{*} (\lo  X)\tg T(u,v)$ 
		as Hopf algebras in which $|u|=m-1,|v|=r-1$, \big(for example, $X=\s^{2}(S^{m-2}\cup e^{r-2})$\big),  then  the fiber sequence $\lo F_{q}\xrightarrow[]{\lo j} \lo X \xrightarrow[]{\lo q}\lo S^{r}$ induces  a short exact sequence of Hopf algebras
		\[H_{*} (\lo F_{q})\xrightarrow[]{(\lo j)_{*}} H_{*} (\lo X) \xrightarrow[]{(\lo q)_{*}}H_{*} (\lo S^{r}).\] 
		\noindent
		Moreover, $\lo j$ induces an isomorphism of Hopf algebras
		$$H_{*} (\lo F_{q})\tg T(\n\mathrm{ad}^{i} (v)(u)\;|\; i\geq0\nn)= T(u,\,[u,v],\,[[u,v],v],\,[[[u,v],v],v],\;\cdots).$$    
	\end{thm0}

	\; \\\large\textbf{Acknowledgments}\normalsize
	
	\;\\ \indent The first and  second authors are supported in part by the Natural Science Foundation of China (NSFC grant no.~12331003).  The first and  third authors are supported in part by the High-level Scientific Research Foundation of Hebei Province.  The second  and third authors are also supported in part by the start-up research fund from Beijing Institute of Mathematical Sciences and Applications.

	\section{Conventions}\label{secconv}
	
	\subsection{Spaces, maps, and homotopy conventions}
	
	Throughout this paper, all spaces are assumed to be CW complexes. Spaces and maps are pointed;  basepoints and constant maps are denoted by $*$. 
	Unless otherwise stated, no distinction will be made between “maps” and “homotopy classes of maps,” nor between “homotopy-commutative” and “commutative” when referring to diagrams of spaces.  We usually denote $\mathrm{id}_{S^{n}}$ by $\e_{n}$.
	
	Suppose that $X$ is a space homotopy equivalent to $Y$; then we employ the identification   $X=Y$     if      no confusion arises.
	
	\subsection{Localization and homology conventions}
	
	Let $p$ be  a prime.
	The  ring of
	$p$-local integers is denoted by  $\z_{(p)}$. The functor ${H}_{\ast} (-)$  denotes the mod $p$ unreduced homology and $\widetilde{H}_{\ast} (-)$  denotes the mod $p$ reduced homology. The same conventions apply to the mod $p$ unreduced or reduced 
	cohomology. As is customary in most of the literature, the pointed structure of a space is forgotten  when using homology or cohomology.  For convenience, in a   $\z_{(p)}$-module, we say that  $x= y$ up to a unit in $\z_{(p)}$ if  $x=\ell y$ where $\ell$ is some unit in the ring $\z_{(p)}$.
	\subsection{Algebraic conventions}
	We use $\mathrm{Span}_{\z/p}\n a_{1},a_{2},\cdots,a_{n}\nn$ to denote the graded $\z/p$-vector space with a basis $a_{1},a_{2},\cdots,a_{n}$. All algebras, coalgebras, Hopf algebras, and modules are graded unless otherwise
stated. The notation $T(V)$ stands for the  tensor algebra  generated by  $V$.
	Unless otherwise specified, by a Hopf algebra
	$T(V)$ we mean  the Hopf algebra obtained by declaring elements of 
	$V$  to be primitive. For homogeneous $x,y\in T(V)$, define $[x,y]=x y-(-1)^{|x||y|}yx\in T(V);$ set $$\mathrm{ad}^{0} (y)(x)=x,\;\mathrm{ad}^{1} (y)(x)=[x,y] \;\;\text{and}\;\; \mathrm{ad}^{i} (y)(x)=[\mathrm{ad}^{i-1} (y)(x),y].$$ \noindent If no confusion arises, we denote an element $a\otimes b$ in a tensor product or a tensor algebra by $ab$ for short.
	
	\subsection{$p$-local CW complexes}

	When working $p$-locally, by a $p$-local cell  of $p$-local dimension $n$, namely, $e^{n}_{(p)}$,  we mean the corresponding   cell $C(S^{n-1}_{(p)})\backslash S^{n-1}_{(p)}$, and the dimension of a $p$-local CW complex $X$ means the highest $p$-local dimension of its  $p$-local cells, denoted by $\dim (X)$.  If no confusion arises,
	we  omit the adjective \enquote{\(p\)-local} and simply speak of cells and dimensions; the subscripts \enquote{\((p)\)} in the symbols for spaces \(X_{(p)}\), maps \(f_{(p)}\) and cells \(e^n_{(p)}\) are also  omitted.

	Suppose that $X$ is a $p$-local path-connected CW complex.  Then $X$ is said to be of finite type if $X$  has finitely many $p$-local cells $e^{i}$ for each $i$.
	It is said to be finite if  it has only finitely many $p$-local cells.
	
	\subsection{Default minimal CW structures}
	Let $X$ be a simply-connected CW complex of finite type. Unless there is a risk of confusion,  the CW structure on $X$ is understood to be minimal, in the sense that it has the minimum number of cells compatible with its homology: one $n$-cell for each $\mathbb{Z}$-summand of $H_n(X;\z)$, and a pair of cells in dimensions $n$ and $n+1$ for each $\mathbb{Z}/k$-summand of $H_n(X;\z)$.  A similar convention will be used for the  $p$-local analogue.

	\subsection{Products and diagonals}
	
	For a space $X$, $X^{\wedge n}$ denotes the $n$-fold smash product of $X$ with itself and $X^{\wedge 0}=S^{0}$, and $X^{\times n}$ denotes the $n$-fold Cartesian product of $X$ with itself. Denote by $\Delta_{X}\colon X\to X\times X$ the diagonal map; if the space is clear from the context, we write $\Delta$ instead.
	
	\subsection{Suspension notation in homology and graded modules}
	
	We denote the homology suspension isomorphism
	$
	\widetilde{H}_{i} (-)\rightarrow \widetilde{H}_{i+1} (\Sigma -)
	$
	by $\cg$.  The homology suspension homomorphism  $\widetilde{H}_{i} (\lo- )\rightarrow\widetilde{H}_{i+1} (- )$ is denoted by $\cg'$.

	Accordingly, for a graded module $V$, we write $\cg\colon V\rightarrow\cg(V)$ for the isomorphism shifting degrees up by $1$.  Similarly, we write
	$\cg^{-1}\colon V\rightarrow\cg^{-1} (V)$ for the isomorphism shifting degrees down by $1$. We also call the homomorphism $\cg$ above  the suspension, and  call $\cg^{-1}$ the  desuspension.
	\subsection{Default coefficient
		field and  ground field}
	
	Unless otherwise stated, in our setting of localization at $p$, the coefficient
	field and the ground field shall be understood to be $\z/p$ without further explanation. 
	
	\section{Preliminary  homology theory}\label{secptd}
	In this section, spaces are understood in the ordinary sense and are not assumed to be localized at any prime $p$.

	\subsection{Some results for the Hopf algebra structure}

	When applying the Bott-Samelson theorem \cite[Corollary 4.1.5, p.~111]{nei}, one should keep track of the coproduct structure carefully.   We state this theorem explicitly here because it will be used later in the analysis of the homology of the looped homotopy fiber and in the application of Selick-Wu's $\mathrm{A}^{\min}$-theory.

	\begin{pro}\label{neihopf}  (\textbf{the Bott-Samelson theorem}) Let $X$ be a path-connected space and take homology with
		coefficients in a field $\mathbbm{k}$.  Then the following
		properties hold. (Recall that by a Hopf algebra $T(V)$ we mean the tensor Hopf algebra  primitively generated by $V$.) 
		\begin{itemize}
			\item [\rm (1)]  $H_{*} (\lo\s X)\tg T(\widetilde{H}_{*} (X))$
			as algebras. (We shall refer to this canonical homomorphism  as the Bott-Samelson homomorphism.)
			\item [\rm (2)]  Let $(T(\widetilde{H}_{*} (X)))'$ be the tensor algebra   $T(\widetilde{H}_{*} (X)).$
			Equip $(T(\widetilde{H}_{*} (X)))'$  with
			a  Hopf algebra structure. Its coproduct  is obtained by extending the homomorphism
			$(\Delta_{X}) _{*}\colon H_{*} (X)\rightarrow H_{*} (X)\otimes H_{*} (X)$, the  coalgebra structure map of $H_{*} (X)$. Then, $H_{*} (\lo\s X)\tg (T(\widetilde{H}_{*} (X)))'$ 	as Hopf algebras.\hfill\boxed{}
		\end{itemize}  
	\end{pro}

	Let $f\in\pi_{n+k} (S^{n})$ with $n\ge 2$ and $k\ge 0$. Let $F$ be the homotopy fiber of the pinch map
	\[
	\Sigma C_f=S^{n+1}\cup_{\Sigma f} e^{n+k+2} \longrightarrow S^{n+k+2},
	\]
	which collapses $S^{n+1}$ to the basepoint.
	When $k=0$ and working $p$-locally, we additionally assume that
	the degree of the  map $S^{n}\stackrel{f}\longrightarrow S^{n}$ is $p^{a}$  for some $a\in\z_{+}$. \textbf{Hereafter,  the meanings of the symbols $f,n,k$ and $F$ are fixed.}

	The following lemma is well-known.
	\begin{lem}\label{hjg} Localize spaces at $p$ and
		take homology and cohomology with coefficients in $\z/p$. Recall that the CW complex $C_{f}=S^{n}\cup e^{n+k+1}$  is  finite. The following assertions hold.

		\begin{itemize}
			\item [\rm (1)] All elements in the coalgebra $\widetilde{H}_{*} (C_{f})$ are primitive if and only if all cup products  in the algebra $\widetilde{H}^{*} (C_{f})$  are trivial.
			\item [\rm (2)] For the space $C_{f}=S^{n}\cup e^{n+k+1}$,  if  \vspace{-0.8\baselineskip}\\\\\indent\quad $\bullet$ $C_{f}$ is a suspension, \\ \indent\quad $\bullet$   or $k+1\neq n$,\\ \indent\quad $\bullet$ or $k+1=n$ but the Hopf invariant of $f$ is divisible  by $p$, \vspace{-0.8\baselineskip}\\\\then $H^{*} (C_{f})$ is a trivial algebra and hence   $H_{*} (\lo \s C_{f})\tg T(\widetilde{H}_{*} (C_{f}))$ as Hopf algebras. \hfill\boxed{}
	\end{itemize}\end{lem}

	\subsection{The action of the Samelson product at the homology level}

	The following lemma serves to examine the homology behavior  of maps under Selick-Wu's $\mathrm{A}^{\min}$-theory.
	
	In the case that  $A_{1}, A_{2}$ are spheres,  assertion (3) of
	the following lemma is an exercise in Neisendorfer's book \cite[p.~114]{nei}.
	\begin{lem}\label{samelson}
		Let $A_{1},A_{2}$ and $X$ be path-connected spaces. Let $f_{i}\colon A_{i}\rightarrow\lo\s X$ be maps and $E\colon X\rightarrow\lo\s X$ be the canonical inclusion. Write $I=\mathrm{id}_{\lo\s X}$. Take homology  with coefficients in a field $\mathbbm{k}.$
		The following statements hold for the Samelson products $$\x I,I\xx\colon \lo\s X\wedge \lo\s X\rightarrow \lo\s X,\;\;\x E,E\xx  \colon  X\wedge  X\rightarrow \lo\s X \;\;\text{and}\;\; \x f_{1},f_{2} \xx\colon A_{1}\wedge A_{2}\rightarrow \lo\s X. $$ 
		\begin{itemize}
			\item[\rm (1)] If $a,b\in\widetilde{H}_{*} (\lo\s X)$ are primitive, then  the homomorphism $\x I,I\xx_{*}\colon \widetilde{H}_{*} (\lo\s X)\otimes\widetilde{H}_{*} (\lo\s X) \rightarrow \widetilde{H}_{*} (\lo\s X) $  sends $a\otimes b$ to $a b-(-1)^{|a||b|}b a=[a,b].$
			\item[\rm (2)] If $a,b\in\widetilde{H}_{*} ( X)$ are primitive, then  the homomorphism $\x E,E\xx_{*}\colon \widetilde{H}_{*} ( X)\otimes\widetilde{H}_{*} ( X) \rightarrow \widetilde{H}_{*} (\lo\s X) $  sends $a\otimes b$ to $[a,b].$
			
			\item[\rm (3)]   If $x_{1}\in\widetilde{H}_{*} (A_{1})$ and $x_{2}\in\widetilde{H}_{*} (A_{2})$ are primitive, then 
			the homomorphism $\x f_{1},f_{2}\xx_{*}\colon \widetilde{H}_{*} ( A_{1})\otimes\widetilde{H}_{*} ( A_{2}) \rightarrow \widetilde{H}_{*} (\lo\s X) $  sends $x_{1}\otimes x_{2}$ to $[f_{1*} (x_{1}),f_{2*} (x_{2})].$
		\end{itemize}
		
	\end{lem}
	
	\begin{proof} We need only to prove assertion (1). Assertions (2) and (3) follow from assertion (1) and the property   $$\x f_{1},f_{2}\xx=\x I,I\xx\hc (f_{1}\wedge f_{2}).$$
		\noindent For a topological group $G$, denote by $(-1)$ the inversion map $G\rightarrow G$, $x\mapsto x^{-1}.$ Denote by $\Delta$ the diagonal map of a space. Now, replace $\lo\s X$ by a homotopy equivalent topological group and write $G=\lo\s X$. The following commutator map $\w$ factors through the canonical projection $G\times G\rightarrow G\wedge G$, giving $\x I,I\xx\colon G\wedge G\rightarrow G$: 
		\[
		\begin{tikzcd}[column sep=1.7cm]
			\w\colon G \times G 
			\arrow[r, "\Delta \times \Delta"] 
			& G^{\times 2} \times G^{\times 2} 
			\arrow[r, "(23)"] 
			& G^{\times 2} \times G^{\times 2} 
			\arrow[r, "{(-1,\,-1) \times(I,\,I)}\;"] 
			& G^{\times 2} \times G^{\times 2} 
			\arrow[r, "\bar{\mu}"] 
			& G.
		\end{tikzcd}
		\]

		\noindent Here, the map $(23)$ is given by $(x_{1},x_{2},x_{3},x_{4})\mapsto (x_{1},x_{3},x_{2},x_{4})$ and $\cq$ is the 4-fold product on $G.$ Suppose that the elements $a,b\in\widetilde{H}_{*} (G)$ are primitive. We examine $\w_{*} (a\otimes b)$. Clearly, $\Delta_{*} (a)=a\otimes1+1\otimes a$ and $\Delta_{*} (b)=b\otimes1+1\otimes b$.
		
		So, $$(\Delta\times\Delta)_{*} (a\otimes b)=(a\otimes1+1\otimes a)\otimes(b\otimes1+1\otimes b)=a\otimes 1\otimes b\otimes 1+a\otimes 1\otimes 1\otimes b+1\otimes a\otimes b\otimes 1+1\otimes a\otimes 1\otimes b.$$ \noindent
		Then, $(23)_{*} (\Delta\times\Delta)_{*} (a\otimes b)$ is  equal to $$a\otimes b\otimes 1\otimes 1+a\otimes 1\otimes 1\otimes b+(-1)^{|a||b|}1\otimes b\otimes a\otimes 1+1\otimes 1\otimes a\otimes b. $$
		\noindent Note that $(-1)_{*} (1)=1$,  $(-1)_{*} (a)=-a$ and $(-1)_{*} (b)=-b$.  It follows that $((-1,-1)\times (I,I))_{*} (23)_{*}(\Delta\times\Delta)_{*} (a\otimes b)$
		is $$a\otimes b\otimes 1\otimes 1-a\otimes 1\otimes 1\otimes b-(-1)^{|a||b|}1\otimes b\otimes a\otimes 1+1\otimes 1\otimes a\otimes b .$$ Hence, $\w_{*} (a\otimes b)=\cq_{*} ((-1,-1)\times (I,I))_{*} (23)_{*}(\Delta\times\Delta)_{*} (a\otimes b)=[a,b].$ Therefore, $\x I, I\xx_{*} (a\otimes b)=[a,b].$
	\end{proof}
	
	\subsection{Moore's result for the  Leray-Serre spectral  sequence}

	The following  proposition is due to Moore in his 
	1956 paper \cite{Moore56}.  A  review of the proof of this result is given in \cite[Section 2]{pxlfx}.

	Let $g\colon X\rightarrow Y$ be a map between simply-connected spaces and let $F_{g}$ be the homotopy fiber of $g$. Then we have a fiber sequence \[\lo Y\rightarrow F_{g}\rightarrow X\stackrel{g}\rightarrow Y.\] Suppose that $\mathbbm{k}$ is a field. Using these symbols, we give  
	the following  proposition.

	\begin{pro}\label{mdpxl}   In the  $\mathbbm{k}$-homology Leray-Serre spectral sequence $\n (E^{\bullet}_{\bullet,\bullet},d^{\bullet})\nn$ associated with the homotopy fibration $\lo Y\rightarrow F_{g}\rightarrow X$, we have $d^{r} (u\cdot v)=(d^{r}u)\cdot v$ for all $r\geq2$, all $i,j\in\z_{+}$, all $u\in E^{r}_{i,0}$, and all $v\in H_{j}(\lo Y)$, and the product ``$\cdot$" is induced by the right action of $\lo Y$ on  $F_{g}$\;$\colon F_{g}\times\lo Y\rightarrow F_{g}$.   This  spectral sequence is a right $H_{*} (\lo Y)$-module spectral sequence.  If $r=2$, we have $u\cdot v=u\otimes v\in E^{2}_{i,j}$.\hfill $\boxed{}$\end{pro}

	\section{Preliminary homotopy  theory}\label{sectl}
	
	In this section, spaces are understood in the ordinary sense and are not assumed to be localized at any prime $p$.
	
	\subsection{Gray's relative James construction}
	
	Gray's relative James construction $J(X,A)$ (which is denoted by $(X,A)_{\wq}$ in Gray's notation) 
	provides a beautiful model of the homotopy fiber of the pinch map $X\cup CA \rightarrow (X\cup CA)/X=\s A$. 
	
	For convenience, when applying the relative James construction, we denote the mapping cylinder of a map $X \rightarrow Y$ by $M_Y$.

	\begin{pro}\label{grdl} (\cite{Gray,hp2}) Let $X$ be a path-connected  \textit{CW}  complex and $A$ be its path-connected  subcomplex. Let $i\colon  A\hookrightarrow X$  be the inclusion. We have the inclusion  $i\colon  (A,A)\hookrightarrow (X,A)$. Then, the following statements hold.
		\begin{itemize}
			
			\item[\rm(1)]\label{grdl1a} $J(A,A)=J(A)$, where $J(A)$ is the ordinary James construction. There is a fiber sequence
\[
\begin{tikzcd}
    J(A)
    \arrow[r, hook, "J(i)"]
    & J(X,A)
    \arrow[r]
    & X \cup_i CA
    \arrow[r, "p"]
    & \Sigma A,
\end{tikzcd}
\]
where $p$ is the pinch map and $J(i)$ is the inclusion induced by
the inclusion of pairs $i\colon (A,A)\hookrightarrow (X,A)$. Moreover, if  $$K\stackrel{g}\longrightarrow L \stackrel{j_{_{L}}} \longrightarrow C_{g}\stackrel{q}\longrightarrow\s K$$ \noindent is a cofiber sequence where  $K$,\;$L$ are path-connected  \textit{CW}  complexes,\; then there exists a fiber sequence \;\vspace{-0.6\baselineskip} 
			\[\begin{tikzcd}[column sep=0.8cm]
				J(K) \arrow[r,hook,"J(i_g)"] & J(M_L, K) \arrow[r] & C_g \arrow[r, "q"] & \Sigma K,
			\end{tikzcd}\]
			\noindent where \textbf{$M_{L}$ is the mapping cylinder of $g$}, and 
			$ J(i_{g})$  is the inclusion   induced by  $i_{g} \colon  (K,K)\hookrightarrow (M_{L}, K)$.
			\item[\rm(2)]\label{grdl2a}   If  $X/A$ is path-connected, and  $\lo\s A$,  $X/A$ are \textit{CW} complexes of finite type,  (for example, $X$ is a finite simply-connected CW complex and $A$ is its skeleton), then$$\widetilde{H}_{\ast} (J(X,A); \mathbbm{k})\tg \widetilde{H}_{\ast} (X;\mathbbm{k})\otimes H_{\ast} (\lo\s A;\mathbbm{k})\;\;\text{ where} \;\mathbbm{k}\;\text{is a field}.$$  
			\item[\rm(3)]\label{grdl3a}There exists a filtration  where \,$J_{0} (X,A)=*,\;\, J_{1} (X,A)=X$, $$J_{0} (X,A)\subseteq J_{1} (X,A)\subseteq J_{2} (X,A)\subseteq\cdots,\quad J(X,A)=\bigcup_{m\geq 0}J_{m} (X,A).$$\vspace{-1.5\baselineskip}
			\item[\rm(4)]If  $X=\s X' $,\; $A=\s A'$,\;then $J_{2} (X,A)=X \cup _{[\mathrm{id}_{X},i]} C(X\wedge A')$, where $[\mathrm{id}_{X},i]$  is the Whitehead product.\item[\rm(5)]\label{grdl5a}  $J_{m} (X,A)/J_{m-1} (X,A)=X\wedge A^{\wedge (m-1)},\;\; \s J(X,A)\simeq\s \bigvee_{m\geq 0}X\wedge A^{\wedge m}.$\item[\rm(6)] There exist relative James-Hopf invariants $H_{m}\colon  J(X,A)\rightarrow J(X\wedge A^{\wedge (m-1)})$, which are natural with respect to pairs.
		\end{itemize}\hfill\boxed{}\;
	\end{pro}
	The following lemma is useful for studying the skeletons of the homotopy fibers. It follows directly from assertion (5) of the above proposition. For details on minimal cell structures, see Hatcher's textbook \cite[Proposition 4C.1, p.~429]{xh}. 
	\begin{lem}\label{j2}  Let $m\geq2$ be an integer and $X= Y\cup_{g} e ^{m+1}$ for some map $g$, where $Y$ is a \textit{CW} complex satisfying the following three conditions:  
		\begin{itemize}
			\item [\rm(i)] $Y$ is simply-connected and not contractible;
			\item [\rm(ii)]  $Y$ is of finite type, and the cell structure of $Y$ is minimal, that is, the cell structure of $Y$ is  consistent with its homology;
			\item [\rm(iii)]$\mathrm{dim} (Y)\leq m.$
		\end{itemize}Let $r\geq2$ be the  dimension of
		the bottom cell(s) of  $Y$. Then,  for the homotopy fiber $J( M_{Y},S^{m})$ of $X\xrightarrow[]{pinch}S^{m+1},$ we have $J_{i} ( M_{Y},S^{m})=\mathrm{sk}_{j} (J( M_{Y},S^{m}))$ where  $j=mi+r-1.$\hfill\boxed{}
		
	\end{lem}

	We make the following observation.

	\begin{rem}\label{Fdll} 
		Recall that $f\in\pi_{n+k} (S^{n})$ and $F$ is the homotopy fiber of the pinch map $\s C_{f} \rightarrow  S^{n+k+2}$.  Set the map $\af=[\mathrm{id}_{S^{n+1}},\s f]\in\pi_{2n+k+1} (S^{n+1})$. By Proposition~\ref{grdl} and Lemma~\ref{j2}, we see that  $F=J(M_{S^{n+1}},  S^{n+k+1})$. Its homology groups are given by the following isomorphism of graded groups  $$\widetilde{
			H}_{*} (F;\z)\tg \widetilde{
			H}_{*} (\bigvee_{i\geq0} S^{n+1} \wedge (S^{n+k+1})^{\wedge i};\z).$$  It follows  that  \begin{equation}\label{dafdy}F=(S^{n+1}\cup_{\af}  e^{2n+k+2})\cup_{\beta}  e^{3n+2k+3}\cup e^{4n+3k+4}  \cup\cdots
		\end{equation}
		for $\af=[\mathrm{id}_{S^{n+1}},\s f]$ and some $\beta\in \pi_{3n+2k+2} (S^{n+1}\cup_{\af}  e^{2n+k+2})$. Denote by $F^{(m)}$ the skeleton of $F$ containing $m$ cells. So,\begin{equation}\label{f2dy}F^{(2)}=\sk_{3n+2k+2} (F)=S^{n+1}\cup_{\af}  e^{2n+k+2},\end{equation}\begin{equation}\label{f3dy}F^{(3)}=\sk_{4n+3k+3} (F)=(S^{n+1}\cup_{\af}  e^{2n+k+2})\cup_{\beta}  e^{3n+2k+3}.\end{equation}

		\noindent  The maps $j_{3}$ and $j_{4}$  in Theorem \ref{hao1} are in fact given by the following sequence after localization at a prime $p$ different from 3,
		\begin{equation}\label{jhbbb}
			S^{3n+k+2} \overset{\,j_{4}\,}{\xyd} (S^{n+1}\vee S^{3n+k+2})\cup e^{5n+2k+3}\cup\cdots\stackrel{j_{3}}\longrightarrow \fe;
		\end{equation}
		\noindent here, $j_{3}$ is the homotopy fiber inclusion in the following  homotopy fibration,
		\begin{equation}\label{gyj3d5}
			(S^{n+1}\vee S^{3n+k+2})\cup e^{5n+2k+3}\cup\cdots\stackrel{j_{3}}\longrightarrow \fe\xrightarrow[]{pinch}S^{2n+k+2}.
		\end{equation}  
		\noindent See  Lemma~\ref{lpgfl} for $\mathrm{hofib} ( \fe\xrightarrow[]{pinch}S^{2n+k+2})$. \end{rem}

	\subsection{A lemma on the relation between the cofibration and the fibration}
	The following lemma seems well-known to experts. However, we have not found a citable reference containing a proof, so we include it here.
	\begin{lem}\label{zzd} Localize spaces at  a prime $p$. Let $m\geq2$ and let
		\[
		S^{m}\stackrel{g}\longrightarrow Y\stackrel{i}\hookrightarrow X = Y\cup_{g} e^{m+1}
		\]
		be a  cofiber sequence, where $Y$ is a simply-connected CW complex of finite type with $\dim(Y)\le m$; if $\mathrm{dim} (Y)=m$, additionally assume that the homomorphism $g_{*}=H_{m} (g)$ is trivial.
		Denote the homotopy fiber of $Y\stackrel{i}\hookrightarrow X$ by $F_{i}$. Let $S^{m}=\sk_{m} (F_{i})\stackrel{\,j'}\hookrightarrow F_{i}$ and $F_{i}\stackrel{j}\rightarrow Y$ be the  skeleton inclusion and  the homotopy fiber inclusion, respectively. Then  $ g= jj'\in\pi_{m} (Y)$
		up to a unit in  $\z_{(p)}$.  
	\end{lem}
	\begin{proof} By observing the $\z/p$-homology Leray-Serre spectral sequence associated with the homotopy fibration $F_{i}\rightarrow Y\rightarrow X$, we infer $\sk_{m} (F_{i})=S^{m}.$ Consider the exact sequence $\pi_{m} (F_{i})\stackrel{j_{*}}\longrightarrow \pi_{m} (Y)\stackrel{i_{*}}\longrightarrow \pi_{m} (Y\cup_{g} e^{m+1})$, that is, the exact sequence 
		$\x \,j'\,\xx\stackrel{j_{*}}\longrightarrow \pi_{m} (Y)\rightarrow \pi_{m} (Y)/\x g\xx$.  (Here, we make free use of the fundamental techniques for computing homotopy groups; see Remark~\ref{jssq} for an account of these techniques.) The exactness forces  $\x\, j_{*} (j')\,\xx=\x g\xx$. Hence the result holds.
	\end{proof}

	\subsection{The$\m p$ spherical elements in homology groups}
	There is a long history of studying$\m p$ spherical elements. For example, see Fred Cohen \cite[p.~77]{cds} and Wu \cite[pp.~17--18]{wu2003}. A nontrivial$\m p$ spherical element can be used to decompose the skeleton of the CW complex. The following lemma is well-known, and its straightforward proof is left to the reader.
	
	\begin{lem}\label{sfrc}
		Localize spaces at  a prime $p$. Let $X$ be a simply-connected CW complex of finite type and write  the $\z/p$-homology $H_{m} (S^{m})=\x \,\e\,\xx$. Suppose that there exists a map  $g\in\pi_{m} (X)$  with $m\geq2$ such that  $0\neq g_{*} (\e)\in H_{m} (X)$. (Such an element $g_{*} (\e)$, a$\m p$ Hurewicz image,  is said to be$\m p$ spherical.)  Then,
		$\sk_{m} (X)\simeq K\vee S^{m}$ for some simply-connected  CW complex $K$  with  $\mathrm{dim} (K)\leq m$. \hfill$\boxed{}$
	\end{lem}
	 \indent As an application of the$\m p$ spherical element, we present the following proposition.
	It is a reformulation and generalization of \cite[Proposition~2.1, p.~3249]{CW2013}, whose original proof is partially incorrect. (We denote their $\af$ and $n$ by $g$ and $r$, respectively.)   
	
	This proposition plays an essential role in the proof of the main theorem; in Lemma~\ref{q3max}, it will be used to show that $\mathrm{SQ}_{3}^{\mathrm{max}} (C_{f})\simeq\s^{2n+k+2}C_{f}$ for $p\neq3$.
	
	\begin{pro} \label{tl}Localize  spaces at  a prime $p$. Let $X=S^{m}\cup _{h}e^{r} \not\simeq*$ be given where $1\leq m<r$. Set the homology
		$\widetilde{H}_*(X)=\mathrm{Span}_{\z/p}\n u,v\nn$ where    $|u| = m$ and  $|v| = r$. Denote by $\e_{m+r}$ the generator of $ H_{m+r} (S^{m+r})\tg\z/p.$
		\begin{itemize} 
			
			\item [\rm(1)] If $p=2$,  then there exists a map $g \colon S^{m+r} \to X \wedge X$ such that $$0\neq g_{*} (\e_{m+r})=uv-vu=[u,v]\in H_{m+r} (X \wedge X  ).$$
			\item [\rm (2)] If $p\geq3$, then there exists a map $g \colon S^{m+r} \to X \wedge X$ and some integer $t\in [1,p-1]$
			such that $$0\neq g_{*} (\e_{m+r})=uv+tvu\in H_{m+r} (X \wedge X  ).$$
			
		\end{itemize}
		
	\end{pro}
	
	\begin{proof}  Consider the homotopy cofibration $X\wedge S^{r-1} \xrightarrow{\mathrm{id}_X\wedge h} X\wedge S^m\xrightarrow{\mathrm{id}_X\wedge i} X\wedge X$  where       $i\colon S^{m}\rightarrow X$ denotes the inclusion of the bottom cell.
		\noindent Take the $(2r-1)$--skeleton of $X\wedge X$ and note that $S^m\wedge S^{r-1} \simeq \mathrm{sk}_{2r-2} (X\wedge S^{r-1})$;
		hence we obtain a homotopy cofibration \begin{equation}\label{cfb1}
			S^m\wedge S^{r-1} \xrightarrow{i\wedge h} X\wedge S^m
			\to \mathrm{sk}_{2r-1} (X\wedge X).
		\end{equation}
		\noindent  Next we will show that $i\wedge h=*$. Since $i\wedge h$
		is equal to the composition
		\[
		S^m\wedge S^{r-1}
		\xrightarrow{\mathrm{id}_{S^m}\wedge h}
		S^m\wedge S^m
		\xrightarrow{i\wedge \mathrm{id}_{S^m}}
		X\wedge S^m,
		\]
		we obtain the following commutative diagram with bottom row a homotopy cofibration:
		\[
		\begin{tikzcd}[column sep=large]
			S^{m}\wedge S^{r-1}
			\arrow[r,"\mathrm{id}_{S^{m}}\wedge h"]
			\arrow[d,"\tau'"']
			&
			S^{m}\wedge S^{m}
			\arrow[d,"\tau"]
			\\
			S^{r-1}\wedge S^{m}
			\arrow[r,"h\wedge \mathrm{id}_{S^{m}}"]
			&
			S^{m}\wedge S^{m}
			\arrow[r,"i\wedge \mathrm{id}_{S^{m}}"]
			&
			X\wedge S^{m},
		\end{tikzcd}
		\]
		\noindent   
		where  $\tau'$ and $\tau$ are the switching maps.   So, $$(i\wedge \mathrm{id}_{S^{m}})\hc[\pm 1]\hc(\mathrm{id}_{S^{m}}\wedge h)=(i\wedge \mathrm{id}_{S^{m}})\hc(h\wedge \mathrm{id}_{S^{m}})\hc[\pm 1]=*.$$
		\noindent The left side of this equation is just $\pm(i\wedge h)$.
		Thus, $i\wedge h=*$ in the homotopy cofibration (\ref{cfb1}). Therefore $$\mathrm{sk}_{2r-1} (X\wedge X)\simeq (\s ^{m}X)\vee S^{m+r}.$$
		\noindent   
		Define $g'$ to be the composition of the  inclusions $S^{m+r}\stackrel{i_{0}}\hookrightarrow \mathrm{sk}_{2r-1} (X\wedge X)
		\hookrightarrow X\wedge X.$ Then  $0\neq g'_{*} (\e_{m+r})$ is$\m p$ spherical. 
		
		Consider  the case $p=2$. If $\Sigma^{m}X\not\simeq S^{2m}\vee S^{m+r}$, then neither $uv$ nor $vu$ is$\m2$ spherical in $$H_{m+r} (X\wedge X )=\n0,uv,vu,[u,v]\nn,$$ which forces
		$[u,v]$ to be$\m2$ spherical. Thus $g'_{*} (\e_{m+r})=[u,v]$. In this case, take $g=g'.$ If $\Sigma^{m}X\simeq S^{2m}\vee S^{m+r}$, then all elements in $H_{m+r} (X\wedge X )$  are$\m2$ spherical; note that we can make use of the self homotopy equivalences in $[S^{m+r}\vee S^{m+r},S^{m+r}\vee S^{m+r}]$ to change the basis of $H_{m+r} (\mathrm{sk}_{2r-1} (X\wedge X))$. In this case, take $g$ to be the composition $$S^{m+r}\stackrel{i_{0}}\hookrightarrow
		\mathrm{sk}_{2r-1} (X\wedge X)\xrightarrow[\simeq]{e} \mathrm{sk}_{2r-1} (X\wedge X)
		\hookrightarrow X\wedge X$$
		for some self homotopy equivalence $e$ in $[\mathrm{sk}_{2r-1} (X\wedge X),\mathrm{sk}_{2r-1} (X\wedge X)]\tg[S^{2m}\vee S^{m+r}\vee S^{m+r},S^{2m}\vee S^{m+r}\vee S^{m+r}].$
		
		Consider  the case $p\geq3$.  If $\Sigma^{m}X\not\simeq S^{2m}\vee S^{m+r}$, then  neither $uv$ nor $vu$ is$\m p$ spherical in $$H_{m+r} (X\wedge X )=\n xuv+yvu\;|\;0\leq x\leq p-1,\;0\leq y\leq p-1\nn,$$ which forces $g'_{*} (\e_{m+r})=t_{1}uv+t_{2}vu$  for some integers $t_{1},t_{2}\in [1,p-1].$  Then choose $g=\frac{1}{t_{1}}g'$, satisfying the desired property. If $\Sigma^{m}X\simeq S^{2m}\vee S^{m+r}$, then all elements in $H_{m+r} (X\wedge X )$  are$\m p$ spherical; as in the case $p = 2$, we can obtain the desired map $g$ by using a suitable self‑
homotopy equivalence.\end{proof}
	\;\\ \indent Here is the reason why the proof of \cite[Proposition~2.1, p.~3249]{CW2013}  is incorrect. Let $p=2.$
	The last sentence of its proof claims that, since $[u,v]$ is the only primitive generator of
	$\widetilde{H}_{*} (X\wedge X )$, one has
	\begin{equation}\label{cdc}
		g_{*} (\e_{m+r})=[u,v]\in \widetilde{H}_{*} (X\wedge X ).
	\end{equation}
	
	\noindent The justification for this claim (i.e., ``since $[u,v]$ is the only primitive generator") is incorrect, although the conclusion itself is correct.
	If 
	$\Sigma^{m}X\simeq S^{2m}\vee S^{m+r}$, then the elements $uv$, $vu, [u,v]$ are all$\m2$ spherical and hence  $uv$, $vu, [u,v]$ are all  primitive generators of
	$\widetilde{H}_{*} (X\wedge X )$.
	
	\section{The homology of the looped homotopy fiber}\label{seccmn}

	In this section, spaces are understood in the ordinary sense and are not assumed to be localized at any prime $p$. But the homology is taken with coefficients in $\z/p$ where  $p$ is prime. And the ground field of all algebras, coalgebras and Hopf algebras is $\z/p$.

	In order to prove our generalization of the result of Cohen, Moore, and Neisendorfer,
	we need the following proposition \cite[Theorem 4.4.6, p.~171]{hopfjc} to examine the Hopf module structure. It should be understood that  Proposition~\ref{emsdy} remains valid in both the graded and ungraded settings, since the proof is    essentially unchanged in either working category.

	To state Proposition \ref{emsdy} \cite[Theorem 4.4.6, p.~171]{hopfjc}, let us recall the definition  of Hopf modules \cite[Definition 4.4.1, p.~169]{hopfjc}.
	Take the ground ring to be a field $\mathbbm{k}$. Let $A$ be a Hopf algebra. A vector space $M$ is called a right $A$-Hopf module if $M$ has a right $A$-module structure (the action of  $\af \in A$ on $\beta \in M$ will be denoted by $\beta \star\af$), and a right $A$-comodule structure, given by the map $\rho: M \rightarrow M \otimes A,\;\;
	\rho(\beta)=\sum \beta_{(0)}\otimes \beta_{(1)}$ such that for any $\beta\in M$ and $\af\in A$, $$\rho(\beta\star \af)=\sum(-1)^{|\af_{1}||\beta_{(1)}|} (\beta_{(0)}\star \af_{1})\otimes (\beta_{(1)} \af_{2}),$$ where $\beta_{(1)} \af_{2}$ is the product of $\beta_{(1)}$ and $ \af_{2}$ in the Hopf algebra $A$. Note that  $\sum \beta_{(0)}\otimes \beta_{(1)}$ and $\sum \af_{1}\otimes \af_{2}$ denote the Sweedler notation for coactions and coproducts, respectively; see \cite[p.~5, p.~66]{hopfjc}. The subspace of coinvariants $M^{\mathrm{co}\,A}$, defined in \cite[Definition 4.4.4, p.~171]{hopfjc}, is naturally isomorphic to the cotensor product $M\boxed{}_{A}\mathbbm{k}$.

	\begin{pro0}\label{emsdy} (\cite[Theorem 4.4.6, p.~171]{hopfjc},\,\textbf{The fundamental theorem of Hopf modules}) Take the ground ring to be a field $\mathbbm{k}$.
		Let $A$ be a Hopf algebra, and let  $M$ be a right $A$-Hopf module with the comodule structure map $\rho\colon M \rightarrow M \otimes A.$ 
		Then, the map $\kj\colon (M\boxed{}_{A}\mathbbm{k}) \otimes A \rightarrow M,$ defined by $\kj((\beta\otimes1) \otimes \af) = \beta \star \af,$ is an isomorphism of right $A$-Hopf modules. \hfill\boxed{}
	\end{pro0}

	In the context of a tensor algebra, the element denoted by $\mathrm{ad}^{m} (y)(x)$,
	namely the $m$-fold Lie bracket arranged in a left-nested manner,
	with $x$ occupying the leftmost entry and $y$ occupying the remaining entries,
	and with $\mathrm{ad}^{0} (y)(x)=x$,
	is explained in Section~\ref{secconv}.\footnote{\textbf{Warning:} In \cite{CMN1979}, the notation $\mathrm{ad}^{m} (y)(x)$ refers to the
		right-normed Lie bracket.
		We use the left-normed Lie bracket.
		The two conventions agree up to sign and have the same behavior.} Recall that  by a Hopf algebra
	$T(V)$ we mean  the Hopf algebra obtained by letting all elements of 
	$V$ be primitive.
	
	We now state the following theorem, which is a generalization of the Cohen-Moore-Neisendorfer result \cite[Corollary~9.4, p.~150]{CMN1979}.
	\begin{thm0}
		\label{dlydexj}  Let  $X=S^{m}\cup e^{r}$ be given where $2\leq m< r.$ Let $F_{q}$ be the homotopy fiber of the pinch map $q\colon  X\rightarrow S^{r}.$ If   $H_{*} (\lo  X)\tg T(u,v)$ 
		as Hopf algebras in which $|u|=m-1,|v|=r-1$,  then  the fiber sequence $\lo F_{q}\xrightarrow[]{\lo j} \lo X \xrightarrow[]{\lo q}\lo S^{r}$ induces  a short exact sequence of Hopf algebras
		\[H_{*} (\lo F_{q})\xrightarrow[]{(\lo j)_{*}} H_{*} (\lo X) \xrightarrow[]{(\lo q)_{*}}H_{*} (\lo S^{r}).\] 
		\noindent
		Moreover, $\lo j$ induces an isomorphism of Hopf algebras
		$$H_{*} (\lo F_{q})\tg T(\n\mathrm{ad}^{i} (v)(u)\;|\; i\geq0\nn)= T(u,\,[u,v],\,[[u,v],v],\,[[[u,v],v],v],\;\cdots).$$      
	\end{thm0}

	\begin{proof}  In fact, this theorem is essentially due to Cohen-Moore-Neisendorfer \cite[Proposition~9.3, Corollary~9.4]{CMN1979}. The goal of our proof is to verify that their method (whose original presentation is economical and focused on the essential idea) remains effective in the context of the generalized hypotheses. The proof has four steps.  Set the homology $H_{*} (\lo S^{r} )=T(\e)$ where $|\e|=r-1$.\;\vspace{-0.5\baselineskip}\;\\\\
		\noindent\textbf{STEP-1} Examine the $(\lo q)_{*}$-images.  We show that $(\lo q)_{*} (u)=0$ and  $(\lo q)_{*} (v)=\e$  up to a unit in $\z/p$;  hence we may choose the generator $\e$ so that $(\lo q)_{*} (v)=\e$. 
		
		First we analyze $(\lo q)_{*} (v)$. Consider the $\z/p$-homology Leray-Serre spectral sequence $\n (E^{\bullet}_{\bullet,\bullet},d^{\bullet})\nn$ associated with the homotopy fibration $\lo X\rightarrow*\rightarrow X$. By  Proposition \ref{mdpxl}, this  spectral sequence is an  $H_{*} (\lo X)$-module spectral sequence, $d(ab)=(da)b$. Set $\widetilde{H}_{*} (X)=\mathrm{Span}_{\z/p}\n u', v'\nn$ with $|u'|=m,\,|v'|=r$ and $d_{m} (u')=u$. In the $E^{m}$-page, $d(u')=u$, $d(u'\af)= (du')\af= u\af$ for each $\af\in E^{2}_{0,\bullet}$. Note that in $E^{m}_{0,\bullet}\tg E^{2}_{0,\bullet}$, $v\neq u\af$ for any  $\af\in E^{2}_{0,\bullet}$. So $v\in E^{2}_{0,r-1}$ survives to $E^{r}_{0,r-1}.$ Then we infer that  $d(v')=\ell v$ for some $1\leq\ell\leq p-1$. Notice that there is an obvious morphism from the homotopy fibration $\lo X\rightarrow*\rightarrow X$ to the homotopy fibration $\lo S^{r}\rightarrow*\rightarrow S^{r}$ induced by the map $q$.
		Thus,
		$(\lo q)_{*} (v)=\e$ up to a unit in $\z_{(p)}$,  
		resulting from the  naturality of the transgressions of the Leray-Serre spectral sequences. By this naturality, we also obtain $(\lo q)_{*} (u)=0$.
		\;\vspace{-0.5\baselineskip}\\
		
		\noindent\textbf{STEP-2} We proceed to justify that the comodule $H_{*} (\lo X )$ is an injective right $H_{*} (\lo S^{r} )$-comodule.  Write tensors $a\otimes b$ in $T(u,v)$ as $ab$ rather than $a\otimes b$.  At the purely algebraic level, we may use the identification of Hopf algebras $$H_{*} (\lo S^{r} )=T(v)\xyd T(u,v).$$ Additionally,  $H_{*} (\lo X )=T(u,v)$ is a right $H_{*} (\lo S^{r} )$-module whose module structure map is given by $$T(u,v)\otimes T(v)\rightarrow T(u,v),\;a\otimes b\mapsto a\star b:=ab.$$ \noindent Meanwhile, $H_{*} (\lo X) $ is a right $H_{*} (\lo S^{r})$-comodule  with its canonical  comodule structure map $$\rho_{*}\colon  H_{*} (\lo X )\rightarrow H_{*} (\lo X)\otimes H_{*} (\lo S^{r} ),\;\text{where}\; \rho=(\lo\mathrm{id},\lo q)\colon \lo X\rightarrow \lo X\times \lo S^{r}.$$ Observe that $\rho=(\lo\mathrm{id},*)\cdot (*,\lo q)$, the product of maps resulting from the $H$-structure. So,  $$\rho_{*} (u)=u\otimes1, \;\;\rho_{*} (v)=v\otimes1+1\otimes v=\Delta_{*} (v),$$ \noindent where $\Delta\colon \lo S^{r}\rightarrow \lo S^{r}\times \lo S^{r}$ denotes the diagonal map. We know that $\rho_{*}$ and $\Delta_{*}$ are algebra homomorphisms since $\rho$ and $\Delta$ are loop maps under the usual identifications. Let $i\geq1$. We have $$\rho_{*} (v^{i})=(\rho_{*} (v))^{i}=(\Delta_{*} (v))^{i}=\Delta_{*} (v^{i})=\sum (v^{i})_{1}\otimes (v^{i})_{2}.$$
		\noindent
		Take an arbitrary $\beta\in T(u,v)=H_{*} (\lo X )$. Then, $\rho_{*}(\beta)=\sum\beta_{(0)}\otimes\beta_{(1)}$.
Using the graded multiplication on the tensor product, we obtain
\[
\begin{aligned}
\rho_{*}(\beta\star v^{i})
&=\rho_{*}(\beta v^{i})
 =\rho_{*}(\beta)\rho_{*}(v^{i})\\
&=\left(\sum\beta_{(0)}\otimes\beta_{(1)}\right)
  \left(\sum(v^{i})_{1}\otimes(v^{i})_{2}\right)\\
&=\sum(-1)^{|\beta_{(1)}||(v^{i})_{1}|}
  \bigl(\beta_{(0)}(v^{i})_{1}\bigr)
  \otimes\bigl(\beta_{(1)}(v^{i})_{2}\bigr)\\
&=\sum(-1)^{|\beta_{(1)}||(v^{i})_{1}|}
  \bigl(\beta_{(0)}\star(v^{i})_{1}\bigr)
  \otimes\bigl(\beta_{(1)}(v^{i})_{2}\bigr).
\end{aligned}
\]
		\noindent So, $H_{*} (\lo X )$ is a  right $H_{*} (\lo S^{r} )$-Hopf module. Then, by Proposition \ref{emsdy}, we see that $$H_{*} (\lo X)\tg\big(H_{*} (\lo X)\boxed{}_{H_{*} (\lo S^{r})}\z/p\big)\otimes H_{*} (\lo S^{r} )$$ 
		\noindent as right $H_{*} (\lo S^{r} )$-Hopf modules. It follows that  $H_{*} (\lo X )$ is   a free  right $H_{*} (\lo S^{r} )$-comodule and therefore an injective  right $H_{*} (\lo S^{r} )$-comodule. (Recall that a right $C$-comodule $M$ is said to be free if $M\tg N\otimes C$ as right $C$-comodules for some vector space $N$. See \cite[Definition 2.4.1, p.~91]{hopfjc}. Some authors call it  a cofree comodule rather than a free one; we follow the convention of \cite{hopfjc}.)\\ \vspace{-0.6\baselineskip}\\
		\noindent\textbf{STEP-3} Use the Eilenberg-Moore  spectral sequence.

		Applying  \cite[Theorem~3.2]{ms1} and \cite[Proposition~3.1]{ms1}  for the homotopy pull-back of $\lo X\rightarrow\lo S^{r} \leftarrow *$,   we have a   second-quadrant spectral sequence $\n (E^{\bullet}_{\bullet,\bullet},d^{\bullet})\nn\Rightarrow H_{*} (\lo F_{q})$ with $$E^{2}_{*,*}\tg\mathrm{Cotor}_{*,*}^{H_{*} (\lo S^{r})} ( H_{*} (\lo X),\z/p)\text{\,\;as  bi-graded Hopf algebras}.$$
		\noindent Recall that the comodule $H_{*} (\lo X )$ is an injective right $H_{*} (\lo S^{r} )$-comodule. So, we have an injective resolution of $H_{*} (\lo X)$ concentrated in degree zero. Then, by the definition of the functor  $\operatorname{Cotor}$, we obtain
		\[E^{2}_{-i,*}=\operatorname{Cotor}^{H_{*}(\lo S^{r})}_{-i,*}
		\bigl(H_{*}(\lo X),\mathbb{Z}/p\bigr)=
		\begin{cases}
			H_{*}(\lo X)\boxed{}_{H_{*}(\lo S^{r})}\mathbb{Z}/p,
			& i=0,\\[4pt]
			0,
			& i>0.
		\end{cases}
		\]
		So, the spectral sequence collapses at the $E^{2}$-page and $E^{\wq}$ is concentrated on the vertical axis $E^{\wq}_{0,*}\cong E^{2}_{0,*}$. Notice that in this case, the bi-graded Hopf algebras $\mathrm{Gr}\,H_{*} (\lo F_{q})$ and $E^{\wq}$ can be viewed as singly graded Hopf algebras in the obvious way.  It follows that
		$H_{*}(\lo F_{q})=\mathrm{Gr}\,H_{*}(\lo F_{q})\tg E^{\wq}_{0,*}\tg E^{2}_{0,*}$ as  Hopf algebras. 
		Therefore, $\lo j$ induces the following isomorphism of Hopf algebras:  $$H_{*} (\lo 
		F_{q})\tg H_{*} (\lo X)\boxed{}_{H_{*} (\lo S^{r})}\z/p= T(u,v)\boxed{}_{T(\e)}\z/p\text{\; as  Hopf algebras.}$$

		\noindent\textbf{STEP-4}  We justify that 
		$T(u,v)\,\boxed{}_{ T(\e)} \z/p=T(\n\mathrm{ad}^{i} (v)(u)\;|\; i\geq0\nn).$  (Here, we use ``='' since $T(u,v)\,\boxed{}_{ T(\e)} \z/p$ is canonically viewed  as a subset of $T(u,v).$)

		In order to complete this step, we need to
		carry out a calculation of the Euler-Poincar\'e series, parallel to the calculation in the proof of \cite[Proposition~4.5, p.~135]{CMN1979}. (Regrettably, the symbols $V$ and $W$  are used with quite different meanings in \cite[Lemma~3.13, p.~132]{CMN1979} and \cite[Corollary 3.11, p.~131]{CMN1979}, which may cause confusion. Here, we follow the convention of \cite[Lemma~3.13, p.~132]{CMN1979}.)

		Let $L(-)$ be the free Lie algebra functor. Consider the  epimorphism $L(u, v)\rightarrow L(v)$ given by $u\mapsto0,\; v\mapsto v$. Let $L_{0}=L(V')$   be the kernel of this epimorphism. (Note that this kernel is freely generated; this follows from  
		\cite[Corollary~3.11, p.~132]{CMN1979}.) Then we have a short exact sequence of free Lie algebras \vspace{-0.5\baselineskip}
		\[\begin{tikzcd}
			L_0 \arrow[r] 
			\arrow[d, equals] 
			& L \arrow[r] 
			\arrow[d, equals] 
			& L'' 
			\arrow[d, equals] \\L(V')  
			& L(u, v)  
			& L(v).
		\end{tikzcd}\]
		\noindent Applying the universal enveloping algebra functor $U(-)$, we have a short exact sequence of Hopf algebras,
		\[\begin{tikzcd}
			A' \arrow[r] 
			\arrow[d, equals] 
			& A \arrow[r] 
			\arrow[d, equals] 
			& A'' 
			\arrow[d, equals] \\UL(V')  
			& UL(u, v)  
			& UL(v).
		\end{tikzcd}\]
		\noindent We calculate the  Euler-Poincar\'e series of $V'$ following \cite[Corollary 3.11, p.~132]{CMN1979}, $$\chi (V')=1-(\sum_{i=0}^{\wq} (t^{r-1})^{i})\cdot(1-t^{m-1}-t^{r-1})=\sum_{i=0}^{\wq}\,t^{m-1+i(r-1)}.$$
		\noindent Note that the Hopf algebra $T(V')=UL(V')$. By \cite[Lemma~3.12, Lemma~3.13, p.~132]{CMN1979}, we see that the module  of indecomposable elements  $$Q(A')=Q(UL(V'))=Q(T(V'))=V'$$ becomes a free 
		$T(v)$-module  with $v^{i}$ acting on $V'$ by 
		$\mathrm{ad}^{i} (v)(-)$.  Thus,   $V'$ has $\mathrm{ad}^{i} (v)(u)\; ( i\geq0)$ as a $\z/p$-basis. Then, by the definition of short exact sequences of Hopf algebras (\cite[Definition 3.3, p.~129]{CMN1979}),  we obtain  $T(u,v)\,\boxed{}_{ T(\e)} \z/p=UL(V'),$  completing the proof. \end{proof}
We make the following remark about the ground ring.\begin{rem0}\label{jsxs} Cohen, Moore, and Neisendorfer \cite[p.~149]{CMN1979} consider the ground ring ``$R=\z/(p^{s})$ with  $1\leq s \leq r$ if $p$ is an odd prime and with $s = 1$ if $p = 2$''.  In their \cite[Corollary~9.4, p.~150]{CMN1979}, the ground ring is taken as such a ring $R$. In  Theorem \ref{dlyde} (Theorem \ref{dlydexj}),
		the integral  homology groups of the spaces $X=S^m\cup e^{r}$ and $F_{q}=\mathrm{hofib}(X\twoheadrightarrow S^{r})$  are torsion free if $r-m\geq2$. Now, assume that  $H_{i}(\lo F_{q};\z_{(p)})$  is a free $\z_{(p)}$-module for each $i$ if $r-m\geq2$. Then,
		replacing   the ground ring $\z/p$ in  Theorem \ref{dlyde} (Theorem \ref{dlydexj})  by  the above ring $R$ does not affect the proof, and the theorem remains valid under the above free $\z_{(p)}$-module assumption. For the study of $p$-local simply connected CW complexes of finite type, homology with $\mathbb{Z}/p$ coefficients is the primary object of interest. 
 \end{rem0}

 \begin{rem0}We point out a minor error in
\cite[Corollary~9.4, p.~150]{CMN1979}.
The authors state that their result includes  the $\z/2$-homology of 
$\lo F=\lo\,\mathrm{hofib}
(\s\mathbb RP^{2}\xrightarrow{pinch}S^{3})$. However, in this case, 
their stated isomorphism is an isomorphism of
algebras, rather than of (primitively generated) Hopf algebras. Additionally, their $[[u,v],v]$ is not primitive in this case.

In fact it can be shown that $H_{*}(\lo\s\mathbb RP^{2};\z/2)$ is not a primitively generated Hopf algebra.  Let $u$ and $v$ generate $H_{1}(\mathbb RP^{2};\z/2)$ and $H_{2}(\mathbb RP^{2};\z/2)$, respectively. Note that $$H_{*}(\mathbb RP^{\wq};\z/2)=(H^{*}(\mathbb RP^{\wq};\z/2))^{\mathrm{dual}}=\Gamma (u),\;(|u|=1),$$
the divided polynomial coalgebra. (See Hatcher's textbook \cite[Example 3C.5, p.~286]{xh} for divided polynomial coalgebras.) So, in $H_{*}(\mathbb RP^{2};\z/2)\xyd H_{*}(\lo\s\mathbb RP^{2};\z/2)$, the coproducts are given by $ \Delta(u)=1\otimes u+u\otimes1$, $\Delta(v)=1\otimes v+v\otimes1+u\otimes u$, ($v=u_{2}\in \Gamma (u)$). Direct calculations show that $[[u,v],v]$ is not primitive in $H_{*}(\lo F;\z/2)\xyd H_{*}(\lo\s\mathbb RP^{2};\z/2)$.
 Here, we use ``$\xyd$'' since the corresponding map of spaces induces an injection. \end{rem0}
	Recall that for a graded module $V$, the homomorphism $\cg^{-1}\colon V\rightarrow\cg^{-1} (V)$ is the isomorphism shifting degrees down by $1$. We also call $\cg^{-1}$ the  desuspension. We also record the following generalization. This theorem will not be used in the sequel. Its proof is similar to that of the preceding theorem, but requires more computation, and is therefore omitted.
	\begin{thm0}\label{dadxj} 
		Let $X$ be a finite simply-connected CW complex with a CW decomposition  $X = X' \cup e^{r}$ and $\mathrm{rank} (\widetilde{H}_{*} (X))\geq2$. Here, \(r\geq 3\),  \(X'\) is a  simply-connected subcomplex with $\dim (X')\leq r$, and $\mathrm{rank} (\widetilde{H}_{r} (X))=\mathrm{rank} (\widetilde{H}_{r} (X'))+1$. Let $F_{q}$, $j$ and $q$ be given in the following homotopy fibration, $$F_{q}\stackrel{j}\rightarrow X \xrightarrow[pinch]{q}X/X'=S^{r}.$$
		\noindent If $H_{*} (\lo X)\cong T(\cg^{-1} (      \widetilde{H}_{*} (X)  ))$ as Hopf algebras,
		then there exists a short exact sequence of Hopf algebras \begin{equation}\label{zhl}
			H_{*} (\lo F_{q})\xrightarrow[]{(\lo j)_{*}} H_{*} (\lo X) \xrightarrow[]{(\lo q)_{*}}H_{*} (\lo S^{r}).
		\end{equation}
		\noindent In addition, we may denote the desuspension of $\widetilde{H}_{*} (X)$  by \begin{equation}\label{020222d22}
			\cg^{-1} (      \widetilde{H}_{*} (X))=\mathrm{Span}_{\z/p}\n u_{1},u_{2},\cdots,u_{s},v\nn \;\text{with}\;|u_{i}|\leq |v| =r-1,\;   (\lo q)_{*} (u_{i})=0\,(\forall\,i)\;\text{and}\; (\lo q)_{*} (v)\neq0.\;
		\end{equation}
		Then $\lo j$ induces an isomorphism of Hopf algebras \begin{equation} \label{adivuk}
			H_{*} (\lo F_{q})\tg T(\n\mathrm{ad}^{i} (v)(u_{k})\;|\; i\geq0,\; 1\leq k\leq s\nn).
		\end{equation}
		This isomorphism does not depend on the particular choice of the above generators, provided that they satisfy the constraints imposed by $(\lo q)_{*}$. \hfill$\boxed{}$

	\end{thm0}

	\section{Functorial decompositions of looped suspensions--Selick and Wu's $\mathrm{A}^{\min}$-theory}\label{secamin1}
	\subsection{Historical background}
	
	Although the notation  $n$ has already been fixed, in this section the reader should be aware that the variable $n$ is unrelated to the  variable
	$n$ appearing in $C_{f}=S^{n}\cup e^{n+k+1}.$

	In his 1995 PhD thesis \cite{Wubylw}, written under the supervision of Fred Cohen, Wu obtained product decompositions of looped $p$-local suspensions for every prime $p$.  Let $Y$ be a $p$-local path-connected CW complex, $X=\s Y$  and $L_{q} (X)=\mathop{\mathrm{hocolim}}\limits_{\beta_{q}/q}\, X^{ \wedge q}$ where $q$ is a prime distinct from $p$. The homotopy colimit is explained in Remark~\ref{a1002}. Let $\n q_{j}\nn_{j=1}^{\wq}$ be the increasing sequence of all primes distinct from $p$.  Then there exists a space $A$ such that \begin{equation}\label{amqs}
		\lo\s X\simeq A\times\prod _{j=1}^{\wq} \lo\s L_{q_{j}} (X).    
	\end{equation}

	Instead of examining 
	decompositions of loop spaces individually, in 2000, Selick and Wu \cite{sw2000},  by means of a bridge between homotopy theory and 
	representation theory of symmetric groups, deduced  a functorial decomposition of  looped $p$-torsion suspensions, which is a functorial analogue of Equation (\ref{amqs}) in the case that $\mathrm{id}_{ X}\in[ X, X]$ is of order $p^{r}$. In 2006, this result was generalized by them to  looped suspensions of  $p$-local spaces \cite{sw2006}, in which the functorial analogue of Equation (\ref{amqs})  was completely derived; moreover, the condition ``$X=\s Y$'' was removed.  Roughly speaking,  the theory tells us that a functorial coalgebra decomposition of a tensor coalgebra is realized homotopy theoretically, producing a functorial decomposition of the looped suspensions into a product space up to homotopy. 
	\subsection{Conventions}
	We introduce some notation and definitions to state the results of \cite{sw2000} and \cite{sw2006}.

	Suppose that  $\mathbbm k=\z/p\;\text{or}\;\mathbb Q$, and that $V$ is a $\mathbbm k$-vector space. Notice that for a tensor algebra $T(V)$,
	the definition $[a,b]=a b-(-1)^{|a||b|}b a $ gives a Lie algebra structure to $T(V)$. Observe that for a  Hopf algebra $T(V)$, there is a canonical identification $$T(V)=UL(V), $$ \noindent
	where $UL(V)$ is the universal enveloping algebra of the free Lie algebra $L(V)$.
	
	We assume that we are
	working within a Grothendieck  universe.
	Let $\mathbf{CW}_{(p)}^{0}$ be the category of   path-connected $p$-local CW-complexes; let $\hCW$ denote the homotopy category whose objects are the same as those of $\mathbf{CW}_{(p)}^{0}$.

	Denote the category  of  $\mathbbm k$-vector spaces $M$ satisfying $M_{0}=0$ by $\Vect.$ 
	($M_{0}$ means  the component of $M$ of degree 0.)
	Denote the categories of   algebras, coalgebras and Hopf algebras over $\mathbbm k$  by  $\mathbf{Alge}_{\mathbbm k}$, $\mathbf{Coalg}_{\mathbbm k}$ and $\mathbf{Hop}_{\mathbbm k}$, respectively.
	(As mentioned earlier, all these objects are  graded.)    
	
	For categories $\mathcal{C}$ and $\mathcal{D}$, denote the functor category  whose objects are all  functors $\mathcal{C}\rightarrow\mathcal{D}$ by $\Fun(\mathcal{C},\mathcal{D}).$

	Let $\mathcal{S}_{n}$ denote the $n$-th symmetric group, where permutations are composed from right to left. (For example, $(12)(13)=(132)$ but $(12)(13)\neq (123)$.) 
	
We fix the convention of the action of $\mathcal{S}_n$ on tensor power $V^{\otimes n}$: for $\cg\in \mathcal{S}_n$, and elements $x_i\in V\in \Vect$, \begin{equation}\label{b1011}
    \cg(x_{1}\otimes x_{2}\otimes\cdots \otimes x_n):=(\mathrm{sig})(x_{\cg^{-1}(1)}\otimes x_{\cg^{-1}(2)}\otimes\cdots \otimes x_{\cg^{-1}(n)}).
\end{equation}
Here, $\mathrm{sig}=\mathrm{sig}(\cg;x_{1}\otimes x_{2}\cdots x_n)\in\n\pm1\nn$ is the Koszul sign determined by $(12)(a\otimes b)=(-1)^{|a||b|}b\otimes a.$ By the canonical extensions, this action gives the representations $\mathscr{R}\colon\z_{(p)}[\mathcal{S}_n]\to \mathrm{Hom}(V^{\otimes n},V^{\otimes n})$ with $\mathbbm{k}\in\n\mathbb{Q},\z/p\nn$ and $\overline{\mathscr{R}}\colon\frac{\z}{p}[\mathcal{S}_n]\to \mathrm{Hom}(V^{\otimes n},V^{\otimes n})$ with $\mathbbm{k}=\z/p$. We usually write $\mathscr{R}(\ell)$ and $\overline{\mathscr{R}}(\ell)$ as $\ell$ for convenience.

	The functor $L_{n}$ defined below is of great significance in Selick-Wu's $\mathrm{A}^{\min}$-theory. As we will see in Proposition \ref{amindl}, $\cg^{-1}\widetilde{H}_{*} (\mathrm{SQ}_{n}^{\max} (X))$ is always a vector subspace of $L_{n} (\widetilde{H}_{*} (X)).$
	
	For convenience,  we use the convention that 
	$[a,b]$ is a 2-fold commutator and 
	$[[a,b],c]$ is a 3-fold commutator;  similar conventions apply to commutators involving more entries and to iterated Lie brackets, even though  $[[a,b],c]$ is usually called  a 2-fold   bracket. 

   \begin{definition}\label{ydddy}
		Let $n\geq2$ be an integer and $V\in\Vect$.   Define $L_{n} (V)$ to be the vector subspace of $V^{\otimes n}\xyd T(V)$ spanned by all $n$-fold  commutators  with  entries in $V$. 
		This yields a functor $\Vect\rightarrow\Vect$, denoted by   $L_{n}$. Note that we therefore have a decomposition  of the free Lie algebra $L(V)=\bigoplus_{j\geq1}\,L_{j} (V).$ Here,  $L_{1} (V)=V.$
	\end{definition}

	\subsection{Some preliminaries on representation theory}
	We introduce some fundamental ideas used by  Selick-Wu in \cite{sw2000,sw2006}; also see Wu \cite[Section~3]{wu2003}. Let  $X\in\hCW$  where $p$ is a prime. For a $p$-local suspension, denote its self-map of degree $a\in\z_{(p)}$ by $[a]$. Then $\widetilde{H}_{*} ([a])=a(-)$; that is, the induced homomorphism acts as multiplication by $a$. Recall that $\mathcal{S}_{n}$ denotes the $n$-th symmetric group.  For each $\tau\in\mathcal{S}_{n}$, we have a map $      \underline{\tau}\colon   \emph{X}^{\wedge n}\rightarrow\emph{X}^{\wedge n}$ by permuting smash factors: $$x_{1}\wedge x_{2}\wedge \cdots \wedge x_{n}\stackrel{\underline{\tau}}\longmapsto x_{\tau^{-1}(1)}\wedge x_{\tau^{-1}(2)}\wedge \cdots\wedge  x_{\tau^{-1}(n)}.$$ \noindent We adopt the convention that $\s X^{\wedge n}=S^{1}\wedge X^{\wedge n}$ (but not $X^{\wedge n}\wedge S^{1}$). Employing the    group structure of $[\Sigma X^{\wedge n},\Sigma X^{\wedge n}]$, we  further obtain  maps $||\,\ell\, ||\colon   \Sigma\emph{X}^{\wedge n}\rightarrow\Sigma\emph{X}^{\wedge n}$ for all $\ell$  in the group ring $\mathbb{Z}_{(p)}[\mathcal{S}_{n}]$. More precisely, let $\{\tau_i\}_{i=1}^{n!}$ denote all elements of $\mathcal{S}_{n}$; notice the fact that these $\tau_{i}$ form a basis of  $\mathbb{Z}_{(p)}[\mathcal{S}_{n}]$; setting
	$\ell=k_{1}\tau_{1}+k_{2}\tau_{2}+\cdots+k_{n!}\tau_{n!}\in\z_{(p)}[\mathcal{S}_{n}]$ in  which $k_{\bullet}\in\z_{(p)}$, take $||\,\ell\, ||\in[\Sigma X^{\wedge n},\Sigma X^{\wedge n}]$ as

	\begin{equation}\label{el}
		||\,\ell\, ||=(\s \underline{\tau_{1}}\,)\hc [k_{1}]+(\s \underline{\tau_{2}}\,)\hc [k_{2}]+\cdots+(\s \underline{\tau_{n!}}\,)\hc [k_{n!}].
	\end{equation}

	\noindent This  yields a  representation:    $$\mathbb{Z}_{(p)}[\mathcal{S}_{n}]\rightarrow \mathrm{Hom}_{\Vect} (\widetilde{H}_{*} (\s X^{\wedge n} ;\mathbbm{k}),\widetilde{H}_{*} (\s X^{\wedge n} ;\mathbbm{k}) ),\quad\ell\mapsto ||\,\ell\, ||_{*}.$$That is to say, $ \widetilde{H}_{*} (\s X^{\wedge n} ;\mathbbm{k})$ turns out to be a left module over $\mathbb{Z}_{(p)}[\mathcal{S}_{n}]$. (The field $\mathbbm{k}=\z/p$ or $\mathbb{Q}$.) The map   $||\,\ell\, ||\in[\s X^{\wedge n},\s X^{\wedge n}] $ and the element $\ell\in \z_{(p)}[\mathcal{S}_{n}]$ are both denoted by $\ell$ if no confusion arises.

	Now we recall the  Dynkin-Specht-Wever element from representation theory and show some properties of it.

	Recall that $\mathcal{S}_{n}$ is the $n$-th symmetric group  whose permutations are composed from   right to left.  Denote by $\mathcal{S}_{n}^{op}$ the $n$-th symmetric group whose permutations are composed from left to right.\begin{pro}\label{btmpro}   Let $n\geq2$ and    $D_{n}=(12\cdots n)\in \mathcal{S}_{n},$ the cycle of length $n$. Hence $D_{n}\in \mathcal{S}_{n}^{op}.$ The following assertions hold.
		
		\begin{itemize}
			\item [\rm (1)]  \begin{itemize}
				\item[\rm (i)] 
				Define the Dynkin-Specht-Wever elements $\beta_{n}\in\z [\mathcal{S}_{n}]\xyd \z_{(p)} [\mathcal{S}_{n}]\xyd \mathbb{Q}[\mathcal{S}_{n}]$ as follows: $$\beta_{n}=(1-D_{n})(1-D_{n-1})\cdots (1-D_{2});$$
				\noindent
				that is, $\beta_{2}=1-D_{2},\;\beta_{n}=(1-D_{n})\beta_{n-1}=\beta_{n-1}-D_{n}\beta_{n-1}.$
				\noindent Then, $\beta_{n}\beta_{n}=n\beta_{n}.$  Let  $V$  be an ungraded $\mathbb{Q}$-vector space  and let $ \mathcal{S}_{n}$    act on $V^{\otimes n}$ from the left by permuting factors. Then the  action of $\beta_{n}$ on $V^{\otimes n}$  is given by $$\beta_{n} (a_{1}a_{2}\cdots a_{n})=[[\cdots[a_{1}, a_{2}], \cdots],  a_{n}], \;\;(a_{i}\in V),$$
				the ungraded $n$-fold left-normed Lie bracket.
				\item[\rm (ii)] Whether  regarded as elements of  $\mathbb{Q}[\mathcal{S}_{n}]$  or $\mathbb{Q}[\mathcal{S}_{n}^{op}]$, the following elements are  idempotent,
				$$\frac{1}{n} (1-D_{n})(1-D_{n-1})\cdots (1-D_{2}),\qquad\frac{1}{n} (1-D_{n}^{-1})(1-D_{n-1}^{-1})\cdots (1-D_{2}^{-1}),$$
				$$\frac{1}{n} (1-D_{2})(1-D_{3})\cdots (1-D_{n}),\qquad\frac{1}{n} (1-D_{2}^{-1})(1-D_{3}^{-1})\cdots (1-D_{n}^{-1}).$$\noindent Their 
				$n$-multiples lying in $\z[\mathcal{S}_{n}]$ or $\z[\mathcal{S}_{n}^{op}]$ are   called  Dynkin-Specht-Wever elements. (This is why some references on $\mathrm{A}^{\min}$-theory define their ``\,$\beta_n$'' to be $(1-D_n^{-1})(1-D_{n-1}^{-1})\cdots(1-D_2^{-1})\in \mathbb{Z}[\mathcal{S}_{n}]$.)
			\end{itemize} 
			\item [\rm (2)] Let $X\in\hCW.$ Take the maps $|| \beta_{n} ||' \in [\s X^{\wedge n}, \s X^{\wedge n}]$  as follows: $$ ||\beta_{2}||'=\mathrm{id}_{\s X^{\wedge2}}+\s \underline{D_{2}}\hc [-1],\;||\beta_{n}||'= ||\beta_{n-1}||'\wedge \mathrm{id}_{X}+ \big(( \s \,\underline{D_{n}}\, )  \hc (||\beta_{n-1}||'\wedge \mathrm{id}_{X})\big)\hc[-1].  $$
			Then, $|| \beta_{n} ||'_{*}=|| \beta_{n} ||_{*}\in\mathrm{Hom}_{\Vect} (\widetilde{H}_{*} (\s X^{\wedge n} ;\mathbbm{k}),\widetilde{H}_{*} (\s X^{\wedge n} ;\mathbbm{k}) ).$ 
			
			\item [\rm (3)] Write $V=\widetilde{H}_{*} ( X ;\mathbbm{k})$ and so 
			$\cg (V^{\otimes n})=\widetilde{H}_{*} (\s X^{\wedge n} ;\mathbbm{k})$. Suppose $a_{i}\in V$. Then the action of $\beta_{n}$ on homology  is given by\begin{align*}|| \beta_{n} ||_{*}&\colon \cg (V^{\otimes n})\longrightarrow \cg (V^{\otimes n}),\\ \cg (a_{1}a_{2}\cdots a_{n})&\mapsto\cg[[\cdots[a_{1}, a_{2}], \cdots],  a_{n}],\end{align*}
			\noindent $(\text{the}\;\cg\text{-image of the}\; n\text{-fold left-normed Lie bracket} ).$

			\item [\rm (4)] Use the notation of (3).  Let $\e=\mathrm{id}_{\s X}$  and $W_{n}=[[\cdots[\e, \e], \cdots],  \e]\in[\s X^{\wedge n},\s X]$, the left-normed Whitehead product. Further suppose that all elements in $V=\widetilde{H}_{*} (X;\mathbbm{k})$ are primitive. 
			Then, at the homology  level, $(\lo W_{n})_{*}\colon H_{*} (\lo\s X^{\wedge n};\mathbbm{k})\rightarrow H_{*} (\lo\s X;\mathbbm{k})$  is given by \begin{align*}(\lo W_{n})_{*}&\colon T (V^{\otimes n})\longrightarrow T (V),\\ a_{1}a_{2}\cdots a_{n}&\mapsto[[\cdots[a_{1}, a_{2}], \cdots],  a_{n}],\quad(a_{i}\in V).\end{align*}
			
			Accordingly, we write  $(\lo W_{n})_{*} (a_{1}a_{2}\cdots a_{n})=\beta_{n} (a_{1}a_{2}\cdots a_{n})$.
		\end{itemize}
	\end{pro}
	\begin{proof}\begin{itemize}
			\item [\rm (1)] These results are well-known in representation theory. They, or equivalent formulations of them, were proved independently by Dynkin \cite{Dynkin} 
			in 1947, Specht \cite{Specht}  in 1948 and Wever \cite{Wever}  in 1949. Also see Reutenauer's book \cite[Theorem 8.16, p.~195]{btn}.

			\item [\rm (2)] This follows from  (1) and the definition  of the notation $|| \beta_{n}||$; see Equation (\ref{el}).
			\item [\rm (3)] For convenience, we suppress the homology suspension isomorphism $\cg$ in the notation throughout this proof whenever no confusion can arise. Note that $\underline{D_{2}}=\underline{(12)}$, the switching map.   At the homology level, $\underline{D_{2}}_{*} (a_{1}a_{2})=(-1)^{|a_{1}||a_{2}|}a_{2}a_{1}.$ Hence, \[
			\beta_2(a_{1}a_{2}) =a_{1} a_{2} - (-1)^{|a_{1}||a_{2}|} a_{2}a_{1}=[a_{1},a_{2}].\]
			
			Assume that $\beta_{n-1} (a_{1}a_{2}\cdots a_{n-1})=[[\cdots[a_{1}, a_{2}], \cdots],  a_{n-1}].$ We show that $$\beta_{n} (a_{1}a_{2}\cdots a_{n})=[[\cdots[a_{1}, a_{2}], \cdots],  a_{n-1}] a_{n}-(-1)^{r}a_{n}[[\cdots[a_{1}, a_{2}], \cdots],  a_{n-1}],$$
			where $r=|[[\cdots[a_{1}, a_{2}], \cdots],  a_{n-1}]|\cdot |a_{n}|$; that is, we  need to show that $\beta_{n} (a_{1}a_{2}\cdots a_{n})=[[\cdots[a_{1}, a_{2}], \cdots],  a_{n}].$ By the assumption and the relation  $||\beta_{n}||'= ||\beta_{n-1}||'\wedge \mathrm{id}_{X}+ \big(( \s \,\underline{D_{n}}\, )  \hc (||\beta_{n-1}||'\wedge \mathrm{id}_{X})\big)\hc[-1]$ given in assertion (2), we have $$\beta_{n} (a_{1}a_{2}\cdots a_{n})=[[\cdots[a_{1}, a_{2}], \cdots],  a_{n-1}] a_{n}-\underline{D_{n}}_{*}\big([[\cdots[a_{1}, a_{2}], \cdots],  a_{n-1}] a_{n}\big).$$
			
			Thus we need only  show that \begin{equation}\label{gnjs}
				\underline{D_{n}}_{*}\big([[\cdots[a_{1}, a_{2}], \cdots],  a_{n-1}] \otimes a_{n}\big)=(-1)^{r}a_{n}\otimes[[\cdots[a_{1}, a_{2}], \cdots],  a_{n-1}].
			\end{equation}  Recall that $D_{n}=(12\cdots n)$. Therefore $\underline{D_{n}} (x_{1}\wedge x_{2}\wedge\cdots \wedge x_{n})=x_{n}\wedge x_{1}\wedge x_{2}\wedge\cdots \wedge x_{n-1},\;(x_{i}\in X).$ Then Equation (\ref{gnjs})  follows from the commutative diagram: \begin{equation}\label{gy123}
				\begin{tikzcd}
					X^{\wedge n}\arrow[r,"\underline{D_{n}}"] \arrow[d, equal] & X^{\wedge n} \arrow[d, equal] \\
					X^{\wedge (n-1)}\wedge X \arrow[r, "\tau"]                           & X\wedge X^{\wedge (n-1)}.\end{tikzcd}
			\end{equation}
			\noindent
			Here, $\tau$ denotes the switching map $(x_{1}\wedge x_{2}\wedge\cdots \wedge x_{n-1})\wedge x_{n} \mapsto x_{n}\wedge (x_{1}\wedge x_{2}\wedge\cdots \wedge x_{n-1})$.

			\item[\rm (4)] Essentially, we have proved this assertion in Lemma \ref{samelson}. Note that here we write $a\otimes b$ as $ab$ for short. For $n=2$, the result follows from Lemma \ref{samelson} (2).
			For $n\geq3$, apply Lemma \ref{samelson} (3) with the map $f_{1}=\x \cdots \x E, E\xx,\cdots , E\xx$ ($(n-1)$-fold) and the map $f_{2}=E$. It is not hard to obtain the result   by induction  on $n$. 
		\end{itemize}
	\end{proof}	
	
	In  view of (2) of the above proposition, following \cite{t2003} and \cite{BW}, the maps   $|| \beta_{n} ||',|| \beta_{n} ||\in[\s X^{\wedge n},\s X^{\wedge n}] $, the homomorphism $|| \beta_{n}  ||_{*}$ and the element $\beta_{n}\in \z_{(p)}[\mathcal{S}_{n}]$ are all denoted by $\beta_{n}$ if no confusion arises.

The following lemma supplements the $\z_{(p)}$-module description of $\mathrm{Q}_{n}^{\max}(V)$ in \cite[p.~439]{sw2006}.
    
	\begin{lem}\label{bcsw2006}
Let $n\geq2$ and let
$\rho\colon\z_{(p)}[\mathcal{S}_{n}]\to\frac{\z}{p}[\mathcal{S}_{n}]$ 
be the reduction.
Let $\beta_{n}\in\z[\mathcal{S}_{n}]
\subseteq\z_{(p)}[\mathcal{S}_{n}]$
be the DSW element defined in
Proposition~\ref{btmpro}~(1). Write $\rho(\beta_{n})$ as $\beta_{n}$ for short.
Suppose that $x\in\frac{\z}{p}[\mathcal{S}_{n}]$ and that 
$\beta_{n}x\in\frac{\z}{p}[\mathcal{S}_{n}]$ is idempotent.
Then there exists $X\in\z_{(p)}[\mathcal{S}_{n}]$ such that
$\rho(X)=x$ and $\beta_{n}X\in\z_{(p)}[\mathcal{S}_{n}]$ is idempotent.
\end{lem}
\begin{proof}
We use the following idempotent-lifting property of
$\z_{(p)}[\mathcal{S}_n]$: for every right ideal
$I\subseteq\z_{(p)}[\mathcal{S}_n]$ and every
$a\in\z_{(p)}[\mathcal{S}_n]$ with $a^2-a\in I$,
there exists $E\in\z_{(p)}[\mathcal{S}_n]$ satisfying
$E^2=E$ and $E-a\in I$.
This follows from \cite[Corollary~4.3]{Pavesic2010}
and \cite[Section~1]{KLN}.

We next prove that $\beta_n\z_{(p)}[\mathcal{S}_n]\cap
p\z_{(p)}[\mathcal{S}_n]
=p\beta_n\z_{(p)}[\mathcal{S}_n].$ It is enough to show that 
\begin{equation}\label{1012bb}\beta_n\z_{(p)}[\mathcal{S}_n]\cap
p\z_{(p)}[\mathcal{S}_n]
\xyd p\beta_n\z_{(p)}[\mathcal{S}_n].
\end{equation}
\noindent It is known that $\dim_{\mathbb{Q}}
\bigl(\beta_n\mathbb{Q}[\mathcal{S}_n]\bigr)
=
\dim_{\z/p}
\bigl(\beta_n\tfrac{\z}{p}[\mathcal{S}_n]\bigr)
=(n-1)!;$
see \cite[Section~3.5(a)]{BDEM}. Choose $b_1,\ldots,b_{(n-1)!}\in
\beta_n\z_{(p)}[\mathcal{S}_n]$ such that
$\rho(b_1),\ldots,\rho(b_{(n-1)!})$ form a
$\z/p$-basis of $\beta_n\frac{\z}{p}[\mathcal{S}_n]$.
The elements $\rho(b_1),\ldots,\rho(b_{(n-1)!})$ form
a $\z/p$-basis of the subspace
$\beta_n\frac{\z}{p}[\mathcal{S}_n]$.
Choose additional elements of $\frac{\z}{p}[\mathcal{S}_n]$
to extend this basis to a $\z/p$-basis of the entire
vector space $\frac{\z}{p}[\mathcal{S}_n]$.
Choose lifts $b_{(n-1)!+1},\ldots,b_{n!}
\in\z_{(p)}[\mathcal{S}_n]$ of these additional elements.
Then $\rho(b_1),\ldots,\rho(b_{n!})$ form a
$\z/p$-basis of $\frac{\z}{p}[\mathcal{S}_n]$. The coefficient matrix of $b_1,\ldots,b_{n!}$ with respect
to the permutation basis of $\z_{(p)}[\mathcal{S}_n]$
reduces to an invertible matrix over $\z/p$.
Its determinant is therefore a unit in $\z_{(p)}$.
Then, $b_1,\ldots,b_{n!}$ form a
$\z_{(p)}$-basis of $\z_{(p)}[\mathcal{S}_n]$. In particular, $b_1,\ldots,b_{(n-1)!}$ are linearly
independent over $\mathbb{Q}$.
The dimension formula above implies that $\beta_n\mathbb{Q}[\mathcal{S}_n]
=\mathrm{Span}_{\mathbb{Q}}
\{b_1,\ldots,b_{(n-1)!}\}.$
For any $u\in\beta_n\z_{(p)}[\mathcal{S}_n]$, write
$u=\sum_{i=1}^{n!}c_i b_i$ with $c_i\in\z_{(p)}$.
Since $u\in\beta_n\mathbb{Q}[\mathcal{S}_n]$ and
$b_1,\ldots,b_{n!}$ are also a $\mathbb{Q}$-basis of
$\mathbb{Q}[\mathcal{S}_n]$, we have
$c_i=0$ for $i>(n-1)!$.
Conversely, each $b_i$ with $1\leq i\leq(n-1)!$
belongs to $\beta_n\z_{(p)}[\mathcal{S}_n]$.
Thus $$\beta_n\z_{(p)}[\mathcal{S}_n]
=\mathrm{Span}_{\z_{(p)}}
\{b_1,\ldots,b_{(n-1)!}\}.$$ Moreover, because $b_1,\ldots,b_{n!}$ form a
$\z_{(p)}$-basis of $\z_{(p)}[\mathcal{S}_n]$, we see that $\sum_{i=1}^{(n-1)!}c_i b_i
\in p\z_{(p)}[\mathcal{S}_n]$
$\Longleftrightarrow$
$c_i\in p\z_{(p)}$ for all $1\leq i\leq(n-1)!$. This proves Relation (\ref{1012bb}).

Choose $X_0\in\z_{(p)}[\mathcal{S}_n]$ with $\rho(X_0)=x$.
The hypothesis gives $\rho\bigl((\beta_nX_0)^2-\beta_nX_0\bigr)
=(\beta_nx)^2-\beta_nx=0
\in\frac{\z}{p}[\mathcal{S}_n].$
Obviously,
$(\beta_nX_0)^2-\beta_nX_0
=\beta_n(X_0\beta_nX_0-X_0)$.
Thus, $$
(\beta_nX_0)^2-\beta_nX_0
\in\beta_n\z_{(p)}[\mathcal{S}_n]
\cap p\z_{(p)}[\mathcal{S}_n]
=p\beta_n\z_{(p)}[\mathcal{S}_n].$$
Applying the idempotent-lifting property to the right ideal
$p\beta_n\z_{(p)}[\mathcal{S}_n]$, we obtain
an idempotent $E\in\z_{(p)}[\mathcal{S}_n]$ such that
$E-\beta_nX_0=p\beta_nY\in \z_{(p)}[\mathcal{S}_n]$ for some
$Y\in\z_{(p)}[\mathcal{S}_n]$.
Then $X=X_0+pY\in\z_{(p)}[\mathcal{S}_n]$ satisfies
$\rho(X)=x$ in $\frac{\z}{p}[\mathcal{S}_n]$ and
$\beta_nX=E\in\z_{(p)}[\mathcal{S}_n]$.
So, $(\beta_nX)^2=\beta_nX\in\z_{(p)}[\mathcal{S}_n]$.
\end{proof}
	\subsection{Main results of Selick-Wu's $\mathrm{A}^{\mathrm{min}}$-theory}
	
We can now state the following proposition. Assertions (1)-(5) are given in \cite[Theorem~1.1, Lemma~2.1]{sw2006}, while assertion (6) follows from \cite[p.~66, p.~81]{sw2000} and the $\z_{(p)}$-module description of $\mathrm{Q}_{n}^{\max}(V)$ in \cite[p.~439]{sw2006}. Note that Lemma~\ref{bcsw2006} provides a supplementary explanation of the $\z_{(p)}$-module description of $\mathrm{Q}_{n}^{\max}(V)$ in \cite[p.~439]{sw2006}.
    
     Recall that $\mathbbm k=\mathbb{Q}$ or $\z/p.$ Note that for $f\colon X\to Y$, all $\widetilde{H}_{i}(f;\z_{(p)})$ are isomorphisms if and only if all  $\widetilde{H}_{i}(f;\z/p)$ and  all $\widetilde{H}_{i}(f;\mathbb Q)$ are isomorphisms.
	
	\begin{pro} \label{aminq} For the functor $\mathrm{A}^{\mathrm{min}}\in \Fun(\Vect,\Coalg),$ the functor $\mathrm{B}^{\mathrm{max}}\in \Fun(\Vect,\Hop)$, together with  functors $$\mathrm{Q}_{n}^{\mathrm{max}}\in\Fun(\Vect,\Vect)\;\, (n\geq2)$$ constructed by Selick-Wu,    the following assertions hold.
		
		\begin{itemize}
			\item[\rm(1)]  In $ \Fun(\Vect,\Coalg)$, $T\tg \mathrm{A}^{\mathrm{min}}\otimes \mathrm{B}^{\mathrm{max}}.$ Namely,  there exists a natural coalgebra decomposition $$T(V)\tg \mathrm{A}^{\mathrm{min}} (V)\otimes \mathrm{B}^{\mathrm{max}} (V).$$ (Here we omit the forgetful functor $\Hop\rightarrow\Coalg$; similar conventions apply below. Recall that by a Hopf algebra $T(V)$ we always mean the primitively generated tensor Hopf algebra.)

			\item[\rm(2)]
			In $\Hop$, $\mathrm{B}^{\mathrm{max}} (V)\tg T(\jia_{n\geq2}\mathrm{Q}_{n}^{\mathrm{max}} (V))$ naturally.

			\item[\rm(3)] $\mathrm{B}^{\mathrm{max}} (V)$ is a natural sub-Hopf algebra of $T(V)$.
			\item[\rm(4)] Regard $\mathrm{B}^{\mathrm{max}} (V)\xyd T(V)$ in view of assertion $(1)$. Denote by $\mathrm{B}_{n}^{\mathrm{max}} (V)$ the tensor-length-$n$ component of $\mathrm{B}^{\mathrm{max}} (V)\xyd T(V)$. Then, $\mathrm{B}_{1}^{\mathrm{max}}(V)=0$. Moreover,    $L_{n} (V)\xyd \mathrm{B}_{n}^{\mathrm{max}} (V)\xyd\mathrm{B}^{\mathrm{max}} (V) $ naturally if $\mathbbm k=\z/p$ and $n$ is not a power of $p$.

			\item[\rm(5)] Each $\mathrm{Q}_{n}^{\mathrm{max}}$  is a  sub-functor of $L_{n}$. And each
			$\mathrm{Q}_{n}^{\mathrm{max}}$ is a retract of the tensor product functor $(-)^{\otimes n}\colon  V\mapsto V^{\otimes n}$. 

\item[\rm(6)]
            The vector space $\mathrm{Q}_n^{\max}(V)$  is defined to be the tensor-length-$n$ component of the module of algebra indecomposables of $\mathrm{B}^{\max}(V)$; that is $$\mathrm{Q}_n^{\max}(V):=\left(\frac{I\mathrm{B}^{\max}(V)}{(I\mathrm{B}^{\max}(V))^2}\right)_n \cong \frac{\mathrm{B}_n^{\max}(V)}
{\left((I\mathrm{B}^{\max}(V))^2\right)_n},$$ where $I\mathrm B^{\max}(V)=\bigoplus_{n\ge1}\mathrm B_n^{\max}(V)$  is the augmentation ideal and the symbol $``\,(\;)_{n}$'' indicates the tensor-length-$n$ component. Equivalently, $\mathrm{Q}_n^{\max}(V)$ is defined to be $$\mathrm{colim}(V^{\otimes n}\xrightarrow{\beta_{n}\lambda_{n}}V^{\otimes n}\xrightarrow{\beta_{n}\lambda_{n}}V^{\otimes n}\xrightarrow{\beta_{n}\lambda_{n}}\cdots)=\mathrm{Im}(\beta_{n}\lambda_{n}\colon V^{\otimes n}\to V^{\otimes n})$$ for some  
            $\lambda_{n}\in\z_{(p)}[\mathcal{S}_{n}]\xyd\mathbb{Q}[\mathcal{S}_{n}]$ such that $\beta_{n}\lambda_{n}$ is idempotent  in $ \z_{(p)}[\mathcal{S}_{n}]$ and that $\mathrm{Im}(\beta_{n}\lambda_{n}\colon V^{\otimes n}\to V^{\otimes n})$ is isomorphic to $ \left(I\mathrm{B}^{\max}(V)/(I\mathrm{B}^{\max}(V))^2\right)_n$ naturally. \hfill\boxed{}
		\end{itemize}
	\end{pro}
	\noindent\textbf{Warning.}
In \cite{sw2000}, the representation
$\mathcal{R}\colon\mathbf{k}[\mathcal S_n]\to
\mathrm{Hom}(V^{\otimes n},V^{\otimes n})$
is determined by
$$\cg(v_1\otimes\cdots\otimes v_n)
=\mathrm{sig}\cdot v_{\cg(1)}\otimes\cdots\otimes v_{\cg(n)},\;\text{where}\; \cg\in\mathcal S_n, \;v_i\in V$$
and
$\mathrm{sig}=\mathrm{sig}(\cg;v_1\otimes\cdots\otimes v_n)
\in\{\pm1\}$ is the Koszul sign.
The homomorphism
$\mathcal{R}$ is a ring anti-homomorphism. Accordingly, every occurrence
of $\mathbf{k}[\mathcal S_n]$ in \cite{sw2000} involving compositions under $\mathcal{R}$, or its ungraded version, should be revised
as $\mathbf{k}[\mathcal S_n^{\mathrm{op}}]$.

	\begin{rem} In assertion (6) of the above proposition, the element $\lambda_n$ generally differs from the lift to $\z_{(p)}[\mathcal{S}_n]$ of the element $\lambda_n\in\frac{\z}{p}[\mathcal{S}_n]$ used by Selick-Wu in \cite[p.~81]{sw2000}, owing to our different convention for symmetric group representations; see Equation~(\ref{b1011}). Moreover, the element $\beta_n$ used in this paper generally differs from the one used in \cite{sw2000,sw2006}: ours yields left-normed Lie brackets, whereas theirs yields right-normed Lie brackets. Nevertheless, the corresponding results hold under our convention as well, since their proofs use natural linear maps $V^{\otimes n}\to V^{\otimes n}$, and the canonical representation
$\frac{\z}{p}[\mathcal S_n]\to \mathrm{Hom}_{\Fun(\Vecpt,\Vecpt)}((-)^{\otimes n},(-)^{\otimes n})$
is faithful under either convention.
	\end{rem}
	
	Selick and Wu also established a homotopy theoretic analogue of this proposition. It admits a \enquote{homotopy theoretic realization}  (or a \enquote{geometric realization}). Assertions (1), (2) and (4) of  the following proposition are given in 
	Selick-Wu \cite[Theorem~1.2]{sw2006}; for assertion (3), the first part follows  from assertion (2)  and its proof in \cite[p.~451]{sw2006}, and the second part is given in \cite[p.~440]{sw2006}.

	In order to indicate the homology suspension isomorphism when we are examining the homology,
	we denote the functor $\mathrm{Q}_{n}^{\mathrm{max}}\in\Fun(\hCW,\hCW)$ given in \cite{sw2006} by $\mathrm{SQ}_{n}^{\mathrm{max}}$; in this way,
	the homotopy theoretic functor
	$$\mathrm{SQ}_{n}^{\mathrm{max}}\in\Fun(\hCW,\hCW)$$
	corresponds better to the algebraic functor
	$$\mathrm{Q}_{n}^{\mathrm{max}}\in\Fun(\Vect,\Vect).$$
	It should be understood that the notation 
	$\mathrm{SQ}_{n}^{\mathrm{max}}$
	is to be treated as an indivisible symbol, and the letter ``S” in this context does not denote the suspension functor ``$\s$''.

	\begin{pro}\label{amindl}
		For the functors  $$\mathrm{A}^{\mathrm{min}}, \mathrm{B}^{\mathrm{max}},\mathrm{SQ}^{\mathrm{max}}_{n}\in\Fun(\hCW,\hCW),\;(n\geq2)$$   constructed by Selick-Wu,
		the following properties hold  for $X\in\hCW$.  The coefficient
	field  and the ground field should be understood to be $\mathbbm k=\mathbb{Q}$ or $\z/p.$ \begin{itemize} \item [\rm(1)] 
			
			In $\Fun(\hCW,\hCW)$, $\lo\s\tg\mathrm{A}^{\mathrm{min}}\times \mathrm{B}^{\mathrm{max}};$ more precisely,  there exists a natural fiber sequence $$\mathrm{A}^{\mathrm{min}} (X)
			\xrightarrow[]{j_{X}=*}
			\bigvee_{n\geq2}\mathrm{SQ}_{n}^{\mathrm{max}} (X)\stackrel{\pi_{X}}\longrightarrow\s X,$$ resulting in a  natural decomposition $$\lo \s X\simeq \mathrm{A}^{\mathrm{min}} (X)\times\mathrm{B}^{\mathrm{max}} (X),$$ where $\mathrm{B}^{\mathrm{max}}\in \Fun(\hCW,\hCW)$ is defined  to be  $\lo\bigvee_{n\geq2}\mathrm{SQ}_{n}^{\mathrm{max}}$. \footnote{In this paper, we denote the  symbol $\lo Q^{\max} (X)$ in Selick-Wu \cite{sw2006} by $\mathrm{B}^{\max} (X)$. The reader will find this notation more convenient.}

			\item [\rm(2)] The homology of  $\mathrm{A}^{\mathrm{min}} (X)$ is given as follows.       The coalgebra $H_{*} (\mathrm{A}^{\mathrm{min}} (X) )$ 
		 has the quotient filtration induced by the augmentation
ideal filtration of the coalgebra $H_{*}(\lo\s X)$,
             \footnote{Intuitively, the augmentation ideal
				filtration is a 
				filtration measuring the complexity of iterated comultiplications.} giving a natural coalgebra isomorphism  $$\mathrm{Gr}\,H_{*} (\mathrm{A}^{\mathrm{min}} (X) )\tg \mathrm{A}^{\mathrm{min}} (\widetilde{H}_{*} (X) ).$$ \noindent In addition, $\,H_{*} (\mathrm{A}^{\mathrm{min}} (\s X) )\tg \mathrm{A}^{\mathrm{min}} (\widetilde{H}_{*} (\s X) ).$
			\item [\rm(3)] The homology of  $\mathrm{B}^{\mathrm{max}} ( X)$ is given as follows. In $\Hop$, $H_{*} (\mathrm{B}^{\mathrm{max}} ( X))$ is naturally filtered   by the restriction of the Hopf algebra filtration of $H_{*} (\lo\s X)$ induced by the augmentation filtration of the coalgebra $H_{*} (\lo\s X)$, giving a natural Hopf algebra isomorphism, $$\mathrm{Gr}\;H_{*} (\mathrm{B}^{\mathrm{max}} ( X))\tg \mathrm{B}^{\mathrm{max}} ( \widetilde{H}_{*} (X)).$$
			\noindent In $\Hop$,   $\;H_{*} (\mathrm{B}^{\mathrm{max}} (\s X))\tg \mathrm{B}^{\mathrm{max}} ( \widetilde{H}_{*} (\s X)) \;\text{naturally}.$ The homology of  $\mathrm{SQ}_{n}^{\mathrm{max}} (X)$  is given as follows.   \[\widetilde{H}_{*} (\mathrm{SQ}_{n}^{\mathrm{max}} (X))\tg\cg  \mathrm{Q}_{n}^{\mathrm{max}} (\widetilde{H}_{*} (X)   
			) \;\text{naturally}.\]
			\noindent
			Moreover, up to the above isomorphism,  $\widetilde{H}_{*} (\mathrm{SQ}_{n}^{\mathrm{max}} (X))\xyd \cg L_{n} (\widetilde{H}_{*} (X)).$ (Recall that $\cg(-)$ denotes the  isomorphism of graded modules shifting degrees up by $1$, corresponding to the homology suspension isomorphism $\cg\colon\widetilde{H}_{\bullet} (-)\rightarrow\widetilde{H}_{\bullet+1} (\s-)$.)

			\item [\rm(4)] For each $n\geq2$, the space $\mathrm{SQ}_{n}^{\mathrm{max}} (X)$ is a natural retract of $\s X^{\wedge n}$.  Choose $\lambda_{n}\in\z_{(p)}[\mathcal S_{n}]$ as in Proposition \ref{aminq} (6), so that $\beta_{n}\lambda_{n}\in\z_{(p)}[\mathcal S_{n}]$ is idempotent. Then, for each $n\geq2$, 
			$$\mathrm{SQ}_{n}^{\mathrm{max}} (X)=\mathrm{hocolim}(\s X^{\wedge n}\xrightarrow[]{|| \beta_{n}\lambda_{n}||}\s X^{\wedge n}\xrightarrow[]{|| \beta_{n}\lambda_{n}||}\s X^{\wedge n}\to\cdots),$$
			the homotopy colimit of the infinitely repeated sequence. 
			\hfill\boxed{}
			
		\end{itemize}
	\end{pro} 
	
The following lemma will be used to explain the uniqueness of
$\mathrm{SQ}_{n}^{\max}$ up to natural equivalence.
\begin{lem}\label{a1009}  Let $\ell,\ell'\in\z_{(p)}[\mathcal{S}_n]$ be idempotent where $n\geq2$. Suppose that $$\mathrm{Im}(\ell\colon V^{\otimes n}\to V^{\otimes n})\cong\mathrm{Im}(\ell'\colon V^{\otimes n}\to V^{\otimes n})\;\;\text{naturally with respect to}\;V\in\Vecpt;$$  that is, the functor $\ell((-)^{\otimes n}) \colon V\in\Vecpt\mapsto \mathrm{Im}(\ell)=\ell(V^{\otimes n})\in\Vecpt$ and 
the
functor $\ell'((-)^{\otimes n})$ are isomorphic. Then $\ell=u\ell'u^{-1}\in\z_{(p)}[\mathcal{S}_n] $ for some unit $u\in \z_{(p)}[\mathcal{S}_n].$
    
\end{lem}

\begin{proof}
Our argument follows the multilinear construction in Selick and Wu
\cite[proof of Lemma~2.1]{sw2000}.

Let $\eta_V\colon\ell(V^{\otimes n})\to\ell'(V^{\otimes n})$
be the given natural isomorphism.
Take $V$ to be the  vector space with basis $a_1,\ldots,a_n$,
where $|a_i|=2$ for every $i$.
Let $M\subseteq V^{\otimes n}$ be the subspace spanned by the tensors
$a_{j_1}\otimes\cdots\otimes a_{j_n}$ with $\n j_{1},j_{2},\cdots,j_{n}\nn=\n1,2,\cdots ,n\nn$.
For each $i$, let $d_i\colon V\to V$ be the linear map determined by
$d_i(a_i)=0$ and $d_i(a_j)=a_j$ for $j\neq i$.
It is straightforward to verify that
$M=\bigcap_{i=1}^{n}\mathrm{Ker}(d_i^{\otimes n}\colon V^{\otimes n}\to V^{\otimes n})$. The naturality gives
$\eta_V\hc d_i^{\otimes n}=d_i^{\otimes n}\hc\eta_V$.
By injectivity of $\eta_V$, we infer that
$d_i^{\otimes n}\eta_V(z)=0$ if and only if $d_i^{\otimes n}z=0$.
Since $\eta_V$ is also surjective, it restricts to an isomorphism \begin{equation}\label{1014x}
    \chi\colon M\cap\ell(V^{\otimes n})\stackrel{\cong}\longrightarrow M\cap\ell'(V^{\otimes n})\;\;\text{of vector spaces.}
\end{equation}

Equip $V^{\otimes n}$ with the left
$\frac{\z}{p}[\mathcal S_n]$-module structure induced by
the action in Equation~(\ref{b1011}), which permutes the tensor factors.
Then $M$ is a left $\frac{\z}{p}[\mathcal S_n]$-submodule.
In addition,  equip $V^{\otimes n}$ with a right
$\frac{\z}{p}[\mathcal S_n]$-module structure induced by the right action on $V^{\otimes n}$ determined by
$$(a_{j_1}\otimes\cdots\otimes a_{j_n})^\cg
=a_{\cg^{-1}(j_1)}\otimes\cdots\otimes a_{\cg^{-1}(j_n)},\;\cg\in\mathcal S_n.$$
\noindent (It is necessary to point out that this right action
replies on the chosen basis $\n a_{i}\nn_{i=1}^n$ of $V$ but is not the usual  action permitting each tensor factors.) Obviously,
$$(\tau(a))^{\cg}=\tau(a^{\cg}),\;\;\forall\; a=a_{j_1}\otimes\cdots\otimes a_{j_n}\in V^{\otimes n},\;\;\forall\;\cg,\tau\in \mathcal S_n.$$
So the above left and right actions commute. Hence,
$(\ell z)^\cg=\ell(z^\cg)$ and
$(\ell' z)^\cg=\ell'(z^\cg)$
for every $z\in V^{\otimes n}$ and $\cg\in\mathcal S_n$.
Moreover, $M^\cg=M$ for every $\cg\in\mathcal S_n$.
Thus,
$(M\cap\ell(V^{\otimes n}))^\cg
=M\cap\ell(V^{\otimes n})$
and
$(M\cap\ell'(V^{\otimes n}))^\cg
=M\cap\ell'(V^{\otimes n})$
for every $\cg\in\mathcal S_n$. For each $\cg\in\mathcal S_n$, let $g\colon V\to V$
be the linear map defined by
$g(a_i)=a_{\cg^{-1}(i)},\,(\forall\;i)$. Then, 
$$g^{\otimes n}(a_{j_1}\otimes\cdots\otimes a_{j_n})
=a_{\cg^{-1}(j_1)}\otimes\cdots\otimes a_{\cg^{-1}(j_n)}
=(a_{j_1}\otimes\cdots\otimes a_{j_n})^\cg.$$
By linearity, $g^{\otimes n}(z)=z^\cg$
for every $z\in V^{\otimes n}$. That is, any map $(-)^{\cg}\colon z \in V^{\otimes n}\mapsto z^{\cg}\in V^{\otimes n}$ is of the form $g^{\otimes n}$. So, the naturality  gives
$\eta_V(z^\cg)=(\eta_V(z))^\cg$
for every $z\in\ell(V^{\otimes n})$.
Restricting $\ca_{V}$ to $M\cap\ell(V^{\otimes n})$, we obtain
$\chi(z^\cg)=(\chi(z))^\cg$.
Thus $\chi$ given in Equation (\ref{1014x}) is an isomorphism of  right
$\frac{\z}{p}[\mathcal S_n]$-modules.

Regard  $\frac{\z}{p}[\mathcal S_n]$ as a graded module
 concentrated in degree $2n$.
The linear map
$$\Psi\colon\frac{\z}{p}[\mathcal S_n]\to M,
\;\; \cg\in\mathcal S_n\mapsto a_{\cg^{-1}(1)}\otimes\cdots\otimes a_{\cg^{-1}(n)}$$
is a bijection. It is not hard to see that $\Psi$ is a homomorphism of left $\frac{\z}{p}[\mathcal S_n]$-modules and a 
homomorphism of right $\frac{\z}{p}[\mathcal S_n]$-modules.
Hence $\Psi$ is an isomorphism of graded bimodules.
It restricts to an isomorphism
$\ell(\frac{\z}{p}[\mathcal S_n])\cong\ell(M)$
of right $\frac{\z}{p}[\mathcal S_n]$-modules. Similarly, $\ell'(\frac{\z}{p}[\mathcal S_n])\cong\ell'(M)$. 
Notice that $\ell(M)=M\cap\ell(V^{\otimes n}),\;\ell'(M)=M\cap\ell'(V^{\otimes n});$ (use $\ell(\ell (z))=\ell (z)$). Combining these  with $\chi$ given in Equation (\ref{1014x}), we obtain
\begin{equation}\label{iso1}
\ell\frac{\z}{p}[\mathcal S_n]
\cong
\ell'\frac{\z}{p}[\mathcal S_n]
\end{equation}
as right $\frac{\z}{p}[\mathcal S_n]$-modules. Let $\bar\ell,\bar\ell'\in\frac{\z}{p}[\mathcal S_n]$
be the reductions of $\ell,\ell'$.
Consider the decompositions
$$\bar\ell\,\frac{\z}{p}[\mathcal S_n]
\oplus(1-\bar\ell)\frac{\z}{p}[\mathcal S_n]=\frac{\z}{p}[\mathcal S_n]
=\bar\ell'\,\frac{\z}{p}[\mathcal S_n]
\oplus(1-\bar\ell')\frac{\z}{p}[\mathcal S_n].$$
Then, the Krull--Schmidt cancellation theorem  tells us that\begin{equation}\label{iso2}(1-\bar\ell)\frac{\z}{p}[\mathcal S_n]
\cong(1-\bar\ell')\frac{\z}{p}[\mathcal S_n]
\end{equation}
as right $\frac{\z}{p}[\mathcal S_n]$-modules;
see \cite[Section~3.8]{Etingof2011} for the Krull-Schmidt cancellation theorem.
 A well-known result in representation theory shows that the two module
isomorphisms given in Equations (\ref{iso1}), (\ref{iso2}) imply
$\bar\ell=v\bar\ell'v^{-1}$
for some unit $v\in\frac{\z}{p}[\mathcal S_n]$. See  \cite[Lemma~2.2(2)]{kln2018} for this well-known result.

Choose a lift $\widetilde v\in\z_{(p)}[\mathcal S_n]$
of $v$ and take
$u=\ell\widetilde v\ell'
+(1-\ell)\widetilde v(1-\ell')$.
Idempotence gives $\ell u=u\ell'$.
Writing $\bar u$ for the reduction of $u$, the equality
$\bar\ell v=v\bar\ell'$ gives
$\bar u=v\bar\ell'+v(1-\bar\ell')=v$.
Left multiplication by $u$ is an endomorphism of the finite free
$\z_{(p)}$-module $\z_{(p)}[\mathcal S_n]$.
Its reduction modulo $p$ is invertible, so its determinant
is a unit in $\z_{(p)}$ and the map $u\cdot(-)$ is bijective.
Choose $w\in\z_{(p)}[\mathcal S_n]$ with $uw=1$.
Then $u(wu-1)=0$, and injectivity gives $wu=1$.
Thus $u$ is a unit, and $\ell u=u\ell'$ yields
$\ell=u\ell'u^{-1}$.
\end{proof}

The following lemma is used to examine the spaces $\mathrm{SQ}_2^{\max}(X)$ and $\mathrm{SQ}_3^{\max}(X)$.
\begin{lem}\label{xzdzrd} \begin{itemize}
    \item [\rm (1)] Suppose that $p\neq 3$. Then $L_3\cong\mathrm{Q}_3^{\max}$ in $\Fun(\Vecpt,\Vecpt)$.
     \item [\rm (2)]Suppose that $p$ is an odd prime. Then $L_2=\mathrm{Q}_2^{\max}$ in $\Fun(\Vecpt,\Vecpt)$.
\end{itemize}

\end{lem}

\begin{proof}\begin{itemize}
    \item [\rm (1)] First,  Proposition \ref{aminq} (4) gives $L_3(V)\subseteq\mathrm{B}_3^{\max}(V)$ and $\mathrm{B}_1^{\max}(V)=0$.  So, $$((I\mathrm{B}^{\max}(V))^2)_3=\mathrm{B}_1^{\max}(V)\mathrm{B}_2^{\max}(V)+\mathrm{B}_2^{\max}(V)\mathrm{B}_1^{\max}(V)=0.$$  We know $\mathrm{Q}_3^{\max}(V)=\mathrm{B}_3^{\max}(V)/((I\mathrm{B}^{\max}(V))^2)_3$. Hence $\Ker(\mathrm{B}_3^{\max}(V)\rightarrow\mathrm{Q}_3^{\max}(V))=((I\mathrm{B}^{\max}(V))^2)_3=0.$ Thus the quotient map $\mathrm{B}_3^{\max}(V)\rightarrow\mathrm{Q}_3^{\max}(V)$ is injective, and so is its restriction to $L_3(V)$.
Apply \cite[proof of Theorem 8.8, p.~64, Statement (1)]{sw2000} with $\mathrm{B}=\mathrm{B}^{\max}$ and $n=3$. In this case, $\mathrm{Q}_3(V)=\mathrm{Q}_3^{\max}(V)$ by \cite[p.~66, definition preceding Corollary 8.9]{sw2000}. That statement supplies a natural section of $L_3(V)\cap\mathrm{B}^{\max}(V)\rightarrow\mathrm{Q}_3^{\max}(V)$. Since $L_3(V)\cap\mathrm{B}^{\max}(V)=L_3(V)$, the restriction above is also surjective. It is therefore a natural isomorphism.    \item [\rm (2)] Proposition \ref{aminq} (4), (5) imply  $\mathrm{Q}^{\max}_{2}\cong \mathrm{B}^{\max}_{2}\dyd L_{2}\neq0$. Now, use Proposition \ref{aminq}  (6). It suffices to regard $\beta_2=1-(12)$ and $\lambda_2$ as elements of
$\mathbb{Z}/p[\mathcal{S}_2]$ via the reduction homomorphism
$\mathbb{Z}_{(p)}\to\mathbb{Z}/p$.
Note that every element of the form $(1-(12))x$, with
$x\in\mathbb{Z}/p[\mathcal{S}_2]$, is a $\z/p$-scalar multiple of $1-(12)$.
The unique nonzero idempotent of this form
is $\frac{1}{2}(1-(12))=\frac{1}{2}\beta_{2}$. This forces $\mathrm{Q}^{\max}_{2}(V)=\mathrm{Im}(\beta_{2}/2\colon V^{\otimes 2}\to V^{\otimes 2})=L_{2}(V).$
\end{itemize}
\end{proof}

\begin{rem}\label{a1002}  The functor
$\mathrm{SQ}_{n}^{\max}$
is independent of the chosen realizing idempotent in
\(\mathbb Z_{(p)}[\mathcal S_n]\) in the following sense.
Assume that there is a natural isomorphism in $\Vecpt$,
$\mathrm Q_n^{\max}(V)=\mathrm{Im}(a_{n}\colon V^{\otimes n}\to V^{\otimes n} )\cong  \mathrm{Im}(a_{n}'\colon V^{\otimes n}\to V^{\otimes n} )$. Here, elements $a_{n}, a'_{n}\in\mathbb Z_{(p)}[\mathcal S_n]$
are idempotent  and $a_{n}=\beta_{n}\lambda_{n}$ is the chosen realizing idempotent in the second definition of $\mathrm Q_n^{\max}(V)$ given in Proposition \ref{aminq} (6). For each $\ell\in \z_{(p)}[\mathcal{S}_{n}]$, let  \begin{equation}\label{hocoliml}\mathop{\mathrm{hocolim}}\limits_{\ell} \s X^{\wedge n}=\mathrm{hocolim}(\s X^{\wedge n}\xrightarrow[]{|| \ell||}\s X^{\wedge n}\xrightarrow[]{|| \ell||}\s X^{\wedge n}\to\cdots).
\end{equation}
 Write $\mathrm{SQ}_{n,a_{n}}^{\max}(X)= \mathrm{hocolim}_{a_{n} } \s X^{\wedge n}$ and $\mathrm{SQ}_{n,a_{n}'}^{\max}(X)= \mathrm{hocolim}_{a'_{n} } \s X^{\wedge n}$.  Then, Lemma \ref{a1009} and the Whitehead theorem imply that $\mathrm{SQ}_{n,a_{n}}^{\max}(X)\simeq \mathrm{SQ}_{n,a'_{n}}^{\max}(X)$ naturally with respect to $X$.\end{rem}

    \begin{lem}\label{b1009} Follow the notation in Equation (\ref{hocoliml}). If $p\neq3$, then the space $\mathrm{SQ}_{3}^{\mathrm{max}} (X)$ can be taken as $\mathop{\mathrm{hocolim}}\limits_{\beta_{3}/3}\,\s X^{ \wedge 3}$.
	\end{lem}
	\begin{proof} Note that in this case $\mathrm{Q}_{3}^{\max}(V)\cong L_{3}(V)=\mathrm{Im}(\beta_{3}/3\colon V^{\otimes3}\to V^{\otimes3}).$   
In view of Proposition \ref{btmpro} (3), Lemma \ref{xzdzrd} 
    and the above remark, we infer the result.
	\end{proof}
	
	We point out the following fundamental facts in homological algebra, which follow immediately from the definition of the colimit: let $A$ be a $\z/p$-vector space and $e\colon A\rightarrow A$ be an idempotent homomorphism; then the colimit of the infinitely repeated sequence satisfies \begin{equation}\label{img}
		\mathrm{colim} (A\xrightarrow{e} A\xrightarrow{e}  A \rightarrow\cdots)=\mathrm{Im} (e\colon A\rightarrow A).
	\end{equation}

	\noindent Moreover, denoting by $\rho_{1}\colon A_{1}=A\rightarrow \mathrm{colim} (A\xrightarrow{e} A\xrightarrow{e}  A \rightarrow\cdots)$ the canonical homomorphism from the first $A$  in the sequence to the colimit, we have \begin{equation}\label{ps2}
		\rho_{1}(x)=e(x),\;\forall\,x\in A_{1}=A.
	\end{equation}

	We make the following  remark for Selick-Wu \cite[p.~441] {sw2006}  and  \cite[Lemma 3.11, p.~448]{sw2006}, which discuss the construction of the map $\bigvee_{n\geq2} \mathrm{SQ}_{n}^{\max} (X)\rightarrow \s X$.
	
	\begin{rem} \label{wddrmk}
		\begin{itemize}
			\item [\rm (1)] Suppose that $X\in\hCW$.
			Let $\theta_n \colon \mathrm{SQ}^{\max}_{n} (X) \to \Sigma X^{\wedge n}$ be the composition,\begin{equation}\label{thetady}
				\mathrm{SQ}^{\max}_{n} (X)\hookrightarrow \mathrm{SQ}^{\max}_{n} (X)\vee \mathop{\mathrm{hocolim}}\limits_{1-\beta_{n}\lambda_{n}} (\s X^{\wedge n})\xrightarrow[]{\simeq} \s X^{\wedge n}.
			\end{equation} 
			\noindent  Recall that $W_{n}$  is the iterated Whitehead product $[[\cdots[\mathrm{id}, \mathrm{id}], \cdots],  \mathrm{id}]\in[\s X^{\wedge n},\s X]$.
			Let $\phi \colon \bigvee_{n=2}^{\infty} \mathrm{SQ}^{\max}_{n} (X) \to \Sigma X$ be the  map such that $\phi|_{\mathrm{SQ}^{\max}_{n} ( X)}$ is the composition $\mathrm{SQ}^{\max}_{n} (X) \xrightarrow{\theta_n} \Sigma X^{\wedge n} \xrightarrow{\lambda_{n}} \Sigma X^{\wedge n} \xrightarrow{W_n} \Sigma X.$ Then, $\kj$ is the map $\bigvee_{n=2}^{\infty} \mathrm{SQ}^{\max}_{n} (X) \to \Sigma X$ in the fiber sequence of Proposition \ref{amindl} (1).\footnote{This construction is given in Selick and Wu \cite[p.~441]{sw2006}.  In this paper, we denote their  $\lo Q^{\max} (X)$ by $\mathrm{B}^{\max} (X)$. Selick and Wu write their $\lambda_{n}$ in  \cite{sw2000} as $\af_{n}^{\max}$ in \cite[p.~441]{sw2006}}  In particular, $\lambda_{3}$ is taken as $1/3$ when $p\neq3.$
			\item [\rm (2)] Use the notation of (1). Additionally assume that all elements in $\widetilde{H}_{*} (X)$ are primitive, $p\neq3$, and that $\mathrm{SQ}^{\max}_{3} (X)\simeq\s^{r+1}Y$ for some $r\geq0$ and $Y\in\hCW$. As usual, write $V=\widetilde{H}_{*} (X).$ Then the looped restriction $\ck=\lo (\phi|_{\mathrm{SQ}^{\max}_{3} (X)})$ includes $H_{*} (\lo \mathrm{SQ}^{\max}_{3} (X);\mathbbm{k})$ in $H_{*} (\lo\s X)$ by the commutative diagram \begin{equation}
				\begin{tikzcd}\label{jsys}
					H_{*} (\lo \mathrm{SQ}^{\max}_{3} (X)) \arrow[r,"\ck_{*}\;"] \arrow[d,equal] 
					& H_{*} (\lo\s X) \arrow[d, equal] \\
					T(L_{3}  (V)) \arrow[r,"\xyd"] 
					& T(V).
				\end{tikzcd} 
			\end{equation}
			\noindent 
			Recall that 
			$L_{n} (V)\xyd L(V)\cap V^{\otimes n}\xyd T(V)$ is spanned by all  
			$n$-fold  Lie brackets  with entries in $V$; and hence  \begin{equation}\label{tdgs}
				\widetilde{H}_{*} (\mathrm{SQ}^{\max}_{3} ( X))=\mathrm{Im} (||\beta_{3}/3||_{*})=\mathrm{Im} (||\beta_{3}||_{*})=\cg L_{3} (V).
			\end{equation} \noindent
			The commutativity of the above diagram follows from the following commutative diagram, in which $\rho_{1}$ denotes the canonical map from the first $\s X^{\wedge 3}$  to the homotopy colimit, and $\mathrm{Tran}$ denotes the transgression in the homology Leray-Serre spectral sequence with respect to the path fibration,
\[
\begin{tikzcd}[column sep=1.5cm, row sep=1cm]
    H_j(\Omega \Sigma X^{\wedge 3})
    \arrow[r, "(\Omega \rho_1)_*"]
    & H_j(\Omega \mathrm{SQ}^{\max}_3(X))
    \arrow[r, "(\Omega \theta_3)_{*}"]
    & H_j(\Omega \Sigma X^{\wedge 3})
    \arrow[r, "\left(\lo\frac{W_3}{3}\right)_*"]
    & H_j(\Omega \Sigma X). \\
    H_{j+1}(\Sigma X^{\wedge 3})
    \arrow[u, hook, "\mathrm{Tran}"]
    \arrow[r, "\rho_{1*}"]
    & H_{j+1}(\mathrm{SQ}^{\max}_3(X))
    \arrow[u, hook, "\mathrm{Tran}"]
    \arrow[r, "\theta_{3*}"]
    & H_{j+1}(\Sigma X^{\wedge 3})
    \arrow[u, hook, "\mathrm{Tran}"] &
\end{tikzcd}
\]

Note that $\theta_{3*}(a)=a$ ($\forall \,a$) by the definition and $\rho_{1*}(b)=||\beta_{3}/3||_{*}(b)$ ($\forall \,b$) by the fact given in Equation (\ref{ps2}); moreover, for any elements $a_1,a_2,a_3\in V=\widetilde{H}_{*}(X)$,
$$
(\lo\frac{W_3}{3})_*([[a_1,a_2],a_3])
=(\lo\frac{W_3}{3})_*(\beta_3(a_1a_2a_3))
=\frac{\beta_3}{3}(\beta_3(a_1a_2a_3))
=[[a_1,a_2],a_3],
$$
\noindent
by Proposition \ref{btmpro} (4). Here, the graded actions of $\frac{\beta_3}{3}$ and $\beta_3$ are still denoted by $\frac{\beta_3}{3}$ and $\beta_3$, respectively.
		\end{itemize}
	\end{rem}

	The functors 
	$\mathrm{SQ}_{n}^{\mathrm{max}}$, $\mathrm{Q}_{n}^{\mathrm{max}}$ and $\mathrm{A}^{\mathrm{min}}$
	are constructed abstractly and  not easily accessible in general. In some specific situations, explicit descriptions have been obtained in \cite{sw2000}. It is natural to ask how these functors act on spheres. The following lemma provides the answer.
	
	\begin{lem}\label{bmx}\begin{itemize}
			\item[\rm (1)] Localize spaces at $p=2$. Then, $\mathrm{A}^{\mathrm{min}} (S^{n})\simeq\lo S^{n+1}$, $\mathrm{B}^{\mathrm{max}} (S^{n})\simeq  *$ for all $n\geq1.$ 
			\item[\rm (2)] Localize spaces at an odd prime $p$. Then for all $i\ge 1$,
			\[
			\mathrm{A}^{\mathrm{min}} (S^{2i})\simeq \Omega S^{2i+1}, \quad
			\mathrm{B}^{\mathrm{max}} (S^{2i})\simeq *,
			\]
			\[
			\mathrm{A}^{\mathrm{min}} (S^{2i-1})\simeq S^{2i-1} \quad \text{and}\quad
			\mathrm{B}^{\mathrm{max}} (S^{2i-1})\simeq \Omega S^{4i-1}.
			\]
			
			The second row gives  Serre's splitting at odd primes, $\lo S^{2i}\simeq S^{2i-1}\times \lo S^{4i-1}$.
		\end{itemize}
		
	\end{lem}
	\begin{proof} Notice that  $H_{*} (\lo\s S^{n})\tg T(x)$ as Hopf algebras where $|x|=n$. Let  $V=\widetilde{H}_{*} (S^{n})=\mathrm{Span}_{\z/p}\n x\nn.$
		\begin{itemize}
			\item[\rm(1)]   Recall the properties given by Proposition~\ref{aminq} (2) and (5). In this case, $L_{j} (V)=0$ for all $j\geq2$ and hence $$0=L_{j} (V)\dyd \mathrm{Q}_{j}^{\mathrm{max}} (V),$$ leading to $\mathrm{Q}_{j}^{\mathrm{max}} (V)=0$ for all $j\geq2$. Then $\mathrm{SQ}_{j}^{\mathrm{max}} (S^{n})\simeq*$ for all $j\geq2.$
			The remaining part follows from Proposition~\ref{amindl} (1).
			\item[\rm(2)] For the even case, the proof is similar to that of part (1).   We consider  the odd case, setting $n=2i-1$.
			Clearly, in $T(x)$, $$[x,x]=2x^{2},\;[[x,x],x]=[2x^{2},x]=0.$$

			\noindent Thus, $L_{2} (V)=\mathrm{Span}_{\z/p}\n 2 x^{2}\nn,\; (|2 x^{2}|=2n=4i-2)$ and $L_{j} (V)=0$, ($j\geq3$).  Lemma \ref{xzdzrd}  (2) gives $L_{2}=\mathrm{Q}_{2}^{\max}$.
           
            Then we have $$\widetilde{H}_{*} (\mathrm{SQ}_{2}^{\max} (S^{2i-1}))=\cg L_{2} (V)=\mathrm{Span}_{\z/p}\n 2\cg x^{2}\nn,\;\;|2\cg x^{2}|=4i-1,$$$$\widetilde{H}_{*} (\mathrm{SQ}_{j}^{\max} (S^{2i-1}))\xyd \cg L_{j} (V)=0,\;\;j\geq3.$$
			\noindent Then, $\mathrm{SQ}_{2}^{\max} (S^{2i-1})\simeq S^{4i-1}$, $\mathrm{SQ}_{j}^{\max} (S^{2i-1})\simeq *$ for all $j\geq3.$ So, $\mathrm{B}^{\max} (S^{2i-1})=\lo\bigvee_{n\geq2}\mathrm{SQ}_{n}^{\max} (S^{2i-1})\simeq\lo S^{4i-1}.$ It follows that $$\lo\s S^{2i-1}\simeq \mathrm{A}^{\min} (S^{2i-1})\times\lo S^{4i-1} \,\;\text{for all}\;\,i\geq1.$$
			
			\noindent Applying $H_{*} (-;\z_{(p)})$, we have $H_{2i-1} (\mathrm{A}^{\min} (S^{2i-1});\z_{(p)})\tg\z_{(p)}$ and  $\widetilde{H}_{\ell} (\mathrm{A}^{\min} (S^{2i-1});\z_{(p)})=0$
			for all $\ell\neq 2i-1$.
			\noindent Applying $\pi_{*} (-)$, we have $\pi_{2i-1} (\mathrm{A}^{\min} (S^{2i-1}))\tg\z_{(p)}$ and $\pi_{\ell} (\mathrm{A}^{\min} (S^{2i-1}))=0$ for all $\ell<2i-1$. Hence,  we obtain $$\mathrm{A}^{\min} (S^{2i-1})\simeq S^{2i-1}\,\;\text{for all}\;\,i\geq2.$$ In the case $i=1$,  $\mathrm{A}^{\min} (S^{2i-1})=\mathrm{A}^{\min} (S^{1})$ is not simply-connected. Nevertheless, we  may infer that $\mathrm{A}^{\min} (S^{1})\simeq S^{1}$
            by using the Hopf fibration $S^{1}\rightarrow S^{3}\rightarrow S^{2}$.\end{itemize}
		
	\end{proof}

	For the odd case of part (2) of the above lemma, the result is also given in the proof of  \cite[Proposition~3.2]{hnsf}. Our methods provide an improvement over theirs.

	\section{Applications of Selick and Wu's $\mathrm{A}^{\min}$-theory to 2-cell complexes}\label{secamin2}

	Recall that $C_{f}=S^{n}\cup_{f} e^{n+k+1}$ ($n\geq2,\;k\geq0$) and $F$ is the homotopy fiber of the pinch map $\s C_{f}\xrightarrow[pinch]{q}S^{n+k+2}$.

  Throughout this section, when $p\neq2$, we assume that
$n\equiv k+1\equiv0\m{2}$, so that both $|x|$ and $|y|$ are even.

	Applying the  functorial decomposition to the pinch map $\s C_{f}\xrightarrow[pinch]{q}S^{n+k+2}$, we deduce decompositions relative to $\lo\bigvee_{m\geq2}\mathrm{SQ}_{m}^{\max} (C_{f})$.  Essentially, the following lemma has  been proved by  Wu \cite[Proposition 4.16, p.~50]{wu2003}.
	
	\begin{lem0}\label{fjdlyw}  Localize spaces at  any prime $p$, including 3. 
		Then, there exists a space $\widetilde{\mathrm{A}}$ such that  \[\lo F\simeq \widetilde{\mathrm{A}}\times \lo\bigvee_{m\geq2}\mathrm{SQ}_{m}^{\max} (C_{f}).\] \noindent In addition, there exists a commutative diagram whose rows are  split  fiber sequences. The splitting of the second row  is given by Proposition~\ref{amindl} (1). \begin{equation}\label{jhtb}
			\begin{tikzcd}
				\lo\bigvee_{m\geq2}\mathrm{SQ}_{m}^{\max} (C_{f}) \arrow[r] \arrow[d, equal] & \lo F \arrow[r] \arrow[d,"\lo i"] & \widetilde{\mathrm{A}} \arrow[d] \\
				\lo\bigvee_{m\geq2}\mathrm{SQ}_{m}^{\max} (C_{f}) \arrow[r,"\lo\pi_{C_{f}}"]                  & \lo\s C_{f} \arrow[r]           & \mathrm{A}^{\mathrm{min}} (C_{f})
			\end{tikzcd}
		\end{equation} \hfill$\boxed{}$
	\end{lem0}

	Localize spaces at a prime $p$. Recall that the nontrivial CW decomposition
	$C_{f}=S^{n}\cup e^{n+k+1}$ implies that  we may set $V=\widetilde{H}_{*} (C_{f})=\mathrm{Span}_{\z/p}\n x,y\nn$,
	where $|x|=n,\;  |y|=n+k+1$.
	Hereafter, we additionally assume that $H^{*} (C_{f})$ is a trivial algebra, ensuring that  $$H_{*} (\lo\s C_{f})\tg T(x,y)\;\;(|x|=n,\;  |y|=n+k+1)\;\;\; \text{as Hopf algebras}.$$ We use the identification $H_{*} (\lo\s C_{f})=T(x,y).$

	\begin{lem0}\label{spn} Suppose that the prime $p\neq3$. Then there exists an isomorphism
		$$\widetilde{H}_{*} (\mathrm{SQ}_{3}^{\mathrm{max}} (C_{f}))\tg\mathrm{Span}_{\z/p}\n \cg[[x,y],y],\;\cg[[x,y],x]\nn.$$
	\end{lem0}

	\begin{proof}  Recall that $\mathrm{SQ}_{3}^{\mathrm{max}} (X)=\mathop{\mathrm{hocolim}}\limits_{\beta_{3}/3}\,\s X^{ \wedge 3}$. Now, take $X=C_{f}$. According to Proposition \ref{btmpro} (3) and fact (\ref{img}), we have 
		$\widetilde{H}_{*} (\mathop{\mathrm{hocolim}}\limits_{\beta_{3}/3}\;\s X^{\wedge 3})\tg \mathop{\mathrm{colim}}\limits_{||\beta_{3}/3||_{*}}\;\widetilde{H}_{*} (\s X^{\wedge 3})=\mathop{\mathrm{colim}}\limits_{||\beta_{3}/3||_{*}}\;\cg (V^{\otimes 3})=\mathrm{Im}||\beta_{3}/3||_{*}=\mathrm{Span}_{\z/p}\n \cg[[x,y],y],\;\cg[[x,y],x]\nn.$
	\end{proof}

	The following lemma is motivated by \cite[Proposition 3.1]{CW2013}.
	\begin{lem0}\label{q3max} Suppose that the prime $p\neq3$. Then  $\mathrm{SQ}_{3}^{\mathrm{max}} (C_{f})\simeq\s^{2n+k+2}C_{f}.$\end{lem0}\begin{proof}  Write $X=C_{f}=S^{n}\cup e^{n+k+1}$ and so $\widetilde{H}_{*} (X)=\mathrm{Span}_{\z/p}\n x,y\nn$,
		where $|x|=n,\;  |y|=n+k+1$.  Then,
		according to Proposition \ref{tl}, 
		there is a map $g: S^{|x|}\wedge S^{|y|}\rightarrow X\wedge X$ such that $g_{*} (\e_{|x|+|y|})=xy+tyx\in H_{|x|+|y|} (X\wedge X)$ for some  integer $t\in [1,p-1].$ Recall that $\mathrm{SQ}_{3}^{\mathrm{max}} (X)=\mathrm{hocolim(}\s X^{ \wedge 3}\stackrel{\beta_{3}/3\;}\longrightarrow\s X^{ \wedge 3}\stackrel{\beta_{3}/3\;}\longrightarrow\cdots).$ Let $\kj_{2}$ be the canonical map from the first $\s X^{ \wedge 3}$ in the sequence to $\mathrm{SQ}_{3}^{\mathrm{max}} (X)$. 
		Consider the composition $\Psi$, $$ \Psi\colon(\emph{S}^{1}\wedge X)\wedge (S^{|x|}\wedge S^{|y|})\xrightarrow[]{\mathrm{id}\wedge g} (S^{1}\wedge X)\wedge (X\wedge X)\xrightarrow[]{\beta_{3}/3} S^{1}\wedge X\wedge X\wedge X\xrightarrow[]{\kj_{2}}\mathrm{SQ}_{3}^{\mathrm{max}} (X). $$
		Applying $\widetilde{H}_{*} (-),$ we get 
		$$ \Psi_{*}\colon\widetilde{H}_{*} (S^{1}\wedge X)\otimes \widetilde{H}_{*}  (S^{|x|}\wedge S^{|y|})\xrightarrow[]{\mathrm{id}\otimes g_{*}} \widetilde{H}_{*} (S^{1}\wedge X)\otimes \widetilde{H}_{*} (X\wedge X)\xrightarrow[]{\beta_{3}/3} \widetilde{H}_{*} (S^{1}\wedge X\wedge X\wedge X)\xrightarrow[]{\kj_{2*}}\widetilde{H}_{*} (\mathrm{SQ}_{3}^{\mathrm{max}} (X)). $$
		
		\noindent Suppose $a\in \widetilde{H}_{*} (X)$. We examine the value of $(\cg a)\e_{|x|+|y|}$ under $\Psi_{*}$. By  
		Proposition \ref{btmpro} (3) and fact (\ref{ps2}), we infer that
		$$\Psi_{*}\colon (\cg a)\e_{|x|+|y|}\stackrel{\mathrm{id}\otimes g_{*}}\longmapsto(\cg a) (xy+tyx)\stackrel{\beta_{3}/3}\longmapsto\frac{1}{3}\cg([[a,x],y]+t[[a,y],x])\stackrel{\kj_{2*}}\longmapsto\frac{1}{3}\cg([[a,x],y]+t[[a,y],x]). $$
		\noindent Hence, $\Psi_{*} ((\cg x)\e_{|x|+|y|})=\frac{t}{3}\cg([[x,y],x]),\;\,\Psi_{*} ((\cg y)\e_{|x|+|y|})=\frac{1}{3}\cg([[y,x],y])=\pm \frac{1}{3}\cg([[x,y],y]).$
		\noindent  So, $\mathrm{Im} (\Psi_{*})=\widetilde{H}_{*} (\mathrm{SQ}_{3}^{\mathrm{max}} (X)).$   Noting that the domain and codomain of $\Psi_{*}$ are vector spaces of dimension 2, we deduce that $\Psi_{*}$ is an isomorphism. 
		By the $p$-local Whitehead theorem, we know that $\Psi\colon (S^{1}\wedge X)\wedge (S^{|x|}\wedge S^{|y|})\rightarrow \mathrm{SQ}_{3}^{\mathrm{max}} (X)$ 
		is a homotopy equivalence. 
	\end{proof}

	\begin{lem0}\label{spn1}  Suppose that the prime $p\neq3.$ Then
		$$\widetilde{H}_{*} (\s^{2n+k+1}C_{f})\tg\mathrm{Span}_{\z/p}\n [[x,y],y],\;[[x,y],x]\nn\xyd T(x,y)=H_{*} (\lo\s C_{f}).$$ In addition, this isomorphism is induced by the composition$$\s^{2n+k+1}C_{f}\xrightarrow[]{E}\lo\s^{2n+k+2}C_{f}\stackrel{\simeq\;}\rightarrow\lo\mathrm{SQ}_{3}^{\mathrm{max}} (C_{f})\hookrightarrow \mathrm{A}^{\min} (C_{f})\times \mathrm{B}^{\max} (C_{f})\stackrel{\simeq\;}\rightarrow\lo\s C_{f}.$$ 
	\end{lem0}

	\begin{proof} Lemma \ref{spn} and Lemma \ref{q3max} imply   the first part of the  lemma. The second part is in fact a standard result of  Selick-Wu's $\mathrm{A}^{\min}$-theory and is often used without explicit mention; see  Remark~\ref{wddrmk} (2), in particular Diagram (\ref{jsys}) for a comprehensive explanation.\end{proof}
	\;\\ \indent As a side remark, the theory of Selick-Wu has been generalized. In 2006, in collaboration with Theriault,  they generalized the result to    looped $p$-local coassociative co-H spaces \cite{stw}. In 2013,  Grbi\'{c},  Theriault and Wu \cite{amin2013} generalized the result to    looped $p$-local co-H spaces.  It should   be noted that in 2006, Gray \cite{Gray2006} also obtained   decompositions similar to those in \cite{stw}; and in some  cases his  decompositions have better properties. Nevertheless, Selick-Wu's 
	$\mathrm{A}^{\min}$-theory enjoys somewhat more satisfactory properties from both the algebraic and the functorial points of view.

	\section{Proof of the main theorem}\label{secproofbeta}
	In this section,  spaces are assumed to be localized at a prime $p$ such that either
	\begin{itemize}
		\item[$\bullet$]  $p=2$,
		\item[$\bullet$] or  $p\geq5$ and $n\equiv k+1\equiv0\m{2}$, which ensures that both $|x|$ and $|y|$ are even. 
	\end{itemize}
	
	Recall that we have assumed that $H^{*} (C_{f})$ is a trivial algebra.

	\subsection{A table of  notation}\label{hbg1}  The proof of the main theorem in  Section \ref{secproofbeta} requires the use of a substantial amount of notation. For the reader's convenience, the following table summarizes the notation used in Section~\ref{secproofbeta}, including some symbols that will be defined later.

	\begin{longtable}{ll}
		\hline
		Symbol & Description \\[0.3em]
		\hline
		\endfirsthead
		
		\hline
		Symbol & Description \\[0.3em]
		\hline
		\endhead
		
		\hline
		\endfoot
		
		\hline
		\endlastfoot
		
		$f$ & $f\in\pi_{n+k} (S^{n})$, ($n\geq2,k\geq0$); $f=p^{a}\e_{n}\in\pi_{n} (S^{n})$ with $a\in\z_{+}$ if  $k=0$ \\  
		$C_{f}$ & the homotopy cofiber, $C_{f}=S^{n} \cup _{f}e^{n+k+1}$\\
		$\text{Hopf algebra}\;T(W)$ & the tensor Hopf  algebra $T(W)$ obtained by letting elements  in   $W$ be primitive \\ $x,y$ & the generators; the  Hopf  algebra  $H_{*} (\lo\s C_{f})=T(x,y),\;(|x|=n,\;   |y|=n+k+1)$  \\
		$F$ & the homotopy fiber of the pinch map $\s C_{f}\rightarrow S^{n+k+2}$ 
		\\ 
		
		$\fe$ & the 2-cell skeleton of $F$; $\fe=S^{n+1}\cup _{[\e_{n+1},\s f]}\,e^{2n+k+2}$ \\ 
		$\fs$ & the 3-cell skeleton of $F$; $\fs=\fe\cup _{\beta}\,e^{3n+2k+3}$ \\ 
		$G$ & the homotopy fiber of the pinch map $\fe\rightarrow S^{2n+k+2}$; \\ 
		
		& $G=(S^{n+1}\vee S^{3n+k+2})\cup e^{5n+2k+3}\cup\cdots$ \\ 
		
		$A$ & the homotopy fiber of $S^{3n+k+1}\hookrightarrow \s^{2n+k+1}C_{f}$; with bottom cell $S^{3n+2k+1}$ \\ 
		
		$B$ & the homotopy fiber of
		$\fe\hookrightarrow\fs$; with bottom cell $S^{3n+2k+2}$ 
		\\ 
		$i_{1}$ & 
		$i_{1}\colon  S^{3n+k+1}\rightarrow \s^{2n+k+1}C_{f}$ denotes the inclusion
		\\ 
		$i_{2}$ & 
		$i_{2}\colon   \s^{2n+k+1}C_{f}\rightarrow\lo\fs$ denotes the restriction of $\s^{2n+k+1}C_{f}\hookrightarrow\lo F$ \\  
		$i_{3}$ & 
		$i_{3}\colon  S^{3n+2k+2}=\sk_{3n+2k+2} (B)\hookrightarrow B$ denotes the  inclusion  \\
		
		$i_{4}$ & 
		$i_{4}\colon  S^{3n+2k+1}=\sk_{3n+2k+1} (\lo B)\hookrightarrow\lo B$ denotes the inclusion  \\

		$i_{5}$ & 
		in the homotopy fibration $B\stackrel{i_{5}}\longrightarrow \fe\stackrel{\;j_{2}}\hookrightarrow \fs$   \\

		$j_{2}$ & $j_{2}\colon  \fe\hookrightarrow \fs$ denotes the inclusion \\

		$j_{2.5}$ & $j_{2.5}\colon  S^{n+1}\hookrightarrow \fe$ denotes the inclusion  \\
		$j_{3}$ & in the homotopy fibration $G\stackrel{j_{3}}\longrightarrow \fe\xrightarrow[q_{2}]{pinch}S^{2n+k+2}$\\

		$j_{3.5}$ & $S^{n+1} \overset{\,j_{3.5}\,}{\xyd} G=(S^{n+1}\vee S^{3n+k+2})\cup e^{5n+2k+3}\cup\cdots;\;\, j_{2.5}= j_{3}j_{3.5}$ up to a unit in $\z_{(p)}$\\ 
		$j_{4}$ & $S^{3n+k+2} \overset{\,j_{4}\,}{\xyd} G=(S^{n+1}\vee S^{3n+k+2})\cup e^{5n+2k+3}\cup\cdots$ \\ 
		$q_{2}$ & the pinch map $\fe=S^{n+1}\cup e^{2n+k+2}\xrightarrow{}S^{2n+k+2}$  \\

		$\mathrm{Im} (q_{2*})$ & denotes $\mathrm{Im} (\pi_{3n+k+2} (\fe)\xrightarrow{q_{2*}}\pi_{3n+k+2} (S^{2n+k+2}))$  \\
		$\mathrm{d}$ & $\mathrm{d}=\mathrm{dim}_{\z/p} (\z/p\otimes \mathrm{Im} (q_{2*}))$, the number of the direct summands of  $\mathrm{Im} (q_{2*})$ \\
		$\n\s\mathbbm{y}^{(i)}\nn_{i=1}^{\mathrm{d}}$ & an arbitrary minimal generating set of   $\mathrm{Im} (q_{2*})$, viewed as the empty set if $\mathrm{d}=0$\\
		$[\s\mathbbm{y}^{(i)}]$ & the lift of
		$-\s\mathbbm{y}^{(i)}$, viewed as  $0$ if $\mathrm{d}=0$; defined in Lemma \ref{410}\\
		$I$ & $I=j_{2.5}\hc \pi_{3n+k+2} (S^{n+1})+\x \;[\s\mathbbm{y}^{(i)}]\;|\; 1\leq i\leq \mathrm{d}\;\xx.$\\
		$\kj$ & $\phi\colon   S^{3n+k+1}\rightarrow \lo\fe$ is  defined by Diagram (\ref{phiddy}) and satisfies $  (\lo j_{2})\hc\phi= i_{2}\hc i_{1}$ \\ 
		$\rho$ &  $\rho\colon   A\rightarrow \lo B$ is a map defined by   Diagram (\ref{yrd})\\ 
		$\xi,\ck,\lt$ & these three maps are defined by  Diagram (\ref{mldt})\\ 
		$\mathrm{adj} (-)$  
		&   the classical isomorphism $[\s-,-]\stackrel{\tg}\longrightarrow[-,\lo-]$\\ 
		$\mathrm{ad}^{i}(v) (u)$  
		&   $\mathrm{ad}^{i}(v) (u)=[\cdots [[ u,v],v]\cdots,v]$ where $v$ occurs $i$ times;\;$\mathrm{ad}^{0}(v) (u)=u$
	\end{longtable}

	\subsection{The homology of looped $F$ and its skeletons}

	Recall that we have the nontrivial CW decomposition
	$C_{f}=S^{n}\cup e^{n+k+1}$. We  have an identification of the  Hopf algebras $H_{*} (\lo\s C_{f})=T(x,y),\;(|x|=n,\;   |y|=n+k+1).$ 
	Then Theorem~\ref{dlydexj} yields  the following lemma.
	\begin{lem}\label{hltd}  
		$H_{*} (\lo F)\tg T(\n\mathrm{ad}^{m} (y)(x)\;|\; m\geq0\nn)$ as Hopf algebras. \hfill\boxed{}
	\end{lem}

	Proposition~\ref{grdl} (2) implies that 
	$\widetilde{H}_{*} (F)=\widetilde{H}_{*} (S^{n+1})\otimes H_{*} (\lo \s S^{n+k+1}),$ a graded vector space generated by the elements $ab^{m}\;(m\geq0)$ with $|a|=n+1 $ and $  |b|=n+k+1$. One may be interested in the relation between the generators of $H_{*} (\lo F)$ and $H_{*} (F)$. The following lemma answers this question.
	\begin{lem} \label{pxlmb} Let $m\geq0.$ In the $\z/p$-homology Leray-Serre spectral sequence associated with the homotopy fibration $\lo F\rightarrow*\rightarrow F$,  we have $\tau(ab^{m})=\mathrm{ad}^{m} (y)(x)\; $ up to a unit in $\z/p$, where $\tau$ denotes the transgression.  That is, we have $ab^{m}=\cg'(\mathrm{ad}^{m} (y)(x))$ up to a unit in $\z/p$,  where $\cg'$ denotes the homology suspension.\end{lem}
	\begin{proof}
	Consider the $\z/p$-homology Leray-Serre spectral sequence $\n (E^{\bullet}_{\bullet,\bullet},d^{\bullet})\nn$ associated with the homotopy fibration $\lo F\rightarrow *\rightarrow F$ and apply Proposition \ref{mdpxl}. Observe that
	$$E^{2}_{0,*}=H_{*}(\lo F)=T(\n\mathrm{ad}^{m}(y)(x)\;|\;m\geq0\nn),\;\;(|x|=n,\;|y|=n+k+1),$$
	$$E^{2}_{*,0}=H_{*}(F)=\mathrm{Span}_{\z/p}\n 1, ab^{m}\;|\;m\geq0\nn,\;\;(|a|=n+1,\;|b|=n+k+1).$$
	\noindent In the following, ``$=$" stands for being equal up to a unit in $\z/p$. The relations
	$$d(a)=x=\mathrm{ad}^{0}(y)(x),\;d(a\cdot\af)=x\af,\;(a\cdot\af\in E^{n+1}_{n+1,\bullet},\;\;\af\in E^{2}_{0,\bullet})$$
	\noindent in the $E^{n+1}$-page
	force $d(ab)=[x,y]=\mathrm{ad}^{1}(y)(x)$ in the
	$E^{2n+k+2}$-page. The relations
	$$d(ab)=\mathrm{ad}^{1}(y)(x),\;d(ab\cdot\af)=(\mathrm{ad}^{1}(y)(x))\af,\;(ab\cdot\af\in E^{2n+k+2}_{2n+k+2,\bullet},\;\;\af\in E^{2}_{0,\bullet})$$
	in the $E^{2n+k+2}$-page imply $d(ab^{2})=\mathrm{ad}^{2}(y)(x)$ in the
	$E^{3n+2k+3}$-page. It is not hard to obtain $d(ab^{m})=\mathrm{ad}^{m}(y)(x)$ for all $m$ by induction.
\end{proof}
		Recall  that  \vspace{-0.5\baselineskip}\begin{equation}\notag F^{(2)}=\sk_{3n+2k+2} (F)=S^{n+1}\cup_{\af}  e^{2n+k+2},\end{equation}\begin{equation}\notag\text{}F^{(3)}=\sk_{4n+3k+3} (F)=(S^{n+1}\cup_{\af}  e^{2n+k+2})\cup_{\beta}  e^{3n+2k+3}.\end{equation}

	\noindent The homology Hopf algebra structures of     the loop spaces of these two skeletons are  as follows.
	
	\begin{lem}\label{lof2} 
		\begin{itemize}
			\item[\rm(1)] $H_{*} (\lo\fe)\tg T(x, [x,y])$ as Hopf algebras.
			\item[\rm(2)] $H_{*} (\lo\fs)\tg T(x, [x,y],[[x,y],y])$ as Hopf algebras.
		\end{itemize}

	\end{lem}
	\begin{proof}
		\begin{itemize}
			\item[\rm(1)]
			Consider the $\z/p$-homology Leray-Serre spectral sequence $\n (E^{\bullet}_{\bullet,\bullet},d^{\bullet})\nn$ associated with the homotopy fibration $\lo \fe\rightarrow*\rightarrow \fe$. Recall that it is an  $H_{*} (\lo \fe)$-module spectral sequence,  $d(u v)=(du) v$ where $u\in E_{\bullet,0}^{\bullet}$ and $v\in E_{0,\bullet}^{2}$\;(see Proposition \ref{mdpxl}). Notice that the pair $(\lo F, \lo\fe)$ is $((3n+2k+2)-1)$-connected. Thus we have
			$H_{\bullet\leq s} (\lo\fe)\xyd H_{\bullet\leq s} (\lo F)$ where $s=|[x,y]|$. Additionally, $(\lo i)_{*}$ is a Hopf algebra homomorphism. Combining this with the algebra structure of  $H_{*} (\lo F)$, we infer that in $H_{*} (\lo\fe)$, every monomial of the form $$x^{i_{1}} [x,y]^{i_{2}}x^{i_{3}} [x,y]^{i_{4}}\cdots x^{i_{t}} [x,y]^{i_{t+1}}$$ is nonzero where the numbers $i_{\bullet}$ denote non-negative  integers. Then, a calculation following the  method in the proof of Lemma \ref{pxlmb} shows that $H_{*} (\lo\fe)\tg T(x, [x,y])$ as  algebras. To show  $x, [x,y]\in H_{*} (\lo\fe)$ are primitive, notice that $\lo\fe\rightarrow\lo F$ induces a monomorphism of tensor algebras (thereby of Hopf algebras) and $x, [x,y]\in H_{*} (\lo F)$ are primitive. \item [\rm(2)] The proof is similar to that of assertion (1). \end{itemize} \end{proof}

	We  use the identification $H_{*} (\lo F)= T\n\mathrm{ad}^{m} (y)(x)\;|\; m\geq0\nn,$ which is a sub-Hopf algebra of $H_{*} (\lo\s C_{f})=T(x,y)$.  Denote the element $[x,y]\in H_{*} (\lo\fe)= T(x, [x,y])$ by $z$. Thus $|z|=2n+k+1.$
	
	In $H_{*} (\lo G),$ the element $\mathrm{ad}^{1} (z)(x)=[x,z]=[x,[x,y]]=\pm[[x,y],x]$ given in the following lemma will be critical.
	\begin{lem}\label{lpgtd} Let $G$ be the homotopy fiber of the pinch map $\fe\rightarrow S^{2n+k+2}$. Then, there exists an isomorphism between  Hopf algebras  $H_{*} (\lo G)\tg T(\n\mathrm{ad}^{m} (z)(x)\;|\; m\geq0\nn),$ where $|x|=n $ and $|z|=2n+k+1$.
		
	\end{lem}
	
	\begin{proof} The lemma follows from Theorem~\ref{dlyde} and Lemma~\ref{lof2} (1).\end{proof}

	\subsection{Some calculations}
	The following lemma follows directly from a result of Barcus and Barratt \cite[Corollary~7.4, p.~69]{BG}. Notice that $[[\e,\e],\e]=\frac{1}{3} (3[[\e,\e],\e])=0$, where $\e$ denotes the identity map of a path-connected sphere.
	\begin{lem}\label{whgs} Let $\mathbbm{x}\in\pi_{i} (S^{m})$, $\mathbbm{y}\in\pi_{j} (S^{m})$ and write $\mathrm{id}_{S^{m}}=\e_{m}$ where $m\geq2$. Then their Whitehead product   is given by $[\mathbbm{x},\mathbbm{y}]=\pm[\e_{m},\e_{m}]\hc\s^{m-1}\mathbbm{y}\hc\s^{j-1}\mathbbm{x}.$\hfill\boxed{}
		
	\end{lem}
	
	 Lemma~\ref{whgs} implies the following lemma.
	\begin{lem}\label{lpgfl}
		$G\simeq(S^{n+1}\vee S^{3n+k+2})\cup e^{5n+2k+3}\cup\cdots.$\hfill\boxed{}
	\end{lem}

	Let us examine the inclusions of the bottom cells.
	
	\begin{lem}\label{xr} Let $A,\,B$ be the homotopy fibers of the    inclusions  $$S^{3n+k+1}\hookrightarrow \s^{2n+k+1}C_{f}\;\text{and}\,\;  \fe\hookrightarrow\fs,$$ respectively. Then, we have $\sk_{3n+2k+1} (A)=\sk_{3n+2k+1} (\lo B)=S^{3n+2k+1}$ and $\sk_{3n+2k+2} ( B)=S^{3n+2k+2}$.
		
	\end{lem}
	\begin{proof} The argument is straightforward, as the homology groups are readily determined via the Leray-Serre spectral sequences. 
	\end{proof}
	\;\\ \indent Remark~\ref{Fdll} implies that the pair $(\lo F,\lo\fs)$ is $((4n+3k+3)-1)$-connected, and  that the pair $(\lo \fs,\lo\fe)$ and  the pair $(\lo F,\lo\fe)$    are $((3n+2k+2)-1)$-connected.
	Let $i_{0}\colon\s^{2n+k+1}C_{f}\hookrightarrow\lo F$ be the inclusion via the identification $\lo F=\lo\mathrm{SQ}_{3}^{\max} (C_{f})\times K$, where $K$ is the product of the remaining factors from applying the Hilton-Milnor splitting   for the decomposition given in Lemma \ref{fjdlyw}. Let  $$i_{1}\colon  S^{3n+k+1}\rightarrow \s^{2n+k+1}C_{f},\;\;j_{2}\colon  \fe\rightarrow\fs $$ be  the inclusions and
	$i_{2}\colon   \s^{2n+k+1}C_{f}\rightarrow\lo\fs$ be the restriction of the inclusion $i_{0}\colon\s^{2n+k+1}C_{f}\hookrightarrow\lo F$. 
	
	Owing to the connectivity properties of the pairs involved,  there exists a map  $\phi\colon   S^{3n+k+1}\rightarrow \lo\fe$  such that 
	the following diagram is commutative:
	\begin{equation}\label{phiddy}
		\begin{tikzcd}
			S^{3n+k+1} \arrow[r,"i_{1}"] \arrow[d,"\phi"'] & \s^{2n+k+1}C_{f} \arrow[d,"i_{0}"] \\
			\lo\fe \arrow[r,hook] & \lo F.
	\end{tikzcd}\end{equation}
	\begin{lem}\label{cfjht} The following diagram is commutative:
		\[
		\begin{tikzcd}
			S^{3n+k+1} \arrow[r,"i_{1}"] \arrow[d,"\phi"'] & \s^{2n+k+1}C_{f} \arrow[d,"i_{2}"] \\
			\lo\fe \arrow[r,"\lo j_{2}"] & \lo\fs.
		\end{tikzcd}
		\]
		Moreover, for a suitable  generator $\e$ of  
		$H_{3n+k+1} (S^{3n+k+1})$,  we have
		$$i_{1*}\e = [[x,y],x]\in H_{3n+k+1} (\s^{2n+k+1}C_{f}),\quad i_{2*}[[x,y],x]=  [[x,y],x]\in H_{3n+k+1} (\lo\fs), $$
		$$(\lo j_{2})_{*}[[x,y],x]=  [[x,y],x]\in H_{3n+k+1} (\lo\fs)\;\,\text{and}\;\; \phi_{*}\e= [[x,y],x]\in H_{3n+k+1} (\lo\fe). $$
	\end{lem}
	\begin{proof} The homomorphism $\pi_{3n+k+1} (\lo \fs\hookrightarrow\lo F)$ is injective; the maps $\lo j_{2}\hc\kj$ and $i_{2}\hc i_{1}$ 
		are equal after composing with  the map $\lo \fs\hookrightarrow\lo F$. Thus $\lo j_{2}\hc\kj=i_{2}\hc i_{1}$. The remaining part of 
		the lemma follows from Lemma~\ref{lof2}, Lemma~\ref{spn1} and the connectivity properties of the pairs involved.\end{proof}

	We examine the$\m p$ spherical elements and give the following lemma, which is critical to the proof of our main theorem. 
	
	\begin{lem} \label{gyf}    The adjoint map of the inclusion $j_{4}\colon S^{3n+k+2}\hookrightarrow G=(S^{n+1}\vee S^{3n+k+2})\cup e^{5n+2k+3}\cup\cdots$ is the inclusion $$S^{3n+k+1}\hookrightarrow S^{n}\vee S^{3n+k+1}\hookrightarrow \lo(S^{n+1}\vee S^{3n+k+2})\hookrightarrow\lo((S^{n+1}\vee S^{3n+k+2})\cup e^{5n+2k+3}\cup\cdots)=\lo G.$$
		Moreover, in $\z/p$-homology, we have
		$\phi_{*}\e=(\lo j_{3}\hc \mathrm{adj} (j_{4}))_{*} (\e)=[[x,y],x]\in H_{3n+k+1}(\lo\fe)$ up to a unit in $\z/p$. Here, $\e$ generates $H_{3n+k+1}(S^{3n+k+1})$ as in the above lemma.
	\end{lem}

	\begin{proof} The first part is obvious. We consider the second part. Lemmas \ref{lof2} and 
    \ref{lpgtd} tell us that
    $$H_{*} (\lo G)=T(\n\mathrm{ad}^{m} ([x,y])(x)\;|\; m\geq0\nn)\xyd H_{*} (\lo \fe)= T(x, [x,y])$$
   up to the isomorphisms induced by the obvious maps.
   Additionally, $H_{3n+k+1}(\lo j_{3})$ sends $[[x,y],x]\in H_{3n+k+1}(\lo G)$ to $[[x,y],x]\in H_{3n+k+1}(\lo\fe)$. Recall that $\phi_{*}\e=[[x,y],x]\in H_{3n+k+1}(\lo\fe)$ given in Lemma \ref{cfjht}. So we need only to show that $(\mathrm{adj} (j_{4}))_{*} (\e)=[[x,y],x]\in H_{3n+k+1}(\lo G)$  up to a unit in $\z/p$.
By Proposition \ref{grdl}, we may set the vector space $$\widetilde{H}_{*}(G)=\widetilde{H}_{*}(S^{n+1})\otimes H_{*}(\lo\s S^{2n+k+1})=\mathrm{Span}_{\z/p}\n \mathbbm{a}\otimes\mathbbm{c}^{i}\;|\;i\geq0\nn,$$
where $|\mathbbm{a}|=n+1,\;|\mathbbm{c}|=2n+k+1.$ Then, up to a unit in $\z/p$, $j_{4*}(\cg (\e))=\mathbbm{a}\otimes\mathbbm{c}$ where $\cg (\e)$ generates $H_{3n+k+2}(S^{3n+k+2}).$
By an argument entirely parallel to the proof of
Lemma \ref{pxlmb}, we obtain
$$\tau(\mathbbm{a}\otimes\mathbbm{c})=[x,[x,y]]\in\widetilde{H}_{*}(\lo G)$$
up to a unit in $\z/p$. Here, $\tau=\tau_{G}$ is the transgression
in the homology Leray--Serre spectral sequence for
$\lo G\to*\to G$. Since the transgression $\tau$ is natural, we have a commutative diagram: 
\[
\begin{tikzcd}[row sep=1.8em, column sep=4.5em]
H_{3n+k+1}(S^{3n+k+1})
  \arrow[d, "E_*"]
  \arrow[dr, "(\operatorname{adj}(j_4))_*"]
&
\\
H_{3n+k+1}(\Omega S^{3n+k+2})
  \arrow[r, "(\Omega j_4)_*"]
&
H_{3n+k+1}(\Omega G)
\\
H_{3n+k+2}(S^{3n+k+2})
  \arrow[u, "\tau"]
  \arrow[r, "j_{4*}"]
&
H_{3n+k+2}(G).
  \arrow[u, "\tau"']
\end{tikzcd}
\] The above commutative diagram  implies that
$(\mathrm{adj}(j_{4}))_{*}(\e)=[[x,y],x]\in H_{3n+k+1}(\lo G)$
up to a unit in $\z/p$.
 \end{proof}

	\subsection{The maps related to $\beta$}\label{iscor}

	By Lemma~\ref{xr}, we may take
	\[
	i_{3}\colon S^{3n+2k+2}=\sk_{3n+2k+2} (B)\rightarrow B,\quad
	i_{4}\colon S^{3n+2k+1}=\sk_{3n+2k+1} (\lo B)\rightarrow \lo B
	\]
	to be the canonical inclusions, and let $i_{5}\colon   B\rightarrow \fe$ denote the inclusion of the homotopy fiber.

	The diagram in Lemma~\ref{cfjht} results in a morphism of homotopy fibrations for some map $\rho$, \vspace{-0.8\baselineskip}
	\begin{equation}\label{yrd}\begin{tikzcd}
			A \arrow[r] \arrow[d,"\rho"] & S^{3n+k+1}\arrow[r,"i_{1}"] \arrow[d,"\phi"] & \s^{2n+k+1}C_{f} \arrow[d,"i_{2}"] \\
			\lo B \arrow[r,"\lo i_{5}"]            & \lo\fe \arrow[r,"\lo j_{2}"]            & \lo\fs.
		\end{tikzcd}
	\end{equation}

	\noindent Considering the inclusions of the bottom cells of $A $ and $\lo B$ and applying  the cellular approximation theorem, we conclude that there exist  maps  $\xi, \ck$ and $\lt$ making the following diagram commute: \vspace{-0.4\baselineskip}
	
	\begin{equation}\label{mldt}
		\begin{tikzcd}
			&S^{3n+2k+1} \arrow[d, hook] \arrow[ddl,"\ck"'] \arrow[dr,"\xi"]& & \\
			& A \arrow[r] \arrow[d,"\rho"] &S^{3n+k+1}\arrow[d,"\kj"] \arrow[r,"i_{1}"] & \s^{2n+k+1}C_{f} \arrow[d,"i_{2}"] \\
			S^{3n+2k+1} \arrow[r, hook," i_{4}"] \arrow[rr, bend right=25, "\lt"] & \lo B \arrow[r,"\lo i_{5}"] & \lo\fe \arrow[r,"\lo j_{2}"] & \lo\fs.
	\end{tikzcd}\end{equation}
	
	\noindent Consequently, we have a commutative diagram: \vspace{-0.8\baselineskip}
	
	\begin{equation}\label{sgxlzm}\begin{tikzcd}
			S^{3n+2k+1} \arrow[r,"\xi"] \arrow[d,"\ck"] & S^{3n+k+1}\arrow[r,"i_{1}"] \arrow[d,"\phi"] & \s^{2n+k+1}C_{f} \arrow[d,"i_{2}"] \\
			S^{3n+2k+1}\; \arrow[r,"\lt"]            & \lo\fe \arrow[r,"\lo j_{2}"]            & \lo\fs.
	\end{tikzcd}\end{equation}
	\noindent By  Lemma~\ref{zzd}, the map $\xi$ can be taken as $\s^{2n+k+1}f$. We claim that the self-map $\ck$ can be chosen as  $\mathrm{id}.$
	
	\begin{lem} The degree of $\ck$ is a unit  in the ring $\z_{(p)}.$ So $\ck$ can be taken as $\mathrm{id}$, after replacing $\lt$ by $\ell\lt$ for some unit $\ell$ in $\z_{(p)}.$
		
	\end{lem}
	
	\begin{proof}
		Consider the induced homomorphism $H_{3n+2k+1} (\rho)=\rho_{*}$. The space $\s^{2n+k+1}C_{f}$ is $((3n+k+1)-1)$-connected, and hence $(n-1)$-connected. 
		Naturality of the Leray-Serre long exact sequence (\cite[Proposition~3.2.1, p.~68]{ko1996})  gives the following commutative diagram with exact second row:\[\begin{tikzcd}
			& H_{3n+2k+2} (\s^{2n+k+1}C_{f}) \arrow[r,"\pa_{1}"] \arrow[d,"i_{2*}"] & H_{3n+2k+1} (A) \arrow[d,"\rho_{*}"] \\
			H_{3n+2k+2} (\lo\fe) \arrow[r, "(\lo j_{2})_{*}"]            & H_{3n+2k+2} (\lo\fs) \arrow[r,"\pa_{2}"]            & H_{3n+2k+1} (\lo B).
		\end{tikzcd}
		\]
		\noindent The exactness gives $\mathrm{Im} (\pa_{2})\tg\cok(\lo j_{2})_{*}.$
		Combining it with Lemma~\ref{lof2} we infer that \begin{equation}\label{gycok}\mathrm{Im} (\pa_{2})\tg\cok(\lo j_{2})_{*}=\mathrm{Span}_{\z/p}\n [[x,y],y]+\mathrm{Im} (\lo j_{2})_{*}\nn\tg\z/p.
		\end{equation}
		Recall that $[\s^{2n+k+1}C_{f}, \lo\fs]\tg[\s^{2n+k+1}C_{f}, \lo F]$; and note that  $i_{2}\colon  \s^{2n+k+1}C_{f}\rightarrow\lo\fs$ denotes the restriction of $\s^{2n+k+1}C_{f}\stackrel{E}\longrightarrow\lo \mathrm{SQ}^{\max}_{3} (C_{f})\hookrightarrow\lo F$. Then, we have $i_{2*} ([[x,y],y])=[[x,y],y]\in H_{3n+2k+2} (\lo\fs) $ by Lemma~\ref{spn1}. In addition, Equation (\ref{gycok}) implies that $$0\neq\pa_{2}[[x,y],y]\in H_{3n+2k+1} (\lo B).$$ Therefore, $\pa_{2}\hc i_{2*}\neq0$ which results in $\rho_{*}\neq0$. For the space $A$ and the space $\lo B$, the inclusions of the bottom cells  both induce isomorphisms between the $\z/p$-homology groups of dimension $3n+2k+1$. Thus, $\rho_{*}\neq0$ implies $H_{3n+2k+1} (\ck)=\ck_{*}\neq0$,  completing the proof.\end{proof}

	After replacing $\lt$ by $\ell\lt$, ($\ell$:   a unit in $\z_{(p)}$),
	we have a commutative diagram
	\begin{equation}\label{zjt}\begin{tikzcd}
			S^{3n+2k+1}\;\; \arrow[r,"\s^{2n+k+1}f"] \arrow[d, equal] & \;\;S^{3n+k+1}\arrow[r,"i_{1}"] \arrow[d,"\phi"] & \s^{2n+k+1}C_{f} \arrow[d,"i_{2}"] \\
			S^{3n+2k+1}\; \arrow[r,"\lt"]            & \lo\fe \arrow[r,"\lo j_{2}"]            & \lo\fs.
		\end{tikzcd}
	\end{equation}
	\noindent Equation (\ref{f3dy}) results in a cofiber sequence $S^{3n+2k+2}\stackrel{\beta}\longrightarrow\fe\stackrel{j_{2}}\longrightarrow\fs.$   Then $\beta=i_{5}\hc i_{3}$ up to a unit in $\z_{(p)}$ by  Lemma~\ref{zzd}. 
	
	Now, without changing the homotopy type of $\fs$, replace $\beta$ by $\ell\beta$ for some unit $\ell$ in $\z_{(p)}$ such that $\beta=i_{5}\hc i_{3}$ strictly.

	Denote  the isomorphism $[\s X,Y]\stackrel{\tg}\longrightarrow [X,\lo Y]$ by $\mathrm{adj}$ for any spaces $X,Y$.
	
	We shall show that $\mathrm{adj} (\beta)=\lt$ up to a unit in $\z_{(p)}$.
	Applying the functor $\lo(-)$ to $\beta=i_{5}\hc i_{3}$, we have $\lo\beta=\lo i_{5}\hc \lo i_{3}$. Then, we  obtain a  diagram  whose second row is a homotopy fibration, 
	\begin{equation}\label{nbdt}
		\begin{tikzcd}
			\lo S^{3n+2k+2}\arrow[dr,"\lo\beta"]  \arrow[d,"\lo i_{3}"'] & S^{3n+2k+1}  \arrow[d,"\mathrm{adj} (\beta)"] \arrow[l,"E"']& \\
			\lo B \arrow[r,"\lo i_{5}"]  &\lo\fe \arrow[r,"\lo j_{2}"]  & \lo\fs.  \\
			S^{3n+2k+1}  \arrow[u,"i_{4}"]   \arrow[ur,"\lt"']                  &                    & 
		\end{tikzcd}
	\end{equation}
	
	Next we provide the following lemma.
	\begin{lem}\label{gylt}$\mathrm{adj} (\beta)=\lt$ up to a unit in $\z_{(p)}$, where $\mathrm{adj} (\beta)$ is the adjoint map of $\beta.$
	\end{lem}\begin{proof} In this proof, we freely use the commutativity of Diagram (\ref{nbdt}). Clearly, the map $\lo i_{3}\hc E \colon  S^{3n+2k+1}\rightarrow\lo B$ factors as $S^{3n+2k+1}\stackrel{\mu}\longrightarrow S^{3n+2k+1}\stackrel{i_{4}}\longrightarrow \lo B$ for some self-map $\mu.$ Recall that $i_{3}\colon S^{3n+2k+2}\rightarrow B$ is the inclusion of the bottom cell. Then,   $H_{3n+2k+1} (\mu)=\mu_{*}\neq0$ which results in $\lo i_{3}\hc E=i_{4}$ up to a unit in $\z_{(p)}$. Therefore, $\mathrm{adj} (\beta)=\lo i_{5}\hc\lo i_{3}\hc E=\lo i_{5}\hc i_{4}=\lt$  up to a unit  in $\z_{(p)}$.\end{proof}
	
	\subsection{The homotopy group $\pi_{3n+k+1} (\lo\fe)$ and maps lying in it}
	Let $j_{3.5}\colon  S^{n+1}\hookrightarrow G=(S^{n+1}\vee S^{3n+k+2})\cup e^{5n+2k+3}\cup\cdots$\;be the inclusion of the bottom sphere.

	The occurrence of a direct sum is particularly useful for our problem, and we provide the following lemma.
	\begin{lem}\label{asd1} 
		Suppose that $\pa\colon   \pi_{\bullet+1} (S^{2n+k+2})\rightarrow\pi_{\bullet} (G)$ is the connecting homomorphism with respect to the homotopy fibration $G\xrightarrow[]{j_{3}} \fe\xrightarrow[pinch]{q_{2}} S^{2n+k+2}$  and that  $r\leq5n+2k+1$. Then, the homomorphism 
		$(j_{3}j_{4})_{*}\colon  \pi_{r} (S^{3n+k+2})\rightarrow\pi_{r} (\fe)$ is injective and the sum $\mathrm{Im}(j_{2.5*})+\mathrm{Im}(j_{3}j_{4})_{*}\xyd \pi_{r}(\fe)$ is a direct sum.

	\end{lem}
	\begin{proof} Recall that, up to the classical isomorphism $\pi_{\bullet}(\lo-)\tg\pi_{\bullet+1}(-)$, the connecting homomorphism $\pa$ is induced by the map $\lo S^{2n+k+2}\rightarrow G$, which can be identified with the inclusion $$J(S^{2n+k+1})=J(S^{2n+k+1},S^{2n+k+1})
		\stackrel{\delta}\hookrightarrow J(M_{S^{n+1}},S^{2n+k+1});$$ see Proposition \ref{grdl} (1).  Recall that $\fe=S^{n+1}\cup_{[\e_{n+1},\s f]} CS^{2n+k+1}$. Denote by $\mathrm{i}$ the inclusion $S^{2n+k+1}\rightarrow M_{S^{n+1}}$, the replacement of the attaching map $[\e_{n+1},\s f]$. Then, applying Gray's relative  James-Hopf invariant $H_{2}$ given by  Proposition \ref{grdl} (6), we have a commutative diagram
		\[
		\begin{tikzcd}[column sep=large]
			J(S^{2n+k+1},S^{2n+k+1})
			\arrow[r, "{\delta}",hook]
			\arrow[d, "{H_2}"']
			&
			J(M_{S^{n+1}},S^{2n+k+1})
			\arrow[d, "{H_2}"]
			\\
			J(S^{2n+k+1}\wedge S^{2n+k+1})
			\arrow[r, "{J(\mathrm{i}\wedge \mathrm{id})}"]
			\arrow[d, equal]
			&
			J(M_{S^{n+1}}\wedge S^{2n+k+1})
			\arrow[d, "{\simeq}"]
			\\
			J(S^{4n+2k+2})
			\arrow[r, "{J\bigl([\iota_{n+1},\Sigma f]\wedge \mathrm{id}\bigr)}\;"]
			&
			J(S^{3n+k+2}).
		\end{tikzcd}
		\]
		
		\noindent  Hereafter, identify
		$J(M_{S^{n+1}}\wedge S^{2n+k+1})$ with $J(S^{3n+k+2})$. Notice that $[\iota_{n+1},\Sigma f]\wedge \mathrm{id}=0$, leading to $J\bigl([\iota_{n+1},\Sigma f]\wedge \mathrm{id}\bigr)=0$. 
		Applying $\pi_{r}(-)$,
		we have a commutative diagram
		\[
		\begin{tikzcd}
			\pi_{r}(J(S^{2n+k+1},S^{2n+k+1}))
			\arrow[r, "{\delta_{*}}"]
			\arrow[d, "{H_{2*}}"']
			&
			\pi_{r}(J(M_{S^{n+1}},S^{2n+k+1}))
			\arrow[d, "{H_{2*}}"]
			\\\pi_{r}(J(S^{4n+2k+2}))
			\arrow[r, "0"]
			&
			\pi_{r}(J(S^{3n+k+2})).
		\end{tikzcd}
		\]
		The assumption  gives $r\leq5n+2k+1$.
		Up  to the obvious isomorphisms induced by the inclusions of the skeletons and the isomorphism $\pi_{\bullet}(\lo-)\tg\pi_{\bullet+1}(-)$, the above commutative diagram becomes the following commutative diagram,

		\begin{equation}\label{0505aa}
			\begin{tikzcd}
				\pi_{r+1}(S^{2n+k+2})
				\arrow[r, "{\pa}"]
				\arrow[d, "{H_{2*}}"']
				&
				\pi_{r}(S^{n+1}\vee S^{3n+k+2})
				\arrow[d, "{H_{2*}}"]
				\\\pi_{r+1}(S^{4n+2k+3})
				\arrow[r, "0"]
				&
				\pi_{r+1}(S^{3n+k+3}).
			\end{tikzcd}
		\end{equation}
		\noindent
		As with the absolute James-Hopf invariant, Gray's relative James-Hopf invariant $$H_{2}\colon J(M_{S^{n+1}}, S^{2n+k+1})\to J(M_{S^{n+1}}\wedge S^{2n+k+1})$$ has  the property (see Gray \cite{Gray}) that $\mathrm{adj}^{-1}(H_{2})$ restricts to the pinch map
		$$\s J_{2}(M_{S^{n+1}},S^{2n+k+1})
		\simeq \s\bigl(M_{S^{n+1}}\bigvee M_{S^{n+1}}\wedge S^{2n+k+1}\bigr)
		\stackrel{\mathrm{pin}}\twoheadrightarrow
		\s\bigl(M_{S^{n+1}}\wedge S^{2n+k+1}\bigr).$$
		Then, the homomorphism $H_{2*}$ in Diagram (\ref{0505aa}) is induced by the composition $$S^{n+1} \vee S^{3n+k+2}=J_{2}(M_{S^{n+1}},S^{2n+k+1})\xrightarrow[]{\mathrm{adj(pin)}}\lo \s (M_{S^{n+1}}\wedge S^{2n+k+1})=J(M_{S^{n+1}}\wedge S^{2n+k+1}),$$ that is,  the composition
		$S^{n+1} \vee S^{3n+k+2}\xhookrightarrow{E}\lo\s (S^{n+1} \vee S^{3n+k+2})\xrightarrow[]{\lo(\mathrm{pin})}\lo \s (M_{S^{n+1}}\wedge S^{2n+k+1}).$
		So, the homomorphism $H_{2*}$ in Diagram (\ref{0505aa}) is equal to 
		the composition \begin{equation}\label{0505bbb}
			j_{3.5}\hc\pi_{r}(S^{n+1}) \jia j_{4}\hc\pi_{r}(S^{3n+k+2})\jia R=\pi_{r}(S^{n+1}\vee S^{3n+k+2})\stackrel{\s}\rightarrow\pi_{r+1}(S^{n+2}\vee S^{3n+k+3})\xrightarrow[]{proj.} \pi_{r+1}(S^{3n+k+3}),
		\end{equation}
		which  can be identified with the projection to the direct summand $\pi_{r}(S^{3n+k+2})$; notice that any suspension homomorphism $\pi_{r}(S^{3n+k+2})\stackrel{\s}\rightarrow \pi_{r+1}(S^{3n+k+3})$ with $r\leq5n+2k+1$ is an isomorphism.  Here, the summand $$ R=[j_{3.5},j_{4}]\hc\pi_{r}(S^{4n+k+2})\oplus [[j_{3.5},j_{4}],j_{3.5}]\hc\pi_{r}(S^{5n+k+2}) \jia\cdots$$ is given by Hilton-Milnor's splitting  killed by $\s$; see Whitehead's textbook (\cite[Theorem 8.1, p.~533]{GW}.)
		By  Diagram (\ref{0505aa}) and Equation (\ref{0505bbb}), we deduce that $\pa\pi_{r+1}(S^{2n+k+2})\xyd\pi_{r}(S^{n+1})\jia  R $ up to isomorphism. Precisely, as $r\leq5n+2k+1$, identifying the  group $\pi_{r}(G)$ with  $\pi_{r}(S^{n+1}\vee S^{3n+k+2})$, denoting the restrictions  of $j_{3}$ and $j_{4}$ on  $S^{n+1}\vee S^{3n+k+2}$ also by $j_{3} $ and $j_{4}$, respectively,  we have 
		$$\pa\pi_{r+1}(S^{2n+k+2})\xyd (j_{3.5}\hc\pi_{r}(S^{n+1}))\jia    R .$$   
		\noindent Note that $j_{2.5}=j_{3}j_{3.5}$ up to a unit in $\z_{(p)}$.  It is clear that
		the sum of $\mathrm{Im}j_{3.5*}+\mathrm{Im}j_{4*}+R$ is a direct sum in $\pi_{r} (G)$. So, $\mathrm{Im}j_{4*}\cap\mathrm{Im}\pa=0\xyd\pi_{r} (G)$. Note that $j_{4*}$ is injective with  $r\leq5n+2k+1$.  Hence, $\Ker (j_{3}j_{4})_{*}=0\xyd \pi_{r} (S^{3n+k+2}).$  It follows that the homomorphism  $\pi_{r}(j_{3}j_{4})=(j_{3}j_{4})_{*}$ is injective. Suppressing the homomorphisms
$j_{3.5*}=j_{3.5}\hc(-)$ and $j_{4*}=j_{4}\hc(-)$, from the notation,  we have 
		$$\mathrm{Im}(j_{3*})\tg\cok(\pa)= \frac{\big(\pi_{r}(S^{n+1}) \jia  R \big)\jia\pi_{r}(S^{3n+k+2})}{\pa\pi_{r+1}(S^{2n+k+2})}\tg \frac{\pi_{r}(S^{n+1}) \jia  R }{\pa\pi_{r+1}(S^{2n+k+2})}\jia\pi_{r}(S^{3n+k+2}).$$ 
		Up to these isomorphisms, we have $\mathrm{Im}(j_{2.5*})\xyd\frac{\pi_{r}(S^{n+1}) \jia  R }{\pa\pi_{r+1}(S^{2n+k+2})}$ and $\mathrm{Im}(j_{3}j_{4})_{*}=\pi_{r}(S^{3n+k+2}).$
		\noindent Then, we infer that the sum $\mathrm{Im}(j_{2.5*})+\mathrm{Im}(j_{3}j_{4})_{*}\xyd \mathrm{Im}(j_{3*}) \xyd \pi_{r}(\fe)$ is a direct  sum. \end{proof}

	\;\\ \indent Recall that $q_{2}$ is the pinch map $\fe=S^{n+1}\cup e^{2n+k+2}\rightarrow S^{2n+k+2}.$ We examine $\pi_{3n+k+1} (\lo\fe)$ to study $\kj$.

	\begin{lem}\label{410} Let $\mathrm{d}=\mathrm{dim}_{\z/p} (\z/p\otimes \mathrm{Im} (q_{2*}))$, the number of  the direct summands of $$\mathrm{Im} (q_{2*})=\mathrm{Im}\big(q_{2*}\colon\pi_{3n+k+2} (\fe)\rightarrow \pi_{3n+k+2} (S^{2n+k+2})= \s\pi_{3n+k+1} (S^{2n+k+1}) \big),$$ and  let  $\n\s\mathbbm{y}^{(i)} \nn_{ 1\leq i\leq \mathrm{d}} $ be  an  arbitrary minimal generating set of $\mathrm{Im} (q_{2*})$. (In concrete computations, it is often useful to choose those elements
		$\s\mathbbm{y}^{(i)}$ so that
		$\mathrm{Im}(q_{2*})=\bigoplus_{i=1}^{\mathrm{d}}\langle \s\mathbbm{y}^{(i)}\rangle$.
		This condition is not required here.) Take  $$[\s\mathbbm{y}^{(i)}]\in \mathrm{Coext}_{[\e_{n+1},\s f]} (\mathbbm{y}^{(i)})\xyd \n j_{2.5},\;[\e_{n+1},\s f],\;\mathbbm{y}^{(i)} \nn\xyd  \pi_{3n+k+2} (\fe).$$ The  element $[\s\mathbbm{y}^{(i)}]$ is in fact both a lift of $-\s\mathbbm{y}^{(i)}$ and  a coextension of $\mathbbm{y}^{(i)}$ for each $i$ if $\mathrm{d}\geq1$. Then, 
		$\pi_{3n+k+1} (\lo\fe)$ is generated by $$\text{elements lying in}\;(\lo j_{2.5})\hc \pi_{3n+k+1} (\lo S^{n+1}),\;\;\text{the element}\;\mathrm{adj} (j_{3}j_{4})=\lo j_{3}\hc \mathrm{adj} (j_{4})\colon S^{3n+k+1}\rightarrow\lo\fe$$  $$\text{and the elements}\;\;\mathrm{adj} ([\s\mathbbm{y}^{(i)}]),\;(1\leq i\leq \mathrm{d}).$$ 
		
		\noindent If \(\mathrm{d}=0\), so that the symbol
		$
		\n\s\mathbbm{y}^{(i)} \nn_{ 1\leq i\leq \mathrm{d}}$
		does not occur, we regard the $[\s\mathbbm{y}^{(i)}], (1\le i\le \mathrm{d}) $
		part together with the
		$
		\mathrm{adj} ([\s\mathbbm{y}^{(i)}]),\, (1\le i\le \mathrm{d})$
		part as \(0\).
	\end{lem}
	
	\begin{proof}  First consider the case $\mathrm{d}\geq1$. We will freely use some classical methods for calculating homotopy groups; see Remark \ref{jssq} for these methods. The homotopy fibration $G=(S^{n+1}\vee S^{3n+k+2})\cup e^{5n+2k+3}\cup\cdots \stackrel{j_{3}}\longrightarrow\fe\stackrel{q_{2}}\longrightarrow S^{2n+k+2}$ induces an exact sequence \begin{equation}\notag
			\pi_{3n+k+2} (S^{n+1}\vee S^{3n+k+2}) \stackrel{}\longrightarrow\pi_{3n+k+2} (\fe)\stackrel{q_{2*}}\longrightarrow \pi_{3n+k+2} (S^{2n+k+2})\stackrel{\pa}\rightarrow \pi_{3n+k+1} (S^{n+1}),
		\end{equation}
		
		\noindent which is canonically identified with the exact sequence
		$$\pi_{3n+k+1} (\lo (S^{n+1}\vee S^{3n+k+2})) \stackrel{}\longrightarrow\pi_{3n+k+1} (\lo\fe)\stackrel{(\lo q_{2})_{*}}\longrightarrow \pi_{3n+k+1} (\lo S^{2n+k+2})\stackrel{\pa}\rightarrow \pi_{3n+k} (\lo S^{n+1}).$$
		\noindent Obviously, $\pi_{3n+k+2} (S^{2n+k+2})=\s \pi_{3n+k+1} (S^{2n+k+1})$ and $\pa(\s\mathbbm{x})=[\e_{n+1},\s f]\hc\mathbbm{x}$  for each $\s\mathbbm{x}\in\pi_{3n+k+2} (S^{2n+k+2}).$ Let the elements $\s\mathbbm{y}^{(i)},\;(1\leq i\leq \mathrm{d})$ denote   generators of a minimal generating set of $$\mathrm{Im} (q_{2*})=\mathrm{Im} (q_{2*}\colon\pi_{3n+k+2} (\fe)\rightarrow \pi_{3n+k+2} (S^{2n+k+2}) )=\mathrm{Ker} (\pa)$$ where $\mathrm{d}=\mathrm{dim}_{\z/p} (\z/p\otimes \mathrm{Im} (q_{2*}))\geq1$ by the assumption. Therefore,  the chosen element $$[\s\mathbbm{y}^{(i)}]\in \mathrm{Coext}_{[\e_{n+1},\s f]} (\mathbbm{y}^{(i)})\xyd\n j_{2.5},\;[\e_{n+1},\s f],\;\mathbbm{y}^{(i)} \nn$$ is a lift of $-\s\mathbbm{y}^{(i)}$  as well as a coextension of $\mathbbm{y}^{(i)}$ for each $i$. So,  $\pi_{3n+k+2} (\fe)$ is generated by $$\text{elements lying in}\;j_{2.5}\hc \pi_{3n+k+2} (S^{n+1}),\;\;\text{the element}\;j_{3}j_{4}\colon S^{3n+k+2}\hookrightarrow S^{n+1}\vee S^{3n+k+2}\rightarrow \fe$$ \noindent and the elements $[\s\mathbbm{y}^{(i)}],\;(1\leq i\leq \mathrm{d})$.
		Equivalently, $\pi_{3n+k+1} (\lo\fe)$ is generated by $$\text{elements lying in}\;(\lo j_{2.5})\hc \pi_{3n+k+1} (\lo S^{n+1}),\;\;\text{the element}\;\mathrm{adj} (j_{3}j_{4})=\lo j_{3}\hc \mathrm{adj} (j_{4})\colon S^{3n+k+1}\rightarrow\lo\fe$$ \noindent and the elements $\mathrm{adj} ([\s\mathbbm{y}^{(i)}]),\;(1\leq i\leq \mathrm{d})$. Hence, the result holds when $\mathrm{d}\geq1$. It is easy to see that the result   also holds when $\mathrm{d}=0$. \end{proof}

	\;\\ \indent The following two lemmas are used to study the homology behavior of maps lying in $ \pi_{3n+k+1} (\lo\fe)$.

	\begin{lem}\label{gyts}  Suppose that $k\geq1$ or $n\geq2$ is even. Then,  for the space $\fe=S^{n+1}\cup _{[\e_{n+1},\s f]}e^{2n+k+2}$, the module $H_{*} (\lo \fe;\z_{(p)})$ is a free $\z_{(p)}$-module.\end{lem}
	\begin{proof} Recall that in the cohomology  of a space, $\af\cup \af=(-1)^{|\af|^{2}}\af\cup \af$. Then, the assumption $k\geq1$ or $n$ is even tells us  that    $H^{*} (\fe;\z_{(p)})$ is a trivial algebra.  Consider the $\z_{(p)}$-cohomology Leray-Serre spectral sequence  $\n( E_{\bullet}^{\bullet,\bullet}, d_{\bullet})\nn$ associated with the homotopy fibration $\lo \fe\rightarrow *\rightarrow \fe$. We know that $\fe$ is a finite $2$-connected CW complex. The spectral sequence implies that $H^{i}(\lo\fe;\z_{(p)})$ is a finitely generated $\z_{(p)}$-module for each $i$. Then the same holds for $H_{i}(\lo\fe;\z_{(p)})$. The proof is divided into 2 cases.

\textbf{Case 1}: $k=0$ and $n\geq2$ is even. We need only to show that $H^{*} (\lo\fe;\z_{(p)})$ is torsion free. To obtain a contradiction, assume that there exists a nonzero torsion element $a\in H^{*} (\lo\fe;\z_{(p)})$ with $|a|\geq1$. Then, we can find a nonzero torsion element $b\in H^{*} (\lo\fe;\z_{(p)})$ with $|b|\geq1$ such that $b$ generates a direct summand of $H^{|b|} (\lo\fe;\z_{(p)})$ and
 $|b|\leq |a'|$ for all nonzero torsion elements $a'\in H^{*} (\lo\fe;\z_{(p)})$.

The Leibniz rule and the trivial algebra structure of $H^{*} (\fe;\z_{(p)})$  imply that $d_{r} (u)=0$ for all  $u\in E_{r}^{n+1,j}$ with $j\geq 1$ and $r\geq2$. It follows that $d_{n+1}(b)=0$, since its target is free and $b$ is a torsion element. Notice that $d_r(E_r^{n+1,*})=0$ for $r=(2n+2)-(n+1)$. Hence $E_{2n+2}^{2n++2,j}$ is torsion free for every $j<|b|$.
       Then $d_{2n+2}(b)=0$ and so $b\neq0$ survives to $E_{\wq}^{0,|b|}$. However, all elements in $E_{\wq}^{(\bullet,\bullet)\neq0}$ are trivial, giving a contradiction. Thus, in this case,  $H^{*} (\lo \fe;\z_{(p)})$ is torsion  free and hence  $H_{*} (\lo \fe;\z_{(p)})$ is torsion free.

\textbf{Case 2}: $k\geq1$ and $n\geq2$ is arbitrary.  To obtain a contradiction, assume that there exists a nonzero torsion element $c\in H_{*} (\lo\fe;\z_{(p)})$ with $|c|\geq1$. Then, we can find a nonzero torsion element $\lambda\in H_{*} (\lo\fe;\z_{(p)})$ with $|\lambda|\geq1$ such that $\lambda$ generates a direct summand of $H_{|\lambda|} (\lo\fe;\z_{(p)})$ and
 $|\lambda|\leq |c'|$ for all nonzero torsion elements $c'\in H_{*} (\lo\fe;\z_{(p)})$.  
Let $\mathrm{j}\colon \fe\hookrightarrow F\to\s C_{f}$ be the composition. For clarity, in this proof we write the $\z/p$-homology as $H_{*}(-;\z/p)$, rather than using the abbreviation $H_{*}(-)$ adopted elsewhere in this paper. 
By the naturality of the coefficient reduction, the following diagram commutes: \[
\begin{tikzcd}[column sep=large, row sep=large]
H_{*}(\Omega F^{(2)};\mathbb{Z}_{(p)})
  \arrow[r, "(\Omega \mathrm j)_{*}"]
  \arrow[d, "\mathrm{re}"']
&
H_{*}(\Omega\Sigma C_{f};\mathbb{Z}_{(p)})
  \arrow[d, "\mathrm{re}"]
\\
H_{*}(\Omega F^{(2)};\mathbb{Z}/p)
  \arrow[r,hook, "(\Omega  \mathrm j)_{*}"]
&
H_{*}(\Omega\Sigma C_{f};\mathbb{Z}/p).
\end{tikzcd}
\]
\noindent
Here, the homomorphism $\mathrm{re}$ is given by the composition $$H_{*}(X;\z_{(p)})\twoheadrightarrow\frac{H_{*}(X;\z_{(p)})}{pH_{*}(X;\z_{(p)})}\cong H_{*}(X;\z_{(p)})\otimes_{\z_{(p)}}\z/p\hookrightarrow H_{*}(X;\z/p).$$By the choice of $\lambda$, we have $\lambda\notin pH_{*}(X;\z_{(p)})$. Then, $\mathrm{re}(\lambda)\neq0$ and so $0\neq(\lo\mathrm{j})_{*}(\mathrm{re}(\lambda))\in H_{*}(\Omega\Sigma C_{f};\mathbb{Z}/p).$
Note that $H_{*}(\lo\s C_{f};\z_{(p)})\cong T(\widehat{V})$ where $\widehat{V}=\widetilde{H}_{*}(C_{f};\z_{(p)})=\mathrm{Span}_{\z_{(p)}}\n e^{n},e^{n+k+1}\nn$ with $|e^{i}|=i$. So, $H_{*}(\lo\s C_{f};\z_{(p)})$ is $\z_{(p)}$-free. This and the fact $\lambda$ is torsional imply $(\lo\mathrm{j})_{*}(\lambda)=0\in H_{*}(\lo\s C_{f};\z_{(p)})$, leading to $\mathrm{re}((\lo\mathrm{j})_{*}(\lambda))=0\in H_{*}(\Omega\Sigma C_{f};\mathbb{Z}/p)$. It gives a contradiction. Thus, in this case, $H_{*}(\lo\fe;\z_{(p)})$ is torsion free.
\end{proof}

	\begin{lem}\label{gptyl} Suppose that $k=0$ and $n\geq2$ is odd. In this case,   $F^{(2)}=S^{n+1}\cup_{[\e_{n+1},\,p^{a}\e_{n+1}]}e^{2n+2}.$ Then the following assertions hold.
		(This lemma remains valid without the standing assumption of this section
that $n\equiv k+1\equiv0\m{2}$ if $p\geq5$.)

		\begin{itemize}
			\item [\rm (1)] $H^{3n+1}(\Omega F^{(2)};\mathbb{Z}_{(p)})
			\cong
			\mathbb{Z}_{(p)}\oplus(\z_{(p)}/(2p^{a}\z_{(p)}))$ and $
			H^{3n+2}(\Omega F^{(2)};\mathbb{Z}_{(p)})=0.$
			Consequently, in degree $3n+1$, where the element $[[x,y],x]\in H_{*}(\lo \fe)$ lies, we have
			$H_{3n+1}(\Omega F^{(2)};\mathbb{Z}_{(p)})\cong \mathbb{Z}_{(p)}$.
			\item [\rm (2)] In $\pi_{3n+1}(\lo\fe)$, we have $\x\mathrm{adj}(j_{3}j_{4})\xx\tg\z_{(p)}$. In addition,  $H_{3n+1} (\lo \fe;\z_{(p)})=\x\, (\mathrm{adj}(j_{3}j_{4}))_{*}(\,\tilde{\e}\,)\,\xx$  where $\tilde{\e}$ generates $H_{3n+1} (S^{3n+1};\z_{(p)})\tg\z_{(p)}.$
		\end{itemize}
	\end{lem}

	\begin{proof}
		\begin{itemize}
			\item [\rm (1)] Carry out  elementary calculations. Consider the $\z_{(p)}$-cohomology Leray-Serre spectral sequence $\n (E_{\bullet}^{\bullet,\bullet},d_{\bullet})\nn$ associated with the homotopy fibration $\Omega F^{(2)}\to *\to F^{(2)}$. We have a schematic figure combining the $E_{n+1}$-page with the $E_{2n+2}$-page.
			\[
			\resizebox{8cm}{7cm}{\begin{tikzpicture}[scale=1.05, >=stealth]

					\def\xu{1.7}
					\def\xv{4.25}

					\def\ya{1.15}   
					\def\yzeroa{2.00}
					\def\yt{2.75}   
					\def\yzerob{3.55}
					\def\yx{4.80}   
					\def\yz{5.55}   

					\tikzset{
						midarrow/.style={
							thick,
							solid,
							postaction={decorate},
							decoration={
								markings,
								mark=at position 0.5 with {\arrow{stealth}}
							}
						}
					}

					\draw[->, thick] (0,0) -- (5.35,0) node[right] {$E_{\bullet}^{*,0}$};
					\draw[->, thick] (0,0) -- (0,6.15) node[above] {$E_{\bullet}^{0,*\leq 3n+2}$};

					\draw (\xu,0.08) -- (\xu,-0.08);
					\node[below] at (\xu,-0.08) {$n+1$};
					\node[below] at (\xu,-0.55) {$u$};

					\draw (\xv,0.08) -- (\xv,-0.08);
					\node[below] at (\xv,-0.08) {$2n+2$};
					\node[below] at (\xv,-0.55) {$v$};

					\draw (0.08,\ya) -- (-0.08,\ya);
					\node[left] at (-0.08,\ya) {$\mathbb{Z}_{(p)}\quad \alpha \quad n$};

					\node[left] at (-0.08,\yzeroa) {$0$};

					\draw (0.08,\yt) -- (-0.08,\yt);
					\node[left] at (-0.08,\yt) {$\dfrac{\mathbb{Z}_{(p)}}{2p^{a}\mathbb{Z}_{(p)}}\quad\theta\quad 2n+1$};

					\node[left] at (-0.08,\yzerob) {$0$};

					\draw (0.08,\yx) -- (-0.08,\yx);
					\node[left] at (-0.08,\yx) {$\dfrac{\mathbb{Z}_{(p)}}{2p^{a}\mathbb{Z}_{(p)}}\quad \xi\;\;\;3n+1$};

					\draw (0.08,\yz) -- (-0.08,\yz);
					\node[left] at (-0.08,\yz) {$3n+1$};
					\node[left] at (-1.45,\yz) {$\mathbb{Z}_{(p)}\quad\zeta$};

					\fill (0,\ya) circle (2.4pt);
					\fill (0,\yt) circle (2.4pt);
					\fill (0,\yx) circle (2.4pt);
					\fill (0,\yz) circle (2.4pt);

					\fill (\xu,0) circle (2.4pt);
					\fill (\xv,0) circle (2.4pt);

					\fill (\xu,\ya) circle (2.4pt);
					\node[above right] at (\xu,\ya) {$2n+1$};
					\node[above left=2pt] at (\xu,\ya) {$u\af$};

					\fill (\xv,\ya) circle (2.4pt);
					\node[above right] at (\xv,\ya) {$3n+2$};
					\node[below right=2pt] at (\xv,\ya) {$v\alpha$};

					\fill (\xu,\yt) circle (2.4pt);
					\node[above right] at (\xu,\yt) {$_{3n+2}$};
					\node[below right=2pt] at (\xu,\yt) {$u\theta$};

					\fill (\xv,\yt) circle (2.4pt);
					\node[above right] at (\xv,\yt) {$4n+3$};
					\node[below right=2pt] at (\xv,\yt) {$v\theta$};

					\fill (\xu,\yx) circle (2.4pt);
					\node[above right=2pt] at (\xu,\yx-0.22) {$4n+3$};

					\fill (\xv,\yx) circle (2.4pt);
					\fill (\xu,\yz) circle (2.4pt);
					\fill (\xv,\yz) circle (2.4pt);


					\draw[midarrow] (0,\ya) -- (\xu,0);

					\draw[midarrow]
					(0,\yt) .. controls (0.40,2.05) and (3.55,2.20) .. (\xv,0);

					\draw[midarrow] (0,\yx) -- (\xu,\yt);

					\draw[midarrow]
					(0,\yz) .. controls (0.75,5.25) and (3.20,2.20) .. (\xv,\ya);

					\draw[midarrow]
					(\xu,\ya) .. controls (2.35,1.00) and (3.35,0.38) .. (\xv,0)
					node[pos=0.57, above] {$^{\times\,2p^{a}}$};

				\end{tikzpicture}

			}
			\]

			\noindent  In this figure,  the differentials belonging to the two pages can be distinguished by their lengths. The elements $\ck,\xi$ are different generators in the same degree, so we represent these two generators by separate points. The  diagram  is explained as follows.

			\indent  \quad  We know that $H^{*}(F^{(2)};\mathbb{Z}_{(p)})$ is torsion free. So, $E_{2}^{*,*}\tg H^{*}(F^{(2)};\mathbb{Z}_{(p)})\otimes H^{*}(\lo F^{(2)};\mathbb{Z}_{(p)})$ as  bigraded algebras.  Write $H^{*}(F^{(2)};\mathbb{Z}_{(p)})
			=\mathbb{Z}_{(p)}\{1,u,v\},\;(|u|=n+1,\; |v|=2n+2).$ Since the attaching map $[\e_{n+1},\,p^{a}\e_{n+1}]$ has Hopf invariant $2p^{a}$,   we may set $u^{2}=(2p^{a})v.$ Let $\alpha\in E_{n+1}^{0,n}$ be a transgressive element with $d_{n+1}(\alpha)=u.$  Then, $d_{n+1}(\alpha u)=u^{2}=(2p^{a})v.$ Thus, in the $E_{n+2}$-page, the class represented by $v$ generates $\mathbb{Z}_{(p)}/(2p^{a}\z_{(p)})$. Note that
			\begin{equation}\label{0507a}
				E_{2}^{0,i}=0,\; \forall\, i\in (n,2n+1);
			\end{equation}

			here, the case $i=2n$ requires a nontrivial argument, using the triviality of $$d\hc d\colon E_{n+1}^{0,2n}\stackrel{d}\rightarrow E_{n+1}^{n+1,n}=\z_{(p)}\stackrel{d}\rightarrow E_{n+1}^{2n+2,0}=\z_{(p)}.$$ Then there exists
			$\theta\in E_{2}^{0,2n+1}\tg E_{2n+2}^{0,2n+1}$ such that
			$$d_{2n+2}(\theta)= v\in E_{2n+2}^{2n+2,0}\;\,\text{and}\;\,E_{2}^{0,2n+1}=\x\,\theta\,\xx\cong\mathbb{Z}_{(p)}/(2p^{a}\z_{(p)}).$$

			\noindent
			Equation (\ref{0507a}) and the structure of $E_{2}^{*,0}$ imply that $E_{2}^{0,i}=0$ for each $2n+1<i<3n+1$. It follows that there exist generators $\xi,\ck\in E_{2}^{0,3n+1}\tg E_{n+1}^{0,3n+1}$ such that $d_{n+1}(\xi)=u\theta\in E_{n+1}^{n+1,2n+1}$ and $d_{2n+2}(\ck)=v\af\in E_{2n+2}^{2n+2,n}$. Moreover,  $$E_{2}^{0,3n+1}=\x\,\xi\,\xx\jia\x\,\ck\,\xx\tg(\z_{(p)}/(2p^{a}\z_{(p)}))\jia\z_{(p)}.$$
			\noindent Neither
			$E_{2}^{n+1,*}$ nor $E_{2}^{2n+2,*}$ has a nonzero element in total degree $3n+3$. So,  $E_{2}^{0,3n+2}=0.$ Then, \[
			H^{3n+1}(\Omega F^{(2)};\mathbb{Z}_{(p)})
			\cong
			\mathbb{Z}_{(p)}\oplus(\z_{(p)}/(2p^{a}\z_{(p)})),
			\quad
			H^{3n+2}(\Omega F^{(2)};\mathbb{Z}_{(p)})=0.
			\]

			By the universal coefficient theorem, we have $H_{3n+1}(\Omega F^{(2)};\mathbb{Z}_{(p)})\cong \mathbb{Z}_{(p)}.$
			\item [\rm (2)] The statement $\x\,\mathrm{adj}(j_{3}j_{4})\,\xx\tg\z_{(p)}$ follows from Lemma \ref{asd1} and the CW decomposition  of the homotopy fiber $G.$ Recall that  the map $\mathrm{adj} (j_{3}j_{4})$ sends $\e$ to $0\neq[[x,y],x]\in H_{3n+1}(\lo\fe)$ up to a unit in $\z/p$; (see Lemma \ref{gyf}). The homomorphism $H_{3n+1}(-;\z_{(p)})\rightarrow H_{3n+1}(-;\z/p)$ induced by the projection $\z_{(p)}\to\z/p$ is natural.  So, $(\mathrm{adj}(j_{3}j_{4}))_{*}(\,\tilde{\e}\,)\in H_{3n+1} (\lo \fe;\z_{(p)})\tg\z_{(p)}$ cannot be divisible by $p$. Hence the results hold.
		\end{itemize}
	\end{proof}

	\begin{lem}\label{tdxwsy}   The following assertions hold.
		\begin{itemize}
			\item [\rm (1)] At the $\mathbb{Q}$-homology level for the maps $S^{3n+k+1}\rightarrow \lo\fe$,  each element lying in  $(\lo j_{2.5})\hc \pi_{3n+k+1} (\lo S^{n+1})$  induces a zero homomorphism between the $(3n+k+1)$-th homology groups, the homomorphism $H_{3n+k+1}(\mathrm{adj} (j_{3}j_{4});\mathbb{Q})$ is a  monomorphism whose domain and codomain are nontrivial $\mathbb{Q}$-modules, and for each integer $i\in[1,\mathrm{d}]$,
			the  homomorphism $H_{3n+k+1}(\mathrm{adj} ([\s\mathbbm{y}^{(i)}]);\mathbb{Q})$ is trivial.
			\item [\rm (2)]At the$\m p$-homology level,  each element lying in the group $(\lo j_{2.5})\hc \pi_{3n+k+1} (\lo S^{n+1})$ sends $\e$ to 0,  $\mathrm{adj} (j_{3}j_{4})$ sends $\e$ to $[[x,y],x]$ up to a unit in $\z/p,$ and 
			each   $\mathrm{adj} ([\s\mathbbm{y}^{(i)}])$ sends $\e$ to 0.
		\end{itemize}
	\end{lem}
	\begin{proof}  Before proving the two assertions separately, we first make some general preliminary observations. 
		Lemma \ref{gyts} shows that  $H_{j} (\lo \fe;\z_{(p)})$ is a  free $\z_{(p)}$-module for each $j$ in the case 
		$k\geq1$ or $n$ is even. Lemma \ref{gptyl} shows that $$H_{3n+k+1} (\lo \fe;\z_{(p)})=H_{3n+1} (\lo \fe;\z_{(p)})\tg\z_{(p)}$$ \noindent in the case 
		$k=0$ and $n$ is odd. 
		
		Consider the case  $k=0$ and $n$ is odd.   For each $g\in\pi_{3n+1}(\lo\fe)$, we have a commutative diagram  
		\[\begin{tikzcd}
			H_{3n+1} (S^{3n+1};\z_{(p)}) \arrow[r,"g_{*}"] \arrow[d] & H_{3n+1} (\lo \fe;\z_{(p)}) \arrow[d] \\
			H_{3n+1} (S^{3n+1};\z/p)  \arrow[r,"g_{*}"] & H_{3n+1} (\lo \fe;\z/p),
		\end{tikzcd}\]
		where the vertical homomorphisms in both columns are induced by the  projection
		$\mathbb{Z}_{(p)}\rightarrow \mathbb{Z}/p$. These facts imply that, for each $g\in\pi_{3n+1}(\lo\fe)$, we have $  H_{3n+1}(g;\z/p)\neq0\Longrightarrow H_{3n+1}(g;\mathbb{Q})\neq0$; that is, $ H_{3n+1}(g;\mathbb{Q})=0\Longrightarrow H_{3n+1}(g;\z/p)=0.$ 
		
		Thus, in both cases, we have that for each
		$g\in\pi_{3n+k+1}(\lo\fe)$, 
		\begin{equation}\label{0507bbb}
			H_{3n+k+1}(g;\mathbb{Q})=0\Longrightarrow H_{3n+k+1}(g;\z/p)=0.
		\end{equation}

		\begin{itemize}
			\item [\rm (1)]   It was proved that the map $\mathrm{adj} (j_{3}j_{4})$ sends $\e$ to $[[x,y],x]$  up to a  unit in $\z/p$ in Lemma \ref{gyf}. Thus, the homomorphism $H_{3n+k+1}(\mathrm{adj} (j_{3}j_{4});\mathbb{Q})\neq0$. We examine the  behavior of the remaining maps at the $\mathbb{Q}$-homology level.
			Let $\mathcal{L} (-)$ be the rationalization functor,  sending $p$-local nilpotent CW complexes to their rationalizations, and sending  $p$-local maps to their rationalizations. Then $\mathcal{L} (\pi_{3n+k+1} (\lo S^{n+1}))=0$ and hence $\mathcal{L} ((\lo j_{2.5})\hc\pi_{3n+k+1} (\lo S^{n+1}))$ is 0. The elements $\mathbbm{y}^{(i)}$ lie in stable homotopy groups, so $\mathcal{L} (\mathbbm{y}^{(i)})=0$ for every $i$.  Thus, 
			\begin{align}
				\mathcal{L} ([\s\mathbbm{y}^{(i)}]) 
				&\xyd \mathcal{L}\n j_{2.5},\;[\e_{n+1},\s f],\;\mathbbm{y}^{(i)} \nn\label{l1} \\
				&\xyd \n\mathcal{L} j_{2.5},\;\mathcal{L}[\e_{n+1},\s f],\;\mathcal{L}\mathbbm{y}^{(i)} \nn\label{l2} \\
				&=\n\mathcal{L} j_{2.5},\;\mathcal{L}[\e_{n+1},\s f],\;0 \nn  \label{you0} \\
				&= \mathcal{L} j_{2.5}\hc \mathcal{L} \pi_{3n+k+2} ( S^{n+1})\label{you0h} \\
				&= 0.\notag
			\end{align}
			
			\noindent Here, the relation ``\,$\xyd$\,'' between Equation (\ref{l1}) and Equation (\ref{l2}) follows from the equivalent definition of Toda brackets (see Toda \cite[Proposition 1.7, p.~13]{Toda}) and the fact that  $\mathcal{L}$ localizes extensions and coextensions. 
			The relation ``\,$=$\,'' between Equation (\ref{you0}) and Equation (\ref{you0h}) follows from  a property of Toda brackets; see Toda \cite[Proposition 1.2 (0), p.~11]{Toda}. (Warning. In Toda's book, the symbol ``$\ty$'' should never be read as an identity relation. It always denotes congruence. In particular, when $U$ is a coset, Toda writes $U\ty 0$ to mean that $U\ni0$.)  It follows that $\mathcal{L} (\mathrm{adj} ([\s\mathbbm{y}^{(i)}]))=\mathrm{adj} (\mathcal{L} ([\s\mathbbm{y}^{(i)}]))=0.$ Then,  all elements in $(\lo j_{2.5})\hc\pi_{3n+k+1} (\lo S^{n+1})$   and 
			all   $\mathrm{adj} ([\s\mathbbm{y}^{(i)}])$  induce trivial homomorphisms  of reduced $\mathbb{Q}$-homology groups. 
			Hence assertion (1) holds.
			\item [\rm (2)]  We know that the map $\mathrm{adj} (j_{3}j_{4})$ sends $\e$ to $[[x,y],x]$  up to a  unit in $\z/p$, (Lemma \ref{gyf}). The remaining part of assertion (2) follows from assertion (1) and relation (\ref{0507bbb}).
		\end{itemize}
	\end{proof}

	\begin{lem}\label{zxkya} 
		For some unit $r\in\z_{(p)}$, we have\begin{equation}\notag\label{kjdty}
			r\mathrm{adj}^{-1} (\kj)\ty j_{3}j_{4}\m j_{2.5}\hc \pi_{3n+k+2} (S^{n+1}),\;[\s\mathbbm{y}^{(i)}], (1\leq i\leq \mathrm{d})
		\end{equation}
		If 
		$\mathrm{d}=\mathrm{dim}_{\z/p}\big(\z/p\otimes \mathrm{Im} (q_{2*})\big)=0,$  then the corresponding part involving  $[\s\mathbbm{y}^{(i)}],\, (1\leq i\leq \mathrm{d})$  is understood to be $0$.
	\end{lem}
	\begin{proof}   We prove  the lemma by the method of undetermined coefficients. Suppose that $$\kj=\mathbbm{a}+t\mathrm{adj} (j_{3}j_{4})+\sum_{i=1}^{\mathrm{d}}\lambda_{i} \mathrm{adj}[\s\mathbbm{y}^{(i)}]$$ \noindent where $\mathbbm{a}\in (\lo j_{2.5})\hc \pi_{3n+k+1} (\lo S^{n+1})$ and $t,\lambda_{i}\in\z_{(p)}.$  Examine the$\m\,p$ homology. By   Lemma \ref{tdxwsy} (2)   we infer that $t\in\z_{(p)}$ is a unit. Take $r=1/t$. Then the result holds.\end{proof}
	
	\;\\ \indent The following lemma is an immediate consequence of  Lemma \ref{zxkya}.
	\begin{lem}\label{iddy1012} Take  a $\z_{(p)}$-submodule $I\xyd \pi_{3n+k+2} (\fe)$ to be $$I=j_{2.5}\hc \pi_{3n+k+2} (S^{n+1})+\x \;[\s\mathbbm{y}^{(i)}]\;|\; 1\leq i\leq \mathrm{d}\;\xx.$$  
		Then, all maps in $\mathrm{adj}(I)\xyd \pi_{3n+k+1}(\lo\fe)$ induce a zero homomorphism of the $\z/p$-reduced homology. \hfill\boxed{}
	\end{lem}
	
	The following two lemmas  will be used to explain the relationship between the principal term and the indeterminacy term in the main theorem.
	
	\begin{lem}\label{jsidyizh} The sum $\x j_{3}  j_{4}\hc \s^{2n+k+2}f\xx+ I\hc \s^{2n+k+2}f\xyd \pi_{3n+2k+2}(\fe)$ is a direct sum.  Moreover,  if we also have 
		
		\begin{itemize}
			\item [\rm (i)] $r\leq 4n+2k$, or
			\item [\rm (ii)]  $4n+2k <r\le 4n+3k+2$ and  the homomorphism $\s\colon \pi_{r-1}(S^{2n+k+1})\to \pi_{r}(S^{2n+k+2})$ is injective,
		\end{itemize}
		then, for any $g\in \pi_{r}(S^{3n+2k+2})$, the sum
		$$\bigl\langle j_{3}j_{4}\hc \s^{2n+k+2}f\bigr\rangle\hc g
		+
		I\hc \s^{2n+k+2}f\hc g
		\subseteq
		\pi_{r}(\fe)$$
		is a direct sum. (Note that each $\pi_{r}(S^{3n+2k+2})$ with $r\leq 4n+3k+2$ is in the stable range and hence the  above composition $g^{*}(-)=(-)\hc g=(-)\hc \s h$ gives a homomorphism.)
	\end{lem}
	
	\begin{proof}   To prove   that the first sum  appearing in  the lemma  is a direct sum,
		we need only prove that it holds when $\s^{2n+k+2}f\neq0$.   Now, suppose that $\s^{2n+k+2}f\neq0.$  The   homomorphism $\pi_{3n+2k+2}(j_{3}j_{4})$ is injective (Lemma \ref{asd1}) and hence $$j_{3}j_{4}\hc \s^{2n+k+2}f\neq0\Longleftrightarrow \s^{2n+k+2}f\neq0.$$ \noindent So, $j_{3}j_{4}\hc\s^{2n+k+2}f\neq0$.  To obtain a contradiction, assume that there exists $c\in\z_{(p)}$ such that \begin{equation}\label{jstj1}
			0\neq c(j_{3}j_{4}\circ 
			\Sigma^{2n+k+2}f)\in I\circ\Sigma^{2n+k+2}f.
		\end{equation}
		\noindent 
		Then,    for  some numbers $\lambda_{i}\in\z_{(p)}$ and some element $\mathbbm{x}\in\pi_{3n+k+2} (S^{n+1})$, 
		\begin{equation}\label{050115a}
			0\neq c(j_{3}j_{4}\circ 
			\Sigma^{2n+k+2}f)=  j_{2.5}\hc \mathbbm{x}\circ 
			\Sigma^{2n+k+2}f+\sum_{i=1}^{\mathrm{d}}\lambda_{i} [\s\mathbbm{y}^{(i)}]\circ 
			\Sigma^{2n+k+2}f. \end{equation}
		Choosing an element $[\sum_{i=1}^{\mathrm{d}}\lambda_{i} \s\mathbbm{y}^{(i)}]\in \mathrm{Coext}_{[\e_{n+1},\s f]}(\sum_{i=1}^{\mathrm{d}}\lambda_{i} \mathbbm{y}^{(i)})\xyd\n j_{2.5},\;[\e_{n+1},\s f],\;\sum_{i=1}^{\mathrm{d}}\lambda_{i} \mathbbm{y}^{(i)}\nn,$  
		we have the relation $q_{2*}([\sum_{i=1}^{\mathrm{d}}\lambda_{i} \s\mathbbm{y}^{(i)}])=-\sum_{i=1}^{\mathrm{d}}\lambda_{i} \s\mathbbm{y}^{(i)}.$  Thus, 
		\begin{equation}\label{0501b}
			[\sum_{i=1}^{\mathrm{d}}\lambda_{i} \s\mathbbm{y}^{(i)}]\ty\sum_{i=1}^{\mathrm{d}}\lambda_{i}[ \s\mathbbm{y}^{(i)}]\m \mathrm{Im}(j_{3*})=j_{2.5}\hc\pi_{3n+k+2} (S^{n+1})+\x j_{3}j_{4}\xx. 
		\end{equation}
		
		\noindent  Recall that $H_{3n+k+1}(\lo \fe;\z_{(p)})$ is a free $\z_{(p)}$-module. By the method of Lemma \ref{tdxwsy}, we infer that the rationalization $\mathcal{L}(\mathrm{adj}[\sum_{i=1}^{\mathrm{d}}\lambda_{i} \s\mathbbm{y}^{(i)}])=\mathrm{adj}(\mathcal{L}[\sum_{i=1}^{\mathrm{d}}\lambda_{i} \s\mathbbm{y}^{(i)}])$ is zero and hence  $\mathrm{adj}[\sum_{i=1}^{\mathrm{d}}\lambda_{i} \s\mathbbm{y}^{(i)}]\in\pi_{*}(\lo\fe)$ induces a trivial homomorphism  of the reduced $\mathbb{Q}$-homology.  Employing Lemma \ref{tdxwsy} (1)  to examine the $\mathbb{Q}$-homology of the adjoint of Equation (\ref{0501b}), we have $[\sum_{i=1}^{\mathrm{d}}\lambda_{i} \s\mathbbm{y}^{(i)}]\ty\sum_{i=1}^{\mathrm{d}}\lambda_{i}[ \s\mathbbm{y}^{(i)}]\m j_{2.5}\hc\pi_{3n+k+2} (S^{n+1}). $
		Then,  Equation (\ref{050115a}) becomes    \begin{equation}\label{2154a} 0\neq c(j_{3}j_{4}\hc\s^{2n+k+2}f)\ty  [\sum_{i=1}^{\mathrm{d}}\lambda_{i}\s\mathbbm{y}^{(i)}]\hc\s^{2n+k+2}f \m j_{2.5}\hc\pi_{3n+k+2} (S^{n+1})\hc\s^{2n+k+2}f.
		\end{equation}
		\noindent 
		Applying $q_{2*}$ in $\pi_{\bullet}(G)\xrightarrow[]{j_{3*}} \pi_{\bullet}(\fe)\xrightarrow[]{q_{2*}} \pi_{\bullet}(S^{2n+k+2})$ to Equation (\ref{2154a}), we deduce that 
		$$0=\sum_{i=1}^{\mathrm{d}}(\lambda_{i}\s\mathbbm{y}^{(i)}\hc\Sigma^{2n+k+2}f)=(\sum_{i=1}^{\mathrm{d}}\lambda_{i}\s\mathbbm{y}^{(i)})\hc\Sigma^{2n+k+2}f\in\pi_{3n+2k+2}(S^{2n+k+2}).$$
		\noindent Note that $\pi_{3n+2k+1}(S^{2n+k+1})$ is in the stable range. We have 
		\begin{equation}\label{wsmzms}
			0=(\sum_{i=1}^{\mathrm{d}}\lambda_{i}\mathbbm{y}^{(i)})\hc\Sigma^{2n+k+1}f.
		\end{equation}\noindent Making use of the property of Toda brackets (Toda's book, \cite[Proposition 1.4, p.~11]{Toda}), we obtain \begin{align}
			[\sum_{i=1}^{\mathrm{d}}\lambda_{i}\s\mathbbm{y}^{(i)}]\hc\Sigma^{2n+k+2}f &\in \n j_{2.5},\;[\e_{n+1},\s f],\;\sum_{i=1}^{\mathrm{d}}\lambda_{i}\mathbbm{y}^{(i)} \nn\hc\Sigma^{2n+k+2}f \notag \\
			&= j_{2.5}\hc \n [\e_{n+1},\s f],\;\sum_{i=1}^{\mathrm{d}}\lambda_{i}\mathbbm{y}^{(i)} ,\;-\Sigma^{2n+k+1}f\nn \notag \\
			&\xyd \mathrm{Im}(j_{2.5*}). \notag 
		\end{align}
		\noindent Substituting this relation into Equation (\ref{2154a}), we obtain
		\begin{equation}\label{2154b} 0\neq c(j_{3}j_{4}\hc\s^{2n+k+2}f)\ty0\m j_{2.5}\hc\pi_{3n+2k+2} (S^{n+1}).
		\end{equation} In
		$\pi_{3n+2k+2} (\fe)$, the sum $\mathrm{Im}(j_{3}j_{4})_{*}+\mathrm{Im}(j_{2.5*})$ is a direct sum (Lemma \ref{asd1}), giving a contradiction. So, there is no $c\in\z_{(p)}$ such that  Equation (\ref{jstj1}) holds. That is, the sum   $\x j_{3}  j_{4}\hc \s^{2n+k+2}f\xx+ I\hc \s^{2n+k+2}f\xyd \pi_{3n+2k+2}(\fe)$ is a direct sum.
		\\ \indent To prove   that the second sum  appearing in  the lemma  is a direct sum, we need only show the result holds in the case $(\s^{2n+k+2}f)\hc \s h\neq0$ where $\s h=g$. It is easy to see that the preceding argument applies verbatim to this general case, after replacing
		$\s^{2n+k+2}f$
		throughout by
		$(\s^{2n+k+2}f)\hc \s h$, and replacing
		$\s^{2n+k+1}f$
		throughout by
		$(\s^{2n+k+1}f)\hc  h$. The only point requiring care is the analogue of Equation~(\ref{wsmzms}), which is ensured by the stable assumption or the injectivity assumption.\end{proof}

	\subsection{The determination of the map $\beta$}
	
	Note that without changing the homotopy type of $\fs$ and $F$,  we can replace the attaching map $\beta$   by $u\beta$ for any unit $u\in\z_{(p)}$.
	
	Having completed all preliminary work, we can now prove the main result of this paper.
	
	Although the $p$-localization condition and the trivial-algebra condition are assumed throughout this section, we restate  them in the following theorem for completeness and ease of application.
	
	Recall that we assume by default $f\in\pi_{n+k}(S^{n})$ with $n\geq 2$ and $k\geq 0$, and that we   assume that the map
	$f$  is of degree $p^{a}$ with some $a\in\z_{+}$   in the case $k=0$.
	
	\begin{theorem}\label{zdlzm}Localize spaces at a prime $p$ such that either $p=2$ with no additional assumption on $n+k$, or  $p\geq5$ and  $n,k+1$  are additionally assumed to be even.    Assume that the cup product $H^{n} (C_{f})\otimes H^{n} (C_{f})\stackrel{\cup}\longrightarrow H^{2n} (C_{f})$  of the mod $p$ cohomology  is  trivial, where $C_{f}=S^{n}\cup _{ f} e^{n+k+1}$ with $n\geq2$ and $k\geq0$. Then, the homotopy fiber $F$  of  the pinch map  $ \s C_{f} \xrightarrow[]{pinch} S^{n+k+2}$ is given by \[F=\fe\cup_{\beta}  e^{3n+2k+3}\cup e^{4n+3k+4}  \cup\cdots,\;\;\big(\beta\in\pi_{3n+2k+2} (\fe)\big).\]  
		\noindent
		For its attaching map $\beta$,  the following relation holds:  $$\beta\ty j_{3}  j_{4}\hc \s^{2n+k+2}f\m j_{2.5}\hc \pi_{3n+k+2} (S^{n+1})\hc \s^{2n+k+2}f,\;[\s\mathbbm{y}^{(i)}] \hc\s^{2n+k+2}f, (1\leq i\leq \mathrm{d}).$$ 
		\noindent
		That is,   $$\beta\ty j_{3}  j_{4}\hc \s^{2n+k+2}f\m I\hc \s^{2n+k+2}f.$$ \noindent Here,   the main term $j_{3}  j_{4}\hc \s^{2n+k+2}f$   is  the following composition,   \vspace{-0.85\baselineskip}\;\\\;\[\begin{tikzcd}[column sep=0.15cm]
			S^{3n+2k+2}\xrightarrow[]{\s ^{2n+k+2}f} S^{3n+k+2}
			\arrow[r, hook,"j_{4}"]
			& \mathrm{hofib} (F^{(2)} \xrightarrow[]{pinch} S^{2n+k+2})
			\arrow[r, "j_{3}","  \text{fiber inclusion}"'] \arrow[d, equal]  & F^{(2)}\arrow[d, equal]
			\\ & (S^{n+1}\vee S^{3n+k+2})\cup e^{5n+2k+3}\cup\cdots& S^{n+1}\cup_{[\mathrm{id},\s f]}  e^{2n+k+2}.
		\end{tikzcd}\] \noindent  The precise description of the maps  $[\s\mathbbm{y}^{(i)}]$  is given in Lemma \ref{410}.
	\end{theorem}
	\begin{proof} In this proof,  the symbol ``=" stands for equality up to a unit in $\z_{(p)}$. By the left square of Diagram~(\ref{zjt}) and Lemma~\ref{gylt}, we have  $\mathrm{adj} (\beta)=\lt=\phi\hc \s^{2n+k+1}f.$  
		Making use of Lemma~\ref{zxkya} and  applying $\mathrm{adj}^{-1} (-)$, we derive  the result.\end{proof}\;\\ \indent  It is known that the attaching map $\beta$ is in fact a higher Whitehead product of the form $[-,-,-]$ (not the iterated Whitehead product $[[-,-],-]$), due to  Zhu-Jin \cite{zzj}. However, higher Whitehead products are, in general,  intricate and challenging to manage, even for those of elements of homotopy groups of spheres \cite{zzj}. A comprehensive review of the definition and properties of higher Whitehead products is given by Zhu-Jin in \cite[Section~2]{zzj}.

	We consider the homotopy fiber of the pinch map $\fs\rightarrow S^{3n+2k+3}$, that is, $J(M_{\fe},S^{3n+2k+2})$. 
	\begin{cor}\label{bcd} The homotopy fiber of the pinch map $\fs\rightarrow S^{3n+2k+3}$  is  given by $$J(M_{\fe},S^{3n+2k+2})=(\fe\cup _{\gamma} e^{4n+2k+3})\cup e^{5n+3k+4}\cup\cdots.$$ 
		The attaching map $\gamma$ satisfies  $\gamma\ty j_{3}\hc[j_{3.5},\,j_{4}]\hc\s^{3n+k+2}f\m[j_{3}j_{3.5},\;I\hc\s^{2n+k+2}f]$.
	\end{cor}
	\begin{proof} We first introduce the notation. Let $a_{1}\colon   S^{n+1}\rightarrow S^{n+1}\vee S^{3n+2k+2}$ and $a_{2}\colon   S^{3n+2k+2}\rightarrow S^{n+1}\vee S^{3n+2k+2}$ be the inclusions. By the standard technique, we may replace $\beta\colon   S^{3n+2k+2}\rightarrow\fe $  by the inclusion to the mapping cylinder, $i\colon   S^{3n+2k+2}\rightarrow M_{\beta}= M_{\fe}$; let $\bar{j}_{2.5}$ be the composition $ S^{n+1} \xrightarrow[]{j_{2.5}}\fe\hookrightarrow M_{\fe}$.  Denote the relative folding map $$ M_{\fe}\vee S^{3n+2k+2}\rightarrow J_{1} (M_{\fe},S^{3n+2k+2})=M_{\fe} $$ by $\nabla$, which is given by $\nabla(x,*)=\mathrm{id}_{M_{\fe}} (x)=x,\;\nabla(*,a)=i(a)=a.$
		
		According to Gray \cite{Gray}, we have a commutative diagram whose  first row is a  homotopy cofibration,\[
		\begin{tikzcd}
			S^{n+1}\wedge S^{3n+2k+1} \arrow[r,"{[a_{1},a
				_{2}]}"] &S^{n+1}\vee S^{3n+2k+2} \arrow[r,"u_{1}", hook] \arrow[d, hook,"\bar{j}_{2.5}\vee \mathrm{id}"'] &S^{n+1}\times  S^{3n+2k+2}\arrow[d,  hook] \\
			& M_{\fe}\vee S^{3n+2k+2} \arrow[d, "\nabla"']\arrow[r,hook] & M_{\fe}\times S^{3n+2k+2}\arrow[d, two heads] \\
			& J_{2} (M_{\fe}, S^{3n+2k+2}) \arrow[r, hook] & J_{2} (M_{\fe},S^{3n+2k+2}).
		\end{tikzcd}
		\]
		\noindent Considering the homotopy pushout of $\nabla (\bar{j}_{2.5}\vee \mathrm{id})$ and $u_{1}$, we see that the resulting space is 
		the homotopy cofiber of $\nabla (\bar{j}_{2.5}\vee \mathrm{id})[a_{1},a_{2}]$ by employing the homotopy pushout of the map $[a_{1},a_{2}]$ and   $S^{n+1}\wedge S^{3n+2k+1} \rightarrow* $.  (Recall that, in a commutative $3\times 2$ diagram of spaces, if the two inner squares are homotopy pushouts, then the outer rectangle is also a homotopy pushout.)
		By checking the homology, we deduce that
		the resulting space is the 3-cell skeleton of the space $J_{2} (M_{\fe},S^{3n+2k+2})$ up to homotopy. 
		Then, the  attaching map $S^{4n+2k+2}\rightarrow M_{\fe}$ of the 3-cell skeleton of the space $J_{2} (M_{\fe},S^{3n+2k+2})$ can be taken as $$\nabla(\bar{j}_{2.5}\vee \mathrm{id})[a_{1},a_{2}]=[\nabla(\bar{j}_{2.5}\vee \mathrm{id})a_{1},\nabla(\bar{j}_{2.5}\vee \mathrm{id})a_{2}]=[\,\bar{j}_{2.5},\,i].$$
		Restoring the original maps,
		we see that $\gamma$ can be taken as $$\gamma=[j_{2.5},\beta]\ty[j_{2.5},j_{3}j_{4}\hc\s^{2n+k+2}f]\m[j_{2.5},I\hc\s^{2n+k+2}f].$$
		\noindent Clearly,   $j_{2.5}=j_{3}j_{3.5}$ up to a unit in $\z_{(p)}$. Hence, up to a unit in $\z_{(p)}$, the map $[j_{2.5},j_{3}j_{4}\hc\s^{2n+k+2}f]$ is equal to
		\begin{align}
			[j_{3}j_{3.5},j_{3}j_{4}\hc\s^{2n+k+2}f] &= j_{3}\hc [j_{3.5},\;j_{4}\hc\s^{2n+k+2}f] \notag \\
			&= j_{3}\hc [j_{3.5}\hc \s\mathrm{id}_{S^{n}},\,j_{4}\hc\s(\s^{2n+k+1}f)] \notag \\
			&= j_{3}\hc [j_{3.5},\,j_{4}]\hc \s(\mathrm{id}_{S^{n}}\wedge\s^{2n+k+1}f) \notag \\
			&= j_{3}\hc [j_{3.5},\,j_{4}]\hc (\pm\s^{3n+k+2}f ).\notag 
		\end{align} \noindent
		(Here, we freely use the well-known formula of Whitehead products given in \cite[Theorem 8.18, p.~484]{GW}.)
		Without changing the homotopy type,  we can choose  $\gamma$ satisfying
		$\gamma\ty j_{3}\hc[j_{3.5},\,j_{4}]\hc\s^{3n+k+2}f\m[j_{3}j_{3.5},\;I\hc\s^{2n+k+2}f].$\end{proof}

	\section{The application and a computational template}\label{seccp2}
	In this section, we give the application of the main theorem and a computational template, $\pi_{18} (\s^{3}\cpt:2)$.

	Suppose that the $p$-primary ($p\neq 3$) components of all homotopy groups in the $m$-th stem
	$\n\pi_{\ell+m} (S^{\ell}\colon  p)\nn_{\ell\geq2}$ are completely known, where  $m\leq i-n-1$. Based on Gray's results (Proposition~\ref{grdl}), once the attaching map $\beta$ of the
	3-cell skeleton of  $F$ is determined, the computation of
	$\pi_{i} (\Sigma C_{f}\colon  p)$--provided that the fourth cell of $F$ has no effect--can be carried out using well-established methods of the Toda school.
	These include analyzing composition relations among generators,
	studying connecting homomorphisms, and rewriting elements as linear combinations
	of the ``standard generators''. There is a wealth of paradigmatic examples in foundational works such as those of Oguchi~\cite{Og},
	Mimura-Toda~\cite{20s}, Oda~\cite{Oda}, Mimura~\cite{mlq}, Kachi~\cite{Kachilwq,r1516},
	and Mukai~\cite{Mukai1966,Mukai}. Once the attaching map $\beta$ has been determined, to calculate $\pi_{i} (\s C_{f}\colon  p)$ in the above range, our method requires no tools beyond  the classical works.

	For this reason, we present our calculations related to $\pi_{18}(\s^{3}\cpt)$ in a streamlined manner, elaborating only upon the crucial arguments and delicate points. All composition relations among the generators we freely use  can be found in $\hat{\mathrm{O}}$guchi's table \cite[pp.~103--105]{Og}.
	We shall also freely use the notation for generators of $\pi_{\bullet} (S^{\ell})$ given in Toda's book \cite{Toda}. Note that   common  terminology for dealing with Toda brackets is introduced in \cite[Remark~2.5.2]{hp2} and a brief introduction of Toda’s naming convention for the generators of $\pi_{\bullet} (S^{\ell}\colon  2)$ is given in \cite[Section~6]{hp2}.
	
	.

	We explain the philosophy behind our next calculations in the following remark.

	\begin{rem0}\label{jssq}  We aim to determine $\pi_{\bullet} (C_{g})=\pi_{\bullet} (Y\cup _{g}CX)$, where $g\colon X\rightarrow Y$ is a map between simply-connected spaces.  There is a cofiber sequence $$ X\xrightarrow[]{g}Y\stackrel{j_{_{Y}}}\hookrightarrow C_{g}\stackrel{q}\twoheadrightarrow \s X. $$\noindent
		The homotopy fiber of the map $q$ is Gray's relative James construction, i.e., $F_{q}=J(M_{Y},X)$. Then we have a fiber sequence $F_{q}\stackrel{i}\rightarrow C_{g}\stackrel{q}\rightarrow \s X$, resulting in an exact sequence of  homotopy groups,
		$$\pi_{\bullet+1} (\s X)\stackrel{\pa}\rightarrow \pi_{\bullet} (F_{q})\stackrel{i_{*}\;}\rightarrow\pi_{\bullet} (C_{g})\stackrel{q_{*}}\rightarrow\pi_{\bullet} (\s X)\stackrel{\pa}\rightarrow\pi_{\bullet-1} (F_{q}).$$
		
		\noindent After calculating the $\pa$-images, we have a short exact sequence, $$0\rightarrow \mathrm{Im} (i_{*})\stackrel{\xyd\;}\rightarrow\pi_{\bullet} (C_{g})\stackrel{q_{*}}\rightarrow\mathrm{Ker} (\pa)\rightarrow0, $$
		in which $\mathrm{Im} (i_{*})\tg\cok(\pa\colon\pi_{\bullet+1} (\s X)\rightarrow \pi_{\bullet} (F_{q}))$. So, the problem is reduced to calculating the $\pa$-images and solving the group extension problem arising from   generators of $\Ker(\pa)$. \begin{itemize}
			\item [\rm (1)] To calculate the $\pa$-images, we use these formulas:   $$\pa(\s \mathbbm{x})=i_{_{Y}}\hc g\hc \mathbbm{x}, \;\;\pa( \mathbbm{x}\hc\s\mathbbm{y})=\pa( \mathbbm{x})\hc \mathbbm{y},$$  $$\pa\n \mathbbm{x}_{1}, \s^{k+1}\mathbbm{x}_{2},\s^{k+1}\mathbbm{x}_{3}\nn_{k+1}\xyd (-1)^{k+1} \n \pa\mathbbm{x}_{1}, \s^{k}\mathbbm{x}_{2},\s^{k}\mathbbm{x}_{3}\nn_{k},$$
			
			\noindent
			whenever the maps can be composed and the  Toda bracket on the left is well-defined. Here, $$i_{_{Y}}\colon Y\hookrightarrow M_{Y}\hookrightarrow J(M_{Y}, X)=F_{q}$$ \noindent denotes the canonical inclusion. Additionally, $j_{_{Y}}=i\hc i_{_{Y}}$.
			\\\indent\quad During this process, we need to rewrite elements as linear combinations of Toda's generators, using the relations given in Toda's book \cite{Toda}, or given in $\mathrm{\hat{O}}$guchi's table \cite[pp.~103--105]{Og} which is more convenient.  To use these relations, here we emphasize that the commutation law applies to elements in the homotopy groups of spheres which in particular may be unstable: if $\af_{i}\in\pi_{i'}(S^{i})$, $\beta_{j}\in\pi_{j'}(S^{j})$ (for chosen positive  $i,j,i',j'$) and  $\af_{i+k}=\s^{k}\af_{i}$, $\beta_{j+k}=\s^{k}\beta_{j}$, then   the family $\n\af_{\bullet}\nn$ and the family $\n\beta_{\bullet}\nn$
			are commutative on $S^{i+j}$, namely,\begin{equation}\label{jhx}
				\af_{i+j}\beta_{\bullet}=\pm\beta_{i+j}\af_{\bullet}.
			\end{equation} Equation (\ref{jhx}) is just short for $\af_{i+j}\beta_{i'+j}=(-1)^{(i'-i)(j'-j)}\beta_{i+j}\af_{j'+i}$.   In concrete computations, the omitted subscripts in $\af_{\bullet}$ and $\beta_{\bullet}$ in Equation (\ref{jhx}) can usually be recovered from Toda’s naming convention for generators of $\pi_{*}(S^{n})$.

			\item [\rm (2)]To solve the group extension problem coming from the generators of $\Ker(\pa)$, we use  the following fact. Suppose that $\s\mathbbm{x}\in \Ker(\pa)$ is a generator     satisfying $g\hc \mathbbm{x}=0$ and $m\mathbbm{x}=0$. (In concrete calculations, such an element $\s \mathbbm{x}$ often exists.) Then,  there exists  $$\mathrm{coext}_{g} (\mathbbm{x})\in \mathrm{Coext}_{g} (\mathbbm{x})\xyd \n j_{_{Y}},\;g,\;\mathbbm{x} \nn\; $$\noindent such that $q_{*} (\mathrm{coext}_{g} (\mathbbm{x}))=-\s\mathbbm{x}$,   
			and hence $$-m (\mathrm{coext}_{g} (\mathbbm{x}))\in j_{_{Y}}\hc\n g,\;\mathbbm{x},\; [m]\nn,\;\;\big([m]=m(\mathrm{id})\big).
			$$

			\noindent It remains to analyze the  Toda bracket $\n g,\;\mathbbm{x}, \;[m]\nn$, employing the fundamental properties of Toda brackets given in Toda's book \cite[Proposition 1.2, Proposition 1.4~--1.6, pp.~10~--12]{Toda}  and \cite[Equation (3.10), p.~33]{Toda}.
		\end{itemize}
	\end{rem0}

	Conventionally, the homotopy groups
	$\n\pi_{m+q}(S^{q})\nn_{q\geq2}$ are called the homotopy groups of spheres in the $m$-th stem.

	Before computing $\pi_{18}(\s^{3}\cpt:2)$, we first record the following theorem, which summarizes the current range in which our method yields computability.

	Recall that
	the homotopy fiber $F$  of  the pinch map  $ *\not\simeq\s C_{f}=S^{n+1}\cup _{\s f} e^{n+k+2} \xrightarrow[]{pinch} S^{n+k+2}$ ($n\geq2,\;k\geq0$) is given by \[F=\fe\cup_{\beta}  e^{3n+2k+3}\cup e^{4n+3k+4}  \cup\cdots,\;\;\big(\beta\in\pi_{3n+2k+2} (\fe)\big).\]

In view of Lemma \ref{jsidyizh},  we shall refer to the following condition as Hypothesis $(r,n,k)$:
	\begin{itemize} \item [\rm (i)]  $r\leq 4n+2k$, or
		\item [\rm (ii)]  $4n+2k<r\leq 4n+3k+2$ and
		$\s\colon \pi_{r-1}(S^{2n+k+1})\to \pi_{r}(S^{2n+k+2})$ is injective.
	\end{itemize}
	\begin{thm0}\label{0504qqq}  Assume that the hypotheses of  Theorem \ref{hao1} (Theorem~\ref{zdlzm}) hold and that Hypothesis $(r,n,k)$ holds.  Then, in principle, all homotopy groups $\pi_{r}(S^{n+1}\cup _{\s f} e^{n+k+2})$ with $r\leq 4n+3k+2$ 
		are computable. More precisely, we have the following statement. 
		
		Assume that the hypotheses of  Theorem \ref{hao1} (Theorem~\ref{zdlzm}) hold and that Hypothesis $(r,n,k)$ holds.  Suppose that, for every $m\le  3n+3k+1$, the $p$-primary components of all homotopy groups of spheres in the $m$-th stem are known, with  explicit generating sets and the composition relations among the generators. Then, in principle, for all $r\leq 4n+3k+2$ satisfying Hypothesis $(r,n,k)$, the $p$-primary components of the $r$-th homotopy groups  of the  integral space $\s C_{f}=S^{n+1}\cup _{\s f}e^{n+k+2}$, namely, $\pi_{r}(S^{n+1}\cup _{\s f} e^{n+k+2}:p)$, can be determined.   \end{thm0}
	
	\begin{proof} It is based on the elementary technique for computing homotopy groups recalled in Remark~\ref{jssq}. Use the homotopy fibration $ F\rightarrow \s C_{f}\rightarrow S^{n+k+2}.$
		\noindent By Gray's results of the relative James construction given in Proposition~\ref{grdl}, our main theorem, Lemma \ref{jsidyizh} and Corollary~\ref{bcd}, we obtain the theorem.\end{proof}

	Let us determine $\pi_{18} (\s^{3}\cpt)$ localized at 2. Recall that $\cpt=S^{2}\cup _{\ca_{2}}e^{4}$.

	From now on, localize spaces at 2.
	
	First, we introduce  notation to denote the spaces and maps. \[
	\begin{tikzcd}[]
		S^{15}\arrow[r,"j_{4}"] & G \arrow[r,"j_{3}"] & F^{(2)} \arrow[r,"j_{2}"] & F^{(3)} \arrow[r,"j_{1}"] & F\arrow[r,"j_{0}"] & \s^{3}\cpt
	\end{tikzcd}
	\]
	\noindent In the above sequence,  the symbols $F$, $\fe$, $\fs$, $G$, $j_{2}$,  $j_{3}$ and $j_{4}$ are obtained from the notation of Section \ref{secproofbeta} by setting $f=\ca_{4}\in\pi_{5} (S^{4})$ and $p=2$ (see the table in Subsection \ref{hbg1} for convenience). Additionally, let $j_{0}$ be the homotopy fiber inclusion, $j_{1}$ be the inclusion, and recall that $j_{2.5}\colon   S^{5}\rightarrow\fe$  is the inclusion of the bottom cell.

	Note that $[\e_{5},\ca_{5}]=[\e_{5},\e_{5}]\hc\ca_{9}=\nu_{5}\ca_{8}^{2}$. Hence \begin{equation}\label{jtdf}
		F=J(M_{S^{5}},S^{6})=(S^{5}\cup_{\nu_{5}\ca_{8}^{2}} e^{11})\cup_{\beta} e^{17}\cup e^{23}\cup\cdots.  
	\end{equation} Additionally,
	$G=J(M_{S^{5}},S^{10})=(S^{5}\vee S^{15})\cup e^{25}\cup\cdots$\;(see Lemma~\ref{lpgfl}).

	\begin{lem0} Under the above setup,  $\beta\ty j_{3}j_{4}\ca_{15}\m j_{2.5}\nu_{5}\cg_{8}\ca_{15},\,4(j_{2.5}\ck_{5}).$ Additionally, in $\pi_{*}(S^{5}),$ we have $\nu_{5}\cg_{8}\ca_{15}=\nu_{5}\lt_{8}.$\end{lem0}
	
	\begin{proof} We point out that the relation 
		$\nu_{5}\cg_{8}\ca_{15}=\nu_{5}\lt_{8}$
		is not included  
		in $\hat{\mathrm{O}}$guchi's table \cite[pp.~103--105]{Og}, but is given in Toda's book \cite[p.~152]{Toda}.  By Theorem~\ref{hao1} and the known results of the homotopy groups of spheres, we obtain the lemma. \end{proof}

	We shall use the pinch map  fibrations \[G\stackrel{j_{3}}\longrightarrow\fe\stackrel{q_{2}}\longrightarrow S^{11},\;\;Z\stackrel{i_{3}'}\longrightarrow\fs\stackrel{q_{3}}\longrightarrow S^{17}\] as well as $F\stackrel{j_{0}}\longrightarrow\s^{3}\cpt\stackrel{q}\longrightarrow S^{7}$ to determine $\pi_{18} (\s^{3}\cpt)$, where the homotopy fiber
	$Z=J(M_{\fe},S^{16})$.
	
	As mentioned earlier, when describing homotopy group isomorphisms together with chosen generators, we list the generators so that they correspond, in order, to the displayed direct summands. 
	
	The homotopy groups $\pi_{18} (\fe)$, $\pi_{18} (\fs)$ and some information about $\pi_{17} (\fs)$ are  required.
	\begin{lem0}\label{cp2yb}\begin{itemize}
			\item [\rm(1)]The group $\pi_{18} (\fe)$ is given by \begin{align}
				\pi_{18} (\fe) &= \x j_{2.5}\nu_{5}\cg_{8}\nu_{15},j_{3}j_{4}\nu_{15}, \mathrm{coext}_{\nu_{5}\ca_{8}^{2}} (2\cg_{10})\xx \notag \\
				&\tg  \z/2\jia\z/8\jia\z/16. \notag
			\end{align}
			\item [\rm(2)] The group $\pi_{18} (\fs)$ is given by \begin{align}
				\pi_{18} (\fs) &= \x j_{2}j_{2.5}\nu_{5}\cg_{8}\nu_{15},j_{2}j_{3}j_{4}\nu_{15}, j_{2}\hc \mathrm{coext}_{\nu_{5}\ca_{8}^{2}} (2\cg_{10})\xx \notag \\
				&\tg \z/2\jia\z/4\jia\z/16. \notag
			\end{align}
			\item [\rm(3)] In $\pi_{17} (\fs)$, $j_{2}j_{2.5}\nu_{5}^{4}\neq0.$
			
		\end{itemize}
	\end{lem0}
	\begin{proof}  As explained at the beginning of this section, we shall only explain the essential parts of the proof, leaving the remaining details to the reader. The composition relations freely used below are in $\hat{\mathrm{O}}$guchi's table \cite[pp.~103--105]{Og}. For Toda bracket terminology and Toda's naming conventions for generators of $\pi_{\bullet}(S^{\ell})$, see \cite[Remark~2.5.2 and Section~6]{hp2}, respectively.\begin{itemize}
			\item [\rm(1)] Consider $\pa\colon\pi_{19} (S^{11})\rightarrow\pi_{18} (G)$. Recall that $j_{3.5}\colon S^{5}\rightarrow G$ is the inclusion of the bottom cell.
			We have   $\pa(\lt_{11})=j_{3.5}\nu_{5}\ca_{8}^{2}\hc\lt_{10}=j_{3.5}\nu_{5}\hc(\pm2\nu_{8}\cg_{11})=0$, and $\pa(\ch_{11})=j_{3.5}\nu_{5}\ca_{8}^{2}\hc\ch_{10}=j_{3.5}\nu_{5}\ca_{8}\hc\nu_{9}^{3}=0$. Then, $\pa(\pi_{19} (S^{11}))=0$ and $j_{3*}\colon\pi_{18} (G)\rightarrow\pi_{18} (\fe)$ is injective. Notice that $\nu_{5}\ca_{8}^{2}\cg_{10}$ is of order 2. Hence $$\mathrm{Ker} (\pa\colon \pi_{18} (S^{11})\rightarrow\pi_{17} (G))=\z/8\n2\cg_{11}\nn.$$
			Taking \vspace{-0.8\baselineskip} \[\mathrm{coext}_{\nu_{5}\ca_{8}^{2}} (2\cg_{10})\in\n j_{2.5},\nu_{5}\ca_{8}^{2},2\cg_{10}\nn\] we have $q_{2*}\mathrm{coext}_{\nu_{5}\ca_{8}^{2}} (2\cg_{10})=-2\cg_{11}$. The definition of $\mu_{n}$ (\cite[Lemma~6.5, p.~57]{Toda}) and the Jacobi identity (\cite[ii) of Equation (3.9), p.~33]{Toda}) imply  $$\n\ca_{9},2\cg_{10},8\e_{17}\nn\ty\mu_{9}\m \nu_{9}^{3},\ca_{9}\lt_{10},\cg_{9}\ca_{16}^{2}.$$
			\noindent Then we deduce that  for some $x_{1},x_{2},x_{3}\in\z,$
			\begin{align}
				-8(\mathrm{coext}_{\nu_{5}\ca_{8}^{2}} (2\cg_{10})) &\in j_{2.5}\hc\n\nu_{5}\ca_{8}^{2},2\cg_{10},8\e_{17}\nn \notag \\
				&\dyd j_{2.5}\nu_{5}\ca_{8}\hc\n \ca_{9},2\cg_{10},8\e_{17}\nn \notag \\
				&\ni j_{2.5}\nu_{5}\ca_{8}\hc(\mu_{9}+x_{1}\nu_{9}^{3}+x_{2}\ca_{9}\lt_{10}+x_{3}\cg_{9}\ca_{16}^{2}) \label{aa1} \\
				&= j_{2.5}\nu_{5}\ca_{8}\hc\mu_{9}.  \label{aa2} 
			\end{align}
			
			To obtain Equation~(\ref{aa2}) from Equation~(\ref{aa1}), note that $$\nu_{5}\ca_{8}\nu_{9}^{3}=0,\;\;\nu_{5}\ca_{8}\ca_{9}\lt_{10}=\nu_{5}\hc(\pm 2\nu_{8}\cg_{11})=0,$$ $$\nu_{5}\ca_{8}\cg_{9}\ca_{16}^{2}=\nu_{5} (\s\cg'\hc\ca_{15}+\lt_{8}+\ch_{8})\ca_{16}^{2}=\nu_{5}\s\cg'\hc4\nu_{15}+\nu_{5}\hc(\pm 2\nu_{8}\cg_{11})+\nu_{5}\hc\nu_{8}^{3}\ca_{17}=0+0+0=0.$$ \noindent
			The indeterminacy of $\n\nu_{5}\ca_{8}^{2},2\cg_{10},8\e_{17}\nn$ is
			$$\mathrm{Ind}\n\nu_{5}\ca_{8}^{2},2\cg_{10},8\e_{17}\nn=\nu_{5}\ca_{8}^{2}\hc\x\lt_{10},\ch_{10}\xx=0.$$
			\noindent
			It follows that $8(\mathrm{coext}_{\nu_{5}\ca_{8}^{2}} (2\cg_{10}))=j_{2.5}\nu_{5}\ca_{8}\mu_{9}$ is of order 2.
			\item [\rm(2)] By Lemma~\ref{asd1}, we know that $\pi_{18} (S^{17})\stackrel{\pa}\rightarrow\pi_{17} (Z)$ is injective. We use the identification $\pi_{18} (Z)=\pi_{18} (\fe)$.
			Notice that $$\pa(\ca_{17}^{2})=\beta\hc\ca_{16}^{2}\ty j_{3}j_{4}\ca_{15}\hc\ca_{16}^{2}=4(j_{3}j_{4}\nu_{15})\m j_{2.5}\nu_{5}\cg_{8}\ca_{15}\hc\ca_{16}^{2},\; 4(j_{2.5}\ck_{5})\hc\ca_{16}^{2},$$
			and that
			$\nu_{5}\cg_{8}\ca_{15}\hc\ca_{16}^{2}=\nu_{5}\cg_{8}\hc4\nu_{15}=0\in\pi_{18} (S^{5})$ as well as $4(j_{2.5}\ck_{5})\hc\ca_{16}^{2}=0$. So, $\pa(\ca_{17}^{2})=4(j_{3}j_{4}\nu_{15})$.

			\item [\rm(3)] Let us first examine $\pi_{17} (\fe)$. Note that 
			$$\nu_{5}\ca_{8}^{2}\hc \cg_{10}=\nu_{5}\ca_{8}\hc (\lt_{9}+\ch_{9})=\nu_{5}\ca_{8}\lt_{9}+\nu_{5}\hc\nu_{8}^{3}=\nu_{5}\ca_{8}\lt_{9}+\nu_{5}^{4}.$$
			
			Identify $\pi_{17} (S^{5}\vee S^{15})$ with $\pi_{17} (G)$ and consider $\pa\colon \pi_{18} (S^{11})\rightarrow\pi_{17} (G)$.  It follows  that $$\pa(\cg_{11})=j_{3.5}\nu_{5}\ca_{8}^{2}\hc\cg_{10}=j_{3.5} (\nu_{5}\ca_{8}\lt_{9}+\nu_{5}^{4})=(\nu_{5}\ca_{8}\lt_{9}+\nu_{5}^{4},0).$$
			Then, $ j_{2.5}\nu_{5}^{4}\in\pi_{17} (\fe) $ is nonzero.
			\noindent Now,  examine $\pi_{17} (\fs)$. Consider $\pa\colon \pi_{18} (S^{17})\rightarrow\pi_{17} (Z)$ and identify $\pi_{17} (\fe)$ with $\pi_{17} (Z)$. 
			We have $$\pa(\ca_{17})=\beta\hc\ca_{16}\ty j_{3}j_{4}\ca_{15}^{2}\m j_{2.5}\nu_{5}\lt_{8}\hc\ca_{16}=j_{2.5}\nu_{5}^{4}.$$
			\noindent  Then, the restriction $\pi_{17} (j_{2})|_{    \mathrm{Im} (\pi_{17} (j_{2.5}))}$ is injective.
	\end{itemize}\end{proof} 
	Recall that $i_{3}'\colon  Z\rightarrow\fs$ is the homotopy fiber inclusion. Let  $i_{4}'\colon  \fe\hookrightarrow Z=J(M_{\fe},S^{16})$ be the canonical inclusion. Then, up to a unit in $\z_{(2)}$, we conclude that  $j_{2}=i_{3}'i_{4}'$ and that
	$j_{2}j_{2.5}\colon   S^{5}\rightarrow \fs$ is  the inclusion of the bottom cell  by examining the $5$-th homology groups.
	With these preparations, we are now able to compute the homotopy group $\pi_{18} (\s^{3}\cpt)$.
	
	\begin{thm0}\label{tlq} The homotopy group $\pi_{18} (\s^{3}\cpt)$ is given as follows:
		\begin{align}
			\pi_{18} (\s^{3}\cpt) &= \x j_{0}j_{1}j_{2}j_{2.5}\nu_{5}\cg_{8}\nu_{15}  ,\;j_{0}j_{1}j_{2}j_{3}j_{4}\nu_{15},\; j_{0}j_{1}j_{2}\hc \mathrm{coext}_{\nu_{5}\ca_{8}^{2}} (2\cg_{10}),\;\mathrm{coext}_{\ca_{5}} (\ck_{6})\xx \notag \\
			&\tg \z/2\jia\z/4\jia\z/16\jia\z/8. \notag 
		\end{align}
	\end{thm0} 
	\begin{proof}  The homomorphism $\pi_{18} (j_{0})$ is injective since $\pi_{19} (S^{7})=0$. Consider $\mathrm{Ker} (\pi_{18} (S^{7})\stackrel{\pa}\longrightarrow\pi_{17} (F)).$  In $\pi_{17} (F)$, we have $\pa(\ck_{7})=(j_{1}j_{2}j_{2.5}\ca_{5})\hc\ck_{6}=0$ as $\ca_{5}\ck_{6}=0$.  Employing Lemma~\ref{cp2yb} (3) we obtain
		$$\pa(\ch_{7}\nu_{15})=(j_{1}j_{2}j_{2.5}\ca_{5})\hc\ch_{6}\nu_{14}=j_{1}j_{2}j_{2.5}\nu_{5}^{4}\neq0.$$ \noindent Therefore,
		$\mathrm{Ker} (\pi_{18} (S^{7})\stackrel{\pa}\longrightarrow\pi_{17} (F))=\z/8\n\ck_{7}\nn.$ Note that the composition 
		$j_{0}j_{1}j_{2}j_{2.5}=i$ up to a unit in $\z_{(2)}$, where $i\colon S^{5}\rightarrow \s^{3}\cpt$ denotes  the inclusion of the bottom cell.
		
		Taking $\mathrm{coext}_{\ca_{5}} (\ck_{6})\in\n i,\ca_{5},\ck_{6}\nn$,  we have $q_{*}\mathrm{coext}_{\ca_{5}} (\ck_{6})=-\ck_{7}$. Then we deduce that 
		\begin{align}
			-8\mathrm{coext}_{\ca_{5}} (\ck_{6}) &\in i\hc\n \ca_{5},\ck_{6},8\e_{17}\nn \notag \\
			&\xyd i\hc\n \ca_{5},4\ck_{6},2\e_{17}\nn \notag \\
			&= i\hc\n \ca_{5},\ca_{6}^{2}\mu_{8},2\e_{17}\nn \notag \\
			&\dyd i\hc\n 4\nu_{5},\mu_{8},2\e_{17}\nn \notag \\
			&\dyd i\hc2\e_{5}\hc\n 2\nu_{5},\mu_{8},2\e_{17}\nn \notag \\
			&\dyd -i\hc2\e_{5}\hc\s\n 2\nu_{4},\mu_{7},2\e_{16}\nn \notag \\
			&= 0\m0. \notag 
		\end{align}
		Here, we use   $\ca_{5}\hc(\pm  P(\cg_{13}))=\ca_{5}\hc[\e_{6},\e_{6}]\cg_{11}=[\e_{5},\e_{5}]\ca^{2}_{9}\cg_{11}=\nu_{5}\ca_{8}^{3}\cg_{11}=\nu_{5}\hc4\nu_{8}\hc\cg_{11}=0$ to compute  the indeterminacy of $i\hc\n \ca_{5},\ca_{6}^{2}\mu_{8},2\e_{17}\nn$.\end{proof}

\end{document}